\documentclass{amsart} 
\usepackage{latexsym}
\usepackage{amsmath}
\usepackage{amssymb}
\usepackage{amsfonts} 
\usepackage{eucal} 
\usepackage{cmmib57} 
\usepackage{amsbsy}
\usepackage{amsthm}
\usepackage{amscd}
\usepackage[dvips]{graphics}
\usepackage[all]{xy}
\usepackage{MnSymbol}

\newtheorem{thm}{Theorem}[section]
\newtheorem*{thm*}{Theorem}
\newtheorem{lemma}[thm]{Lemma}
\newtheorem{prop}[thm]{Proposition}
\newtheorem{cor}[thm]{Corollary}

\theoremstyle{definition}
\newtheorem{defn}[thm]{Definition}
\newtheorem{example}[thm]{Example}

\theoremstyle{remark}

\newcommand {\hypco} {\ensuremath{\mbox{$\mathcal{H}$}}}

\newcommand {\La}    {\ensuremath{\mathcal{L}}}

\newcommand {\Da}    {\ensuremath{\mbox{$\mathcal{D}$}}}
\newcommand {\Ea}    {\ensuremath{\mbox{$\mathcal{E}$}}}
\newcommand {\Fa}    {\ensuremath{\mbox{$\mathcal{F}$}}}
\newcommand {\Ga}    {\ensuremath{\mbox{$\mathcal{G}$}}}

\newcommand {\Sa}    {\ensuremath{\mbox{$\mathcal{S}$}}}

\newcommand {\Xa}    {\ensuremath{\mathcal{X}}}

\newcommand {\der}   {\ensuremath{\textbf{H}}}

\newcommand {\real}  {\ensuremath{\mathbb{R}}}
\newcommand {\intg}  {\ensuremath{\mathbb{Z}}}

\newcommand {\cplx}  {\ensuremath{\mathbb{C}}}
\newcommand {\rat}   {\ensuremath{\mathbb{Q}}}

\newcommand {\Hom}   {\ensuremath{\operatorname{Hom}}}

\newcommand {\cok}   {\operatorname{coker}}
\newcommand {\im}    {\operatorname{im}}
\newcommand {\rk}    {\operatorname{rk}}
\newcommand {\interi}{\operatorname{int}}

\newcommand {\supp}  {\ensuremath{\operatorname{supp}}}

\newcommand {\emb}   {\ensuremath{\operatorname{emb}}}

\newcommand {\ip}    {\ensuremath{I^{\bar{p}}}}
\newcommand {\iq}    {\ensuremath{I^{\bar{q}}}}
\newcommand {\imi}   {\ensuremath{I^{\bar{m}}}}

\newcommand {\redh}  {\ensuremath{\widetilde{H}}}

\newcommand {\cone}  {\ensuremath{\operatorname{cone}}}

\newcommand {\pt}    {\ensuremath{\operatorname{pt}}}
\newcommand {\id}    {\ensuremath{\operatorname{id}}}

\newcommand {\Harm}  {\ensuremath{\operatorname{Harm}}}
\newcommand {\cyl}   {\ensuremath{\operatorname{cyl}}}

\newcommand {\bc}    {\ensuremath{\partial \mathcal{C}}}
\newcommand {\rel}   {\ensuremath{\operatorname{rel}}}
\newcommand {\si}    {\ensuremath{S^\infty}}
\newcommand {\sps}    {\ensuremath{S^\propto}}

\newcommand {\ms}    {\ensuremath{\mathcal{MS}}}
\newcommand {\bms}    {\ensuremath{{\partial \mathcal{MS}}}}
\newcommand {\oms}    {\ensuremath{\Omega_{\ms}}}
\newcommand {\omsc}    {\ensuremath{\Omega_{\ms, c}}}
\newcommand {\oip}    {\ensuremath{\Omega I_{\bar{p}}}}
\newcommand {\oiq}    {\ensuremath{\Omega I_{\bar{q}}}}

\newcommand {\ft}   {\ensuremath{\operatorname{ft}}}

\newcommand {\proj}   {\ensuremath{\operatorname{proj}}}

\newcommand {\CWkcb}   {\operatorname{\mathbf{CW}}_{k\supset \partial}}
\newcommand {\CWccb}   {\operatorname{\mathbf{CW}}_{c\supset \partial}}
\newcommand {\HoCW}    {\operatorname{\mathbf{HoCW}}}

\begin{document}


\title[A De Rham Complex Describing Intersection Space Cohomology]{Foliated Stratified Spaces and a De 
   Rham Complex Describing Intersection Space Cohomology}

\author{Markus Banagl}

\address{Mathematisches Institut, Universit\"at Heidelberg,
  Im Neuenheimer Feld 288, 69120 Heidelberg, Germany}

\email{banagl@mathi.uni-heidelberg.de}

\thanks{The author was in part supported by a research grant of the
 Deutsche Forschungsgemeinschaft.}

\date{February, 2011}




\begin{abstract}
The method of intersection spaces associates cell-complexes depending on a perversity
to certain types of stratified pseudomanifolds in such a way that Poincar\'e duality
holds between the ordinary rational cohomology groups of the cell-complexes
associated to complementary perversities. The cohomology of these intersection spaces
defines a cohomology theory HI for singular spaces, which is not isomorphic to
intersection cohomology IH. Mirror symmetry tends to interchange IH and HI.
The theory IH can be tied to type IIA string theory, while HI can be tied to IIB theory.
For pseudomanifolds with stratification depth 1 and flat link bundles, the present
paper provides a de Rham-theoretic description of the theory HI by a complex of
global smooth differential forms on the top stratum. We prove that the wedge product
of forms introduces a perversity-internal cup product on HI, for every perversity.
Flat link bundles arise for example
in foliated stratified spaces and in reductive Borel-Serre compactifications of 
locally symmetric spaces. A precise topological definition of the notion of a 
stratified foliation is given.
\end{abstract}

\maketitle


\tableofcontents


\section{Introduction}

Let $\bar{p}$ be a perversity in the sense of intersection homology theory,
\cite{gm1}, \cite{gm2}, \cite{kirwanwoolf}, \cite{banagl-tiss}.
In \cite{banagl-intersectionspaces}, we introduced a general homotopy-theoretic
framework that assigns to certain types of $n$-dimensional stratified pseudomanifolds
$X$ CW-complexes
\[ \ip X, \]
the \emph{perversity-$\bar{p}$ intersection spaces} of $X$, such that for
complementary perversities $\bar{p}$ and $\bar{q}$, there is a Poincar\'e duality
isomorphism
\[ \redh^i (\ip X;\rat) \cong \redh_{n-i} (\iq X;\rat) \]
when $X$ is compact and oriented. In particular, this framework yields a new
cohomology theory $HI^\bullet_{\bar{p},s} (X) = H^\bullet_s (\ip X)$ for singular spaces,
where $H^\bullet_s$ denotes ordinary singular cohomology. For the lower middle
perversity $\bar{p} =\bar{m}$, we shall briefly write $IX = \imi X$ and
$HI^\bullet_s (X) = HI^\bullet_{\bar{m},s} (X)$.
That this theory is indeed
not isomorphic to intersection cohomology $IH^\bullet_{\bar{p}} (X)$ or to Cheeger's
$L^2$-cohomology $H^\bullet_{(2)} (X)$ is apparent from the observation that,
for every $\bar{p}$, $HI^\bullet_{\bar{p},s} (X)$ is an algebra under cup product,
whereas it is well-known that $IH^\bullet_{\bar{p}} (X)$ and $H^\bullet_{(2)} (X)$
cannot generally be endowed with a $\bar{p}$-internal algebra structure compatible
with the cup product. \\

The present paper serves a twofold purpose: It provides a de Rham-type description of
$HI^\bullet_{\bar{p},s} (X;\real)$ in terms of certain global differential forms on the
top stratum of $X$. But by doing so, it simultaneously opens up a way of defining the
theory $HI^\bullet_{\bar{p}} (X)$ on spaces $X$, for which the intersection space
$\ip X$ has not been constructed yet. 
The construction of intersection spaces is reviewed in Section \ref{ssec.backinterspace}. 
That section also lists the space classes 
for which $\ip X$ has been presently constructed and
Poincar\'e duality established. In these constructions, the singularity links are generally
assumed to be simply connected.
Let $X^n$ be a compact, oriented, stratified pseudomanifold of stratification depth $1$ possessing
Mather control data (see Definitions \ref{def.twostrataspace}, \ref{def.stratdepthone} for details),
in particular a link bundle for every component of the singular set $\Sigma$.
Assume that all of these link bundles are flat and that each link can be endowed with
a Riemannian metric such that the structure group of the bundle is contained in the
isometries of the link. (Such a metric can always be found if the structure group is a compact
Lie group.) Do \emph{not} assume that the links are simply connected --- they may or may not be.
For such $X$, we define a subcomplex $\oip^\bullet (X-\Sigma)$ of the complex
$\Omega^\bullet (X-\Sigma)$ of smooth differential forms on the top stratum $X-\Sigma$,
set
\[ HI^\bullet_{\bar{p}} (X) = H^\bullet (\oip^\bullet (X-\Sigma)), \]
and show \\

\noindent \textbf{Theorem \ref{thm.10.6}.}
\textit{(Generalized Poincar\'e Duality.) Let $\bar{p}$ and $\bar{q}$ be complementary perversities.
Wedge product followed by integration induces
a nondegenerate bilinear form 
\[ \begin{array}{rcl}
\int: HI^r_{\bar{p}} (X) \times
 HI^{n-r}_{\bar{q}} (X) & \longrightarrow & \real, \\
([\omega], [\eta]) & \mapsto & \int_{X-\Sigma} \omega \wedge \eta. 
\end{array} \]
}

\noindent For $HI^\bullet_{\bar{p},s} (X;\rat)$, the proofs of the duality 
Theorems 2.12 and 2.47 in \cite{banagl-intersectionspaces}
require choosing certain splittings. Thus the above Theorem \ref{thm.10.6} demonstrates in
particular that the intersection product on $HI^\bullet_{\bar{p}}$ is canonically defined
independent of choices. We prove our de Rham theorem for spaces with only isolated 
singularities. \\

\noindent \textbf{Theorem \ref{thm.derham}.}
\textit{(De Rham description of $HI^\bullet_{\bar{p},s}$.)
Let $X$ be a compact, oriented pseudomanifold with only isolated singularities and
simply connected links. Then integrating a form in $\oip^\bullet (X-\Sigma)$ over a smooth
singular simplex in $X-\Sigma$ induces an isomorphism
\[ HI^\bullet_{\bar{p}} (X) \cong \redh I^\bullet_{\bar{p},s} (X;\real). \]
}

\noindent Again, we will briefly put $HI^\bullet (X) = HI^\bullet_{\bar{m}} (X)$.
An important advantage of the differential form approach adopted in this paper
is that it eliminates the simple connectivity assumption on links. This assumption is generally
needed in forming the intersection space, since the homotopy-theoretic method uses the
Hurewicz theorem. As there is presently no general construction of $\ip X$ available for
$X$ with flat link bundles, this paper extends the theory $HI^\bullet_{\bar{p}}$ to such 
spaces. Let us indicate some fields of application. 
If the link bundle is flat, then the total space of the bundle possesses a foliation so that
the bundle becomes a transversely foliated fiber bundle. Conversely,
flat link bundles arise in foliated stratified spaces.
A precise definition of stratified foliations is given in Section \ref{sec.folstrat}
(Definitions \ref{def.stratfol}, \ref{def.stratfoldepth1}), at least for stratification depth $1$.
Such foliations play a role for instance in the work of Farrell and Jones on the
topological rigidity of negatively curved manifolds, \cite{farrelljonesfolcont1}, 
\cite{farrelljonespnas}. Our definition of a stratified foliation is inspired by the conical
foliations of Saralegi-Aranguren and Wolak, \cite{wolaksaralbicabgrpisom}.
The orbits of an isometric Lie group action on a compact Riemannian manifold,
for example, form a conical foliation. Theorem \ref{thm.strfolflatlinkbundle} of the present paper
confirms that if a stratified foliation is zero-dimensional on the links, then the restrictions
of the link bundle to the leaves of the singular stratum are flat bundles.

Reductive Borel-Serre compactifications of locally symmetric spaces constitute
another field of stratified spaces to which the theory $HI^\bullet$ can be applied.
Let $G$ be a connected reductive algebraic group defined over $\rat$ and
$\Gamma \subset G(\rat)$ an arithmetic subgroup. Let $K\subset G(\real)$
be a maximal compact subgroup and $A_G$ the connected component of the real
points of the maximal $\rat$-split torus in the center of $G$. The associated
symmetric space is $D= G(\real)/KA_G$. The arithmetic quotient
$X=\Gamma \backslash D$ is generally not compact and several compactifications
of $X$ have been studied. For simplicity, let us assume that $\Gamma$ is neat,
so that $X$ is a manifold. (Otherwise, $X$ may have mild singularities; it is in general
a V-manifold. Any arithmetic group contains a neat subgroup of finite index.)
The Borel-Serre compactification $\overline{X}$ (\cite{borelserrecomp}) is a manifold with corners
whose interior is $X$ and whose faces $Y_P$ are indexed by the $\Gamma$-conjugacy
classes of parabolic $\rat$-subgroups $P$ of $G$. Each $Y_P$ fits into a \emph{flat}
bundle $Y_P \to X_P$, called the nilmanifold fibration because the fiber is a compact
nilmanifold. The $X_P$ are arithmetic quotients of the symmetric space associated
to the Levi quotient of $P$. The reductive Borel-Serre compactification
$\widehat{X}$, introduced by Zucker (\cite{zucker}), is the quotient of $\overline{X}$
obtained by collapsing the fibers of the nilmanifold fibrations. The $X_P$ are the strata
of $\widehat{X}$ and their link bundles are the flat nilmanifold fibrations. A basic
class of examples is given by Hilbert modular surfaces $X$ associated to real
quadratic fields $\rat(\sqrt{d})$. For these, the $X_P$ are circles, the nilmanifold links
are $2$-tori and the flat link bundles are mapping tori, see \cite{bk1}. \\

Let us describe some of the features of $HI^\bullet_{\bar{p}}$. Since there is no
general cup product $H^i (M)\otimes H^j (M)\to H^{i+j}(M,\partial M)$ for a manifold
$M$ with boundary $\partial M$, intersection cohomology $IH^\bullet_{\bar{p}} (X)$,
for most $\bar{p}$,
cannot be endowed with a $\bar{p}$-internal cup product.
Similarly, the complex $\Omega^\bullet_{(2)} (X-\Sigma)$ of $L^2$-forms on the
top stratum equipped with a conical metric in the sense of Cheeger
(\cite{cheeger2}, \cite{cheeger1}, \cite{cheeger3}) is not a 
differential graded algebra (DGA) under wedge product of forms --- the product of
two $L^2$-functions need not be $L^2$ anymore. We prove that for every
perversity $\bar{p}$, the DGA-structure $(\Omega^\bullet (X-\Sigma),d,\wedge),$
where $d$ denotes exterior derivation,
restricts to a DGA-structure $(\oip^\bullet (X-\Sigma),d,\wedge)$
(Theorem \ref{thm.dga}). Consequently, the wedge product induces a cup product
\[ \cup: HI^i_{\bar{p}} (X)\otimes HI^j_{\bar{p}}(X)\longrightarrow HI^{i+j}_{\bar{p}} (X). \]
This is of course consistent with our de Rham theorem and our earlier (trivial)
observation that $HI^\bullet_{\bar{p},s} (X)$ possesses a cup product. \\

Contrary to $IH^\bullet_{\bar{p}}$ and $H^\bullet_{(2)}$, the theory $HI^\bullet_{\bar{p}}$
is quite stable under deformation of complex algebraic singularities. Consider for example
the Calabi-Yau quintic
\[ V_s = \{ z\in \cplx P^4 ~|~ z^5_0 +  z^5_1 +  z^5_2 +  z^5_3 +  z^5_4 
 - 5(1+s)z_0 z_1 z_2 z_3 z_4 =0 \}, \]
depending on a complex parameter $s$. The variety $V_s$ is smooth for small $s\not= 0,$
while $V=V_0$ has 125 isolated singular points. Its ordinary cohomology
has Betti numbers $\rk H^2 (V)=1,$ $\rk H^3 (V)=103,$ $\rk H^4 (V)=25$ and its
middle perversity intersection cohomology has ranks
 $\rk IH^2 (V)=25,$ $\rk IH^3 (V)=2,$ $\rk IH^4 (V)=25$. Both of these sets of
Betti numbers differ considerably from the Betti numbers of the nearby smooth
deformation $V_s$ ($s\not= 0$):
$\rk H^2 (V_s)=1,$ $\rk H^3 (V_s)=204,$ $\rk H^4 (V_s)=1$.
Now the calculations of \cite[Section 3.9]{banagl-intersectionspaces}, together
with our de Rham theorem, show that
\[ \rk HI^2 (V)=1, \rk HI^3 (V)=204, \rk HI^4 (V)=1, \]
in perfect agreement with the Betti numbers of $V_s, s\not= 0$. Indeed, jointly with
L. Maxim, we have established the following Stability Theorem, see
\cite{banaglmaximdefostab}: Let $V$ be a complex $n$-dimensional projective
hypersurface with one isolated singularity and let $V_s$ be a nearby smooth
deformation of $V$. Then for all $i<2n,$ and $i\not= n,$
$\redh I^i_s (V;\rat)\cong \redh^i (V_s;\rat).$ For the middle dimension
$HI^n_s (V;\rat)\cong H^n (V_s;\rat)$ if, and only if, the monodromy operator
acting on the cohomology of the Milnor fiber of the singularity is trivial.
At least if $H_{n-1}(L;\intg)$ is torsionfree, where $L$ is the link of the singularity,
the isomorphism is induced by a continuous map $IV\to V_s$ and is thus a 
ring isomorphism. We use this in \cite{banaglmaximdefostab} to endow
$HI^\bullet_s (V;\rat)$ with a mixed Hodge structure so that the canonical map
$IV\to V$ induces homomorphisms of mixed Hodge structures in cohomology.
Even if the monodromy is not trivial, $IV\to V_s$ induces a monomorphism on
homology. This statement for $HI^\bullet$ may be viewed as a ``mirror image''
of the well-known fact that the intersection homology of a complex variety $V$
is a linear subspace of the ordinary homology of any resolution
$\widetilde{V}\to V$, as follows from the Beilinson-Bernstein-Deligne-Gabber
decomposition theorem, \cite{bbd}. If the resolution is small, then
$IH^i (V)\cong H^i (\widetilde{V}).$ Thus the monodromy condition for
deformations may be viewed as a ``mirror image'' of the smallness condition
for resolutions. \\

The relationship between $IH^\bullet$ and $HI^\bullet$ is indeed illuminated well
by mirror symmetry, which tends to exchange resolutions and deformations.
In \cite{morrison} for example, it is conjectured that the mirror of a conifold
transition, which consists of a degeneration $s \to 0$ 
followed by a small resolution, is again a conifold transition, but
performed in the reverse direction. The results of Section 3.8 in
\cite{banagl-intersectionspaces} together with the de Rham theorem of this paper
imply that if $V^\circ$ is the mirror
of a conifold $V$, both sitting in  mirror symmetric conifold transitions, then
\[ \begin{array}{lcl} 
\rk IH^3 (V) & = & \rk HI^2 (V^\circ) + \rk HI^4 (V^\circ)+2, \\
\rk IH^3 (V^\circ) & = & \rk HI^2 (V) + \rk HI^4 (V)+2, \\
\rk HI^3 (V) & = & \rk IH^2 (V^\circ) + \rk IH^4 (V^\circ)+2,
 \text{ and} \\ 
\rk HI^3 (V^\circ) & = & \rk IH^2 (V) + \rk IH^4 (V)+2.
\end{array} \]
Since mirror symmetry is a phenomenon that arose originally in string theory,
it is not surprising that the theories $IH^\bullet,$ $HI^\bullet$ have a specific
relevance for type IIA, IIB string theories, respectively.
While $IH^\bullet$ yields the correct count of massless $2$-branes on a conifold
in type IIA theory, the theory $HI^\bullet$ yields the correct count of massless
$3$-branes on a conifold in type IIB theory, see \cite{banagl-intersectionspaces}. 
The author hopes that the de Rham description of $HI^\bullet$ by differential forms
offered here is closer to physicists' intuition of cohomology than the homotopy
theory of \cite{banagl-intersectionspaces}. The present paper makes it possible,
for example, to obtain differential form representatives for the above mentioned
massless $3$-branes in IIB string theory. \\

A few words about the technical aspects of the paper: Overall, our approach is
topological, as we do not use a Riemannian metric on the top stratum.
We do not even require a metric on the link bundle, only a fixed metric on a
particular copy $L$ of the link. To obtain a de Rham description of intersection
cohomology, one uses a truncation $\tau_{<k}\Omega^\bullet (L)$ of the forms
on the link, as is well-known. To pass from this local normal truncation to a global
complex, one must perform fiberwise normal truncation. This is technically
easy to accomplish, since an automorphism of $L$ induces an automorphism
of $\Omega^\bullet (L)$, which restricts to an automorphism of
$\tau_{<k} \Omega^\bullet (L)$. Ultimately, the result will indeed be a subcomplex of
$\Omega^\bullet (X-\Sigma)$, since there is a canonical monomorphism
$\tau_{<k} \Omega^\bullet (L) \to \Omega^\bullet (L)$. By contrast, a de Rham
model for $HI^\bullet_s$ requires the use of \emph{cotruncation}
$\tau_{\geq k} \Omega^\bullet (L)$. If one uses standard cotruncation of a
complex, one runs into two problems: standard cotruncation comes with a
canonical epimorphism $\Omega^\bullet (L)\to \tau_{\geq k} \Omega^\bullet (L)$,
so one will not obtain a \emph{sub}complex of $\Omega^\bullet (X-\Sigma)$.
Furthermore, one must implement normal cotruncation as a subcomplex in such
a way that it can be carried out in a fiberwise fashion. This paper solves these
problems as follows: In Section \ref{sec.truncoverpoint}, we use Riemannian Hodge theory to define
cotruncation as a subcomplex $\tau_{\geq k} \Omega^\bullet (L)\subset 
\Omega^\bullet (L)$ (Definition \ref{def.cotruncoverpoint}). This is the reason for
requiring a metric on $L$. By Proposition \ref{prop.indepriemmetric},
$\tau_{\geq k} \Omega^\bullet (L)$ is independent (up to isomorphism)
of the metric on $L$. An isometry $L\to L$ induces an automorphism of
$\tau_{\geq k} \Omega^\bullet (L)$, a property that is important for fiberwise
cotruncation and explains why we assume the structure group of the link bundle
to lie in the isometries of $L$. In order to implement fiberwise cotruncation,
we develop a model, called the \emph{multiplicatively structured forms}, for the
forms on the total space of the link bundle, which is structured enough so that
fiberwise cotruncation is fairly straightforward, but at the same time rich enough
so that it computes the ordinary cohomology of the link bundle
(Theorem \ref{thm.fhcomputescohtotspace}). The multiplicative structuring
of forms uses the flatness assumption on the bundle in an essential way. 
These techniques then allow us to construct the subcomplex
$\oip^\bullet (X-\Sigma) \subset \Omega^\bullet (X-\Sigma)$ on
page \pageref{page.oipdef}. Additional tools are required in proving the
de Rham theorem, since the intersection space $\ip X$ is not smooth,
but only a CW-complex. In Section \ref{ssec.partialsmooth}, we introduce a 
\emph{partial smoothing} tool that enables us to recover enough
smoothness of singular simplices $\Delta \to \ip X$ so that forms in
$\oip^\bullet (X-\Sigma)$ can be integrated over them and this induces
an isomorphism. \\

The methods introduced in the present paper radiate out into fields that
are not (directly) linked to singularities. For example, let $\pi: E\to B$
be a flat fiber bundle of closed, smooth manifolds with oriented fiber and
compact Lie structure group.
Then the above method of fiberwise
cotruncation and multiplicatively structured forms
can be used to show that the cohomological Leray-Serre
spectral sequence of $\pi$ for real coefficients collapses at the $E_2$-term.
We can furthermore show that if $M$ is an oriented,
closed, Riemannian manifold and $G$ a discrete group, whose
Eilenberg-MacLane space $K(G,1)$ may be taken to be a closed, smooth
manifold (e.g. $G=\intg^n$), and which acts isometrically on $M$, then the
equivariant cohomology $H^\bullet_G (M;\real)$ of this action can be computed
as
\[ H^k_G (M;\real)\cong \bigoplus_{p+q=k} H^p (G;\der^q (M;\real)), \]
where the $\der^q (M;\real)$ are the cohomology $G$-modules determined by the action.
(We do not assume that $G$ is closed in the isometry group of $M$.)
These consequences will be detailed elsewhere. In a similar vein, the
fiberwise spatial homology truncation methods used to construct intersection spaces
yield, for simply connected singular sets where nontrivial link bundles are not flat,
information on cases of the Halperin conjecture, \cite{halperin78}, \cite{fht}. \\

An analytic description of the cohomology theory $HI^\bullet$ remains to be found. 
A partial result in this direction is the following.
Let $M$ be a smooth, compact manifold with boundary $\partial M$. Let $x$ be
a boundary-defining function, i.e. on $\partial M$ we have $x\equiv 0,$ and
$dx\not= 0$. A Riemannian metric $g$ on the interior $N$ of $M$ is called a
\emph{scattering metric} if near $\partial M$ it has the form
\[ g= \frac{dx^2}{x^4} + \frac{h}{x^2}, \]
where $h$ is a metric on $\partial M$. Let $L^2 \hypco^\bullet (N,g)$ denote
the Hodge cohomology space of $L^2$-harmonic forms on $N$.
From Melrose \cite{melroseasympteuclid}, the work of Hausel, Hunsicker and Mazzeo
\cite{hauselhunsmazzeo}, and the results of \cite{banagl-intersectionspaces}, one can
readily derive:
\begin{prop}
Suppose that $X^n$ is an even-dimensional pseudomanifold
with only one isolated singularity so that $X = M\cup \cone (\partial M)$, where $M$
is a compact manifold with boundary. If the complement $N$ of the singular point is endowed with a
scattering metric $g$ and the restriction map $H^{n/2}(M)\to H^{n/2}(\partial M)$ is zero
(a ``Witt-type'' condition), then
\[ HI^\bullet (X) \cong L^2 \hypco^\bullet (N,g). \]
\end{prop}

\textbf{Acknowledgments.} Parts of this paper were written during the author's
visits to Kagoshima University, the Courant Institute of New York University, the
University of Wisconsin at Madison, and Loughborough University. I would like to
thank these institutions for their hospitality and my respective hosts
Shoji Yokura, Sylvain Cappell, Laurentiu Maxim and Eugenie Hunsicker for many
stimulating discussions. \\

\textbf{General Notation.} For a real vector space $V$, we denote the linear
dual $\Hom (V,\real)$ by $V^\dagger$. The tangent space of a smooth manifold $M$ at a
point $x\in M$ is written as $T_x M$. For a smooth manifold $M$, $H^\bullet (M)$ will
always denote the de Rham cohomology of $M$, whereas $H^\bullet_s (X)$ denotes
the singular cohomology with real coefficients of a topological space $X$. Singular
homology with real coefficients will be written as $H_\bullet (X)$. Reduced cohomology
and homology are indicated by $\redh^\bullet, \redh^\bullet_s, \redh_\bullet$.

\section{Preparatory Material on Differential Forms}
\label{sec.abcomplexes}

Let $X^n$ be a stratified compact pseudomanifold (in the sense of 
Definition \ref{def.twostrataspace}) with two strata, the connected,
compact singular stratum $\Sigma^b$ and the top stratum $X-\Sigma$.
The singular set $\Sigma$ has a link bundle which we assume to be flat and isometrically
structured. Thus $\Sigma$ possesses an open tubular neighborhood $T$ in $X$
such that the boundary $\partial M$ of the compact manifold
$M= X- T$ is the total space of a flat fiber bundle 
$p: \partial M \rightarrow \Sigma$ with fiber $F^m,$ a closed
Riemannian $(m=n-1-b)$-dimensional manifold called the \emph{link}
of $\Sigma$. The structure group of $p$ is the isometries of $F$.
We shall write $B =\Sigma$ whenever we think of the singular
stratum as the base space of its link bundle.
Let $c: (-2,+1] \times \partial M \cong U$ be a smooth collar onto
an open neighborhood $U\subset M$ of the boundary,
$c(1,x) = x$ for $x\in \partial M.$ Via this diffeomorphism, we shall
subsequently write $(-2,+1] \times \partial M$ instead of $U$.
Let $N$ denote the interior of $M$. The noncompact manifold $N$ has
an \emph{end}, $E = (-1, +1) \times \partial M.$
Let $j: E \subset N$ be the inclusion of the end and
$\pi: E \rightarrow \partial M$ the second factor projection.
For any smooth manifold $X$, let $\Omega^\bullet (X)$ denote the
de Rham complex of smooth differential forms on $X$ and let
$\Omega^\bullet_c (X) \subset \Omega^\bullet (X)$ denote the subcomplex
of forms with compact support. The exterior
differential will be denoted by $d_X$ or simply $d$, if $X$ is understood. \\

We define a subspace $\Omega^p_{\rel} (N) \subset \Omega^p (N)$ by
\[ \Omega^p_{\rel} (N) = \{ \omega \in \Omega^p (N) ~|~
   j^\ast \omega = 0 \}. \]
The differential on $\Omega^\bullet (N)$ obviously restricts to
$\Omega^\bullet_{\rel} (N),$ so that we have a subcomplex
$(\Omega^\bullet_{\rel} (N), d)\subset (\Omega^\bullet (N), d)$.
Furthermore, any form on $N$ which vanishes on $E$ has compact support on $N$.
Thus, there is a subcomplex-inclusion
$\Omega^\bullet_{\rel} (N) \subset \Omega^\bullet_c (N)$.
Section \ref{sec.rel} is devoted to a proof of the following result. \\

\noindent \textbf{Proposition \ref{prop.rel}.}
\textit{The inclusion $\Omega^\bullet_{\rel} (N) \subset \Omega^\bullet_c (N)$
induces an isomorphism
\[ H^\bullet (\Omega^\bullet_{\rel} (N)) \cong H^\bullet_c (N), \]
that is, $\Omega^\bullet_{\rel} (N)$ computes the cohomology 
with compact supports of $N$.} \\

\noindent We shall henceforth also write $H^\bullet_{\rel} (N) = 
H^\bullet (\Omega^\bullet_{\rel} (N)).$ \\

\subsection{Forms Constant in the Collar Direction}
\label{sec.bh}

The goal of this section is to show that the complex
\[ \Omega^\bullet_{\bc} (N) = \{ \omega \in \Omega^\bullet (N) ~|~
 j^\ast \omega = \pi^\ast \eta,~ \text{ some } \eta \in \Omega^\bullet
 (\partial M) \} \]
of differential forms constant in the collar direction near the end of $N$
computes the cohomology of $N$. This goal will be achieved in Proposition \ref{lem.ombcinom}.
The restriction $j^\ast: \Omega^\bullet (N) \rightarrow
\Omega^\bullet (E)$ is not surjective. We put
$X^\bullet = \im j^\ast \subset \Omega^\bullet (E)$
and call a form in $X^\bullet$ \emph{extendable}.
The inclusion 
$j_{\cyl} = j\times \id_I: E\times I \subset N \times I$
induces a restriction map
\[ j^\ast_{\cyl}: \Omega^\bullet (N\times I) \longrightarrow
  \Omega^\bullet (E\times I). \]
Set
$X^\bullet (I) = \im j^\ast_{\cyl}.$
For $s\in [0,1]=I,$ let $i_{s,E}: E \rightarrow E\times I$ be the
embedding $i_{s,E}(x) = (x,s).$ These embeddings induce restriction
maps
$i^\ast_{s,E}: \Omega^\bullet (E\times I) \longrightarrow
  \Omega^\bullet (E).$
\begin{lemma} \label{lem.isx}
The maps $i^\ast_{s,E}$ restrict to maps
$i^\ast_{s,X}: X^\bullet (I) \longrightarrow X^\bullet.$
\end{lemma}
\begin{proof}
Define embeddings $i_{s,N}: N \rightarrow N\times I,$
$i_{s,N}(x)=(x,s),$ $x\in N,$ $s\in I$. The commutative square
\begin{equation} \label{equ.eeinni}
\xymatrix{
E \ar[r]^{i_{s,E}} \ar[d]_j & E\times I \ar[d]^{j_{\cyl}} \\
N \ar[r]_{i_{s,N}} & N\times I
} \end{equation}
induces a commutative square
\[ \xymatrix{
\Omega^\bullet (E\times I) \ar[r]^{i^\ast_{s,E}} & \Omega^\bullet (E) \\
\Omega^\bullet (N\times I) \ar[r]_{i^\ast_{s,N}} 
\ar[u]^{j^\ast_{\cyl}} & \Omega^\bullet (N). \ar[u]_{j^\ast}
} \]
Let $\omega \in X^\bullet (I).$ There is a form
$\overline{\omega} \in \Omega^\bullet (N\times I)$ such that
$j^\ast_{\cyl} \overline{\omega} = \omega$. The calculation
\[ i^\ast_{s,E}(\omega) =
  i^\ast_{s,E} j^\ast_{\cyl} (\overline{\omega}) =
  j^\ast i^\ast_{s,N} (\overline{\omega}) \]
shows that $i^\ast_{s,E}(\omega)$ lies in $\im j^\ast = X^\bullet$.
\end{proof}

\begin{lemma} \label{lem.kx}
There exists a homotopy operator $K_X: X^\bullet (I) \rightarrow
X^{\bullet -1}$ between $i^\ast_{0,X}$ and $i^\ast_{1,X}$, that is,
for $\omega \in X^\bullet (I),$ the formula
\[ dK_X (\omega) + K_X d(\omega) = i^\ast_{1,X}(\omega) - 
  i^\ast_{0,X}(\omega) \]
holds.
\end{lemma}
\begin{proof}
Let $K_E: \Omega^\bullet (E\times I) \rightarrow \Omega^{\bullet -1}(E)$
be the standard homotopy operator given by
\[ K_E (\omega) = \int_0^1 (\frac{\partial}{\partial s} \righthalfcup
 \omega) ds, \]
where $\frac{\partial}{\partial s} \righthalfcup \omega$ denotes
contraction of $\omega$ along the vector $\frac{\partial}{\partial s}$.
The operator $K_E$ satisfies
\[ dK_E + K_E d = i^\ast_{1,E} - i^\ast_{0,E} \]
on $\Omega^\bullet (E\times I)$. Similarly, let 
$K_N: \Omega^\bullet (N\times I) \rightarrow \Omega^{\bullet -1}(N)$
be the standard homotopy operator for $N$, constructed analogously and
satisfying
\[ dK_N + K_N d = i^\ast_{1,N} - i^\ast_{0,N}. \]
For $e\in E,$ $v_1, \ldots, v_{p-1} \in T_e E = T_{j(e)} N$ and
$\omega \in \Omega^p (N\times I),$ we calculate
\begin{eqnarray*}
(j^\ast K_N \omega)_e (v_1,\ldots, v_{p-1}) & = &
   (K_N \omega)_{j(e)} (v_1,\ldots, v_{p-1}) \\
& = & \int_0^1 \omega_{(j(e),s)} (\frac{\partial}{\partial s},
   v_1, \ldots, v_{p-1}) ds  \\
& = & \int_0^1 \omega_{j_{\cyl}(e,s)} (\frac{\partial}{\partial s},
   v_1, \ldots, v_{p-1}) ds  \\
& = & \int_0^1 (j^\ast_{\cyl} \omega)_{(e,s)} (\frac{\partial}{\partial s},
   v_1, \ldots, v_{p-1}) ds  \\
& = & (K_E j^\ast_{\cyl} \omega)_e (v_1, \ldots, v_{p-1}).
\end{eqnarray*}
Thus, the square
\[ \xymatrix{
\Omega^\bullet (N\times I) \ar[r]^{K_N} \ar[d]_{j^\ast_{\cyl}} &
 \Omega^{\bullet -1} (N) \ar[d]_{j^\ast} \\
\Omega^\bullet (E\times I) \ar[r]^{K_E}  &
 \Omega^{\bullet -1} (E)
} \]
commutes. We claim that $K_E$ restricts to an operator
$K_X: X^\bullet (I) \rightarrow X^{\bullet -1}.$
To verify the claim, let $\omega \in X^\bullet (I)$ be an extendable
form on the cylinder. By definition, there is a form
$\overline{\omega} \in \Omega^\bullet (N\times I)$ such that
$j^\ast_{\cyl} (\overline{\omega}) = \omega.$ Using the
commutativity of the above square, we compute
\[ K_E (\omega) = K_E j^\ast_{\cyl} (\overline{\omega}) =
 j^\ast K_N (\overline{\omega}) \in \im j^\ast = X^\bullet, \]
verifying the claim. This defines $K_X$. It is now easily verified that
this operator indeed satisfies
$dK_X (\omega) + K_X d(\omega) = i^\ast_{1,X}(\omega) - 
  i^\ast_{0,X}(\omega).$ 
\end{proof}
Let $\sigma_0:\partial M\to E$ be given by $\sigma_0 (x)=(0,x)\in
(-1,1)\times \partial M =E$.
\begin{lemma} \label{lem.homotopyxtoxi}
Let $H: E\times I \rightarrow E$ be the smooth homotopy
$H(t,x,s) = (ts,x),$ $(t,x)\in E,$ $s\in I,$ from
$H(\cdot, \cdot, 0)=\sigma_0 \pi$ to
$H(\cdot, \cdot, 1)=\id_E$. Then the induced map
$H^\ast: \Omega^\bullet (E) \rightarrow \Omega^\bullet (E\times I)$
restricts to a map
\[ H^\ast_X: X^\bullet \longrightarrow X^\bullet (I). \]
\end{lemma}
\begin{proof}
We enlarge the end slightly by setting $E_{-2} = (-2,1)\times \partial M$
with inclusion $j_{-2}: E_{-2} \hookrightarrow N$.
Define $H_{-2}: E_{-2}\times I \rightarrow E_{-2}$ by
\[ H_{-2} (t,x,s) = (ts,x),~ -2<t<+1,~ 0\leq s \leq 1. \]
For $t\in (-1,1),$ we have $H(t,x,s) = (ts,x) = H_{-2} (t,x,s)$ for all
$s\in [0,1].$ Thus $H_{-2}$ is an extension of $H$:
\[ \xymatrix{
E_{-2} \times I \ar[r]^{H_{-2}} & E_{-2} \\
E\times I \ar@{^{(}->}[u]^{\iota_{\cyl}} \ar[r]^H & 
E. \ar@{^{(}->}[u]^{\iota}
} \]
This square induces a commutative diagram
\[ \xymatrix{
\Omega^\bullet (E_{-2}) \ar[r]^{H^\ast_{-2}} \ar[d]_{\iota^\ast} &
 \Omega^\bullet (E_{-2} \times I) \ar[d]^{\iota^\ast_{\cyl}} \\
\Omega^\bullet (E) \ar[r]^{H^\ast} & \Omega^\bullet (E\times I).
} \]
We claim that $\im \iota^\ast \subset X^\bullet$: Let $\omega \in
\Omega^\bullet (E_{-2})$ and let $f: E_{-2} \rightarrow \real$ be
a smooth cutoff function which is identically $1$ on $E$ 
(where the collar coordinate $t$ has values $t\in (-1,1)$) and
identically zero for $t\leq -\frac{3}{2}$. Multiplication by this
cutoff function and extension by zero to all of $N$ yields a smooth form
$f\cdot \omega \in \Omega^\bullet (N)$
such that $j^\ast (f\cdot \omega) = \iota^\ast \omega$. It follows
that $\iota^\ast \omega \in \im j^\ast = X^\bullet$, which proves the claim.
This shows that we can restrict $\iota^\ast$ to obtain a map
$\iota^\ast_X: \Omega^\bullet (E_{-2}) \longrightarrow X^\bullet.$
Let us show that $\iota^\ast_X$ is surjective: If $\omega \in X^\bullet$
is an extendable form, then there exists a form
$\overline{\omega} \in \Omega^\bullet (N)$ with $j^\ast \overline{\omega}
= \omega.$ The surjectivity follows from
\[ \omega = j^\ast \overline{\omega} = (\overline{\omega}|_{E_{-2}})|_E
  = \iota^\ast_X (j^\ast_{-2} \overline{\omega}). \]

We shall next provide a similar construction for the cylinder.
We claim that $\im \iota^\ast_{\cyl} \subset X^\bullet (I)$: 
Let $\omega \in
\Omega^\bullet (E_{-2} \times I)$ and let 
$f_{\cyl}: E_{-2} \times I \rightarrow \real$ be
the smooth cutoff function $f_{\cyl} (t,s) = f(t),$ where $f$ is as above.
Multiplication by $f_{\cyl}$
and extension by zero to all of $N \times I$ yields a smooth form
$f_{\cyl} \cdot \omega \in \Omega^\bullet (N\times I)$
such that $j^\ast_{\cyl} (f_{\cyl}\cdot \omega) 
= \iota^\ast_{\cyl} \omega,$ since $f_{\cyl}$ is identically $1$ on
$E\times I$. It follows
that $\iota^\ast_{\cyl} \omega \in \im j^\ast_{\cyl} = X^\bullet (I)$, 
which proves the claim.
This shows that we can restrict $\iota^\ast_{\cyl}$ to obtain a map
\[ \iota^\ast_{\cyl,X}: \Omega^\bullet (E_{-2}\times I) 
  \longrightarrow X^\bullet (I). \]

Let $\omega \in X^\bullet$ be an extendable form. As $\iota^\ast_X$
is surjective, there is an $\overline{\omega} \in \Omega^\bullet (E_{-2})$
such that $\iota^\ast_X (\overline{\omega}) = \omega.$ We calculate
\[ H^\ast (\omega) = H^\ast \iota^\ast (\overline{\omega}) =
 \iota^\ast_{\cyl} (H^\ast_{-2}(\overline{\omega})) \in X^\bullet (I), \]
since $\im \iota^\ast_{\cyl} \subset X^\bullet (I).$ Hence $H^\ast$
is seen to map $X^\bullet$ into $X^\bullet (I)$ and the lemma is proved.
\end{proof}

The image of $\pi^\ast: \Omega^\bullet (\partial M) \rightarrow
\Omega^\bullet (E)$ lies in $X^\bullet$. Thus $\pi^\ast$ restricts to
a map
$\pi^\ast_X: \Omega^\bullet (\partial M) \longrightarrow X^\bullet.$
Restricting $\sigma^\ast_0: \Omega^\bullet (E) \rightarrow
\Omega^\bullet (\partial M)$ to $X^\bullet$, we get a map
$\sigma^\ast_{0,X}: X^\bullet \longrightarrow \Omega^\bullet (\partial M).$
\begin{lemma} \label{lem.xombndrymhe}
The maps
\[ \xymatrix{X^\bullet & \Omega^\bullet (\partial M) 
\ar@<1ex>[l]^{\pi^\ast_X}  \ar@<1ex>[l];[]^{\sigma^\ast_{0,X}}
} \]
are chain homotopy equivalences, which are chain homotopy inverse to
each other.
\end{lemma}
\begin{proof}
The composition
\[ \Omega^\bullet (\partial M) \stackrel{\pi^\ast_X}{\longrightarrow}
 X^\bullet \stackrel{\sigma^\ast_{0,X}}{\longrightarrow}
 \Omega^\bullet (\partial M) \]
is equal to the identity on $\Omega^\bullet (\partial M)$, since
$\pi_X \sigma_{0,X} = \id_{\partial M}.$ We have to prove that
\[ X^\bullet \stackrel{\sigma^\ast_{0,X}}{\longrightarrow}
 \Omega^\bullet (\partial M) \stackrel{\pi^\ast_X}{\longrightarrow}
 X^\bullet \]
is homotopic to the identity on $X^\bullet$.
Let $H: E\times I \rightarrow E$ be the homotopy of Lemma
\ref{lem.homotopyxtoxi}, from $H(\cdot, \cdot, 0)=\sigma_0 \pi$ to
$H(\cdot, \cdot, 1)=\id_E$, that is,
$H \circ i_{0,E} = \sigma_0 \pi,$ $H \circ i_{1,E} = \id_E.$
From the cube
\[ \xymatrix{
& \Omega^\bullet (\partial M) \ar[rr]^{\pi^\ast_X} \ar@{^{(}->}'[d][dd]
 & & X^\bullet \ar@{^{(}->}[dd] \\
X^\bullet \ar[ur]^{\sigma^\ast_{0,X}} \ar[rr]^>>>>{H^\ast_X \mbox{   }} 
 \ar@{^{(}->}[dd]
 & & X^\bullet (I) \ar[ur]_{i^\ast_{0,X}} \ar@{^{(}->}[dd] \\
& \Omega^\bullet (\partial M) \ar'[r][rr]^{\pi^\ast}
 & & \Omega^\bullet (E) \\
\Omega^\bullet (E) \ar[rr]^{H^\ast} \ar[ur]^{\sigma^\ast_0}
 & & \Omega^\bullet (E\times I), \ar[ur]_{i^\ast_{0,E}}
} \]
obtained by restricting the bottom face to the top face,
we see that for $\omega \in X^\bullet,$
\[ i^\ast_{0,X} H^\ast_X (\omega) = i^\ast_{0,E} H^\ast (\omega) =
 \pi^\ast \sigma^\ast_0 (\omega) = \pi^\ast_X \sigma^\ast_{0,X} (\omega). \]
(The map $H^\ast_X$ is provided by Lemma \ref{lem.homotopyxtoxi}.)
Analogously,
\[ i^\ast_{1,X} H^\ast_X (\omega) = i^\ast_{1,E} H^\ast (\omega)
=\omega. \]
Composing the homotopy operator $K_X$ of Lemma \ref{lem.kx}
with $H^\ast_X,$ we obtain a map
\[ L = K_X \circ H^\ast_X: X^\bullet
 \longrightarrow X^{\bullet -1} \]
such that for $\omega \in X^\bullet,$
\begin{eqnarray*}
Ld(\omega) + dL(\omega) & = &
  K_X H^\ast_X d(\omega) + d K_X H^\ast_X
  (\omega) 
 =  K_X d (H^\ast_X \omega) + d K_X (H^\ast_X \omega) \\
& = & i^\ast_{1,X} (H^\ast_X \omega) - 
        i^\ast_{0,X} (H^\ast_X \omega) 
 =  \id_{X^\bullet} (\omega) - \pi^\ast_X \sigma^\ast_{0,X} (\omega).
\end{eqnarray*}
Thus $L$ is a cochain homotopy between $\pi^\ast_X \sigma^\ast_{0,X}$ and the
identity.
\end{proof}

Put
$\Omega^\bullet_{\bc} (E) = \{ \omega \in \Omega^\bullet (E) ~|~
 \omega = \pi^\ast \eta,~ \text{ some } \eta \in \Omega^\bullet
 (\partial M) \}.$
\begin{prop} \label{lem.ombcinom}
The inclusion $\Omega^\bullet_{\bc} (N) \subset \Omega^\bullet (N)$
induces a cohomology isomorphism.
\end{prop}
\begin{proof}
If a form on $E$ is constant in the collar coordinate, then it is
extendable to all of $N$ by using a slightly larger collar and multiplication 
by a cutoff function. Thus there is an inclusion map
$\iota: \Omega^\bullet_{\bc} (E) \rightarrow X^\bullet$. We shall show
first that this map induces a cohomology isomorphism, in fact, that it
is a homotopy equivalence. 
The maps
\[ \xymatrix{\Omega^\bullet_{\bc}(E) & \Omega^\bullet (\partial M) 
\ar@<1ex>[l]^{\pi^\ast}  \ar@<1ex>[l];[]^{\sigma^\ast_0}
} \]
are mutually inverse isomorphisms of cochain complexes.
(Compare to Lemma \ref{lem.ombheisoharmbm} and its proof.)
By Lemma \ref{lem.xombndrymhe}, the map
$\pi^\ast_X: \Omega^\bullet (\partial M)\rightarrow X^\bullet$ is a
homotopy equivalence.
For $\omega \in \Omega^\bullet_{\bc} (E),$ there is an 
$\eta \in \Omega^\bullet (\partial M)$ with $\omega = \pi^\ast \eta$
and we compute
\[ \pi^\ast_X \sigma^\ast_0 \omega = \pi^\ast_X \sigma^\ast_0 \pi^\ast \eta
 = \pi^\ast_X \eta = \pi^\ast \eta = \omega = \iota (\omega). \]
Thus we have expressed $\iota = \pi^\ast_X \sigma^\ast_0$ as the
composition of an isomorphism and a homotopy equivalence, whence
$\iota$ itself is a homotopy equivalence. 
The kernel of the restriction $j^\ast: \Omega^\bullet (N)\rightarrow
\Omega^\bullet (E)$ is $\Omega^\bullet_{\rel} (N)$.
Consequently, there is a short exact sequence
\[ 0 \rightarrow \Omega^\bullet_{\rel} (N) \longrightarrow
 \Omega^\bullet (N) \stackrel{j^\ast}{\longrightarrow}
 X^\bullet \rightarrow 0. \]
The restriction map $\Omega^\bullet_{\bc} (N) \rightarrow
\Omega^\bullet_{\bc} (E)$ is onto. Since its kernel is again
$\Omega^\bullet_{\rel} (N),$ we get another exact sequence
\[ 0 \rightarrow \Omega^\bullet_{\rel} (N) \longrightarrow
 \Omega^\bullet_{\bc} (N) \longrightarrow
 \Omega^\bullet_{\bc} (E) \rightarrow 0. \]
The various inclusions yield a commutative diagram
\[ \xymatrix{
0 \ar[r] & \Omega^\bullet_{\rel} (N) \ar[r] &
 \Omega^\bullet (N) \ar[r]^{j^\ast} & X^\bullet \ar[r] & 0 \\
0 \ar[r] & \Omega^\bullet_{\rel} (N) \ar[r] \ar@{=}[u] &
 \Omega^\bullet_{\bc} (N) \ar[r] \ar@{^{(}->}[u]^{\iota_N}
& \Omega^\bullet_{\bc} (E) \ar[r] \ar@{^{(}->}[u]^{\iota} & 0, 
} \]
which induces on cohomology a commutative diagram
\[ \xymatrix{
H^\bullet_{\rel} (N) \ar[r] &
 H^\bullet (N) \ar[r]^{j^\ast} & H^\bullet (X^\bullet) \ar[r] & 
  H^{\bullet +1}_{\rel} (N) \\
H^\bullet_{\rel} (N) \ar[r] \ar@{=}[u] &
 H^\bullet_{\bc} (N) \ar[r] \ar[u]^{\iota^\ast_N}
& H^\bullet_{\bc} (E) \ar[r] \ar[u]^{\iota^\ast}_{\cong} & 
 H^{\bullet +1}_{\rel} (N). \ar@{=}[u] 
} \]
By the $5$-lemma, $\iota^\ast_N$ is an isomorphism.
\end{proof}

\subsection{Forms Vanishing Near the Boundary}
\label{sec.rel}

This section is devoted to a proof of Proposition \ref{prop.rel}.
Recall that $i_{s,N}: N \rightarrow N\times I$ are the embeddings
$i_{s,N} (x) = (x,s),$ $x\in N,$ $s\in I$, and 
$j_{\cyl} = j\times \id_I: E\times I \hookrightarrow N\times I$.
We put
\[ \Omega^\bullet_{\rel} (N\times I) = \{ \omega \in
 \Omega^\bullet (N\times I) ~|~ j^\ast_{\cyl} \omega =0 \}. \]
The $i_{s,N}$ induce maps $i^\ast_{s,N}: \Omega^\bullet (N\times I)
\rightarrow \Omega^\bullet (N),$ which restrict to maps
\[ i^\ast_{s,\rel}: \Omega^\bullet_{\rel} (N\times I)
 \longrightarrow \Omega^\bullet_{\rel} (N) \]
because
$j^\ast i^\ast_{s,N} (\omega) = i^\ast_{s,E} j^\ast_{\cyl} (\omega)=0$
for $\omega \in \Omega^\bullet_{\rel} (N\times I),$ as follows from
the commutative diagram (\ref{equ.eeinni}).
\begin{lemma} \label{lem.existskrel}
There exists a homotopy operator $K_{\rel}: \Omega^\bullet_{\rel}
(N\times I) \rightarrow \Omega^{\bullet -1}_{\rel} (N)$ between
$i^\ast_{0,\rel}$ and $i^\ast_{1,\rel}$, that is, for 
$\omega \in \Omega^\bullet_{\rel} (N\times I),$ the formula
\[ dK_{\rel} (\omega) + K_{\rel} d(\omega) =
  i^\ast_{1,\rel}(\omega) - i^\ast_{0,\rel}(\omega) \]
holds.
\end{lemma}
\begin{proof}
Let $K_N: \Omega^\bullet (N\times I) \rightarrow \Omega^{\bullet -1}(N)$
be the homotopy operator for $N$ used in the proof of Lemma \ref{lem.kx}, given by
\[ (K_N \omega)_{x} (v_1,\ldots, v_{p-1}) 
= \int_0^1 \omega_{(x,s)} (\frac{\partial}{\partial s},
   v_1, \ldots, v_{p-1}) ds, \]
$\omega \in \Omega^p (N\times I),$ $x\in N,$ $v_1, \ldots, v_{p-1} \in
T_x N.$ If $\omega \in \Omega^p_{\rel} (N\times I)$ and $x\in E$, then
$\omega_{(x,s)} =0$ for all $s\in I$. Thus $(K_N \omega)_x =0$ for all
$x\in E$, which places $K_N \omega$ in $\Omega^{p-1}_{\rel} (N)$.
We conclude that $K_N$ restricts to an operator
$K_{\rel}: \Omega^\bullet_{\rel} (N\times I) \longrightarrow
 \Omega^{\bullet -1}_{\rel} (N).$
It possesses the desired property:
\begin{eqnarray*}
dK_{\rel} (\omega) + K_{\rel} d(\omega) & = &
  dK_N (\omega) + K_N d(\omega) 
 =  i^\ast_{1,N}(\omega) - i^\ast_{0,N}(\omega) \\
& = & i^\ast_{1,\rel}(\omega) - i^\ast_{0,\rel}(\omega).
\end{eqnarray*}
\end{proof}
We omit the straightforward proof of the next lemma.
\begin{lemma} \label{lem.phistarrel}
Let $\phi: N\times I \rightarrow N$ be a smooth homotopy such that
$\phi (E\times I) \subset E.$ Then the induced map
$\phi^\ast: \Omega^\bullet (N) \rightarrow \Omega^\bullet (N\times I)$
restricts to a map
\[ \phi^\ast_{\rel}: \Omega^\bullet_{\rel} (N) 
  \longrightarrow \Omega^\bullet_{\rel} (N\times I). \]
\end{lemma}
Let $\phi_s,$ $s\in \real,$ be a smooth one-parameter family of
diffeomorphisms $\phi_s: N\rightarrow N$ such that
$\phi_0 = \id_N$,
$\phi_s (E) \subset E$ for all $s$, and
$\phi_1 ((-2,1)\times \partial M) =E.$
By Lemma \ref{lem.phistarrel}, $\phi$ induces a map
$\phi^\ast_{\rel}: \Omega^\bullet_{\rel} (N) 
  \rightarrow \Omega^\bullet_{\rel} (N\times I).$

\begin{lemma} \label{lem.phionestarhtpcid}
The map $\phi^\ast_1: \Omega^\bullet_{\rel} (N) \rightarrow
\Omega^\bullet_{\rel} (N)$ is homotopic to the identity.
\end{lemma}
\begin{proof}
Composing the homotopy operator $K_{\rel}$ of Lemma \ref{lem.existskrel}
with $\phi^\ast_{\rel},$ we obtain a map
$L = K_{\rel} \circ \phi^\ast_{\rel}: \Omega^\bullet_{\rel} (N)
 \longrightarrow \Omega^{\bullet -1}_{\rel} (N)$
such that for $\omega \in \Omega^\bullet_{\rel} (N),$
\begin{eqnarray*}
Ld(\omega) + dL(\omega) & = &
  K_{\rel} \phi^\ast_{\rel} d(\omega) + d K_{\rel} \phi^\ast_{\rel}
  (\omega) 
 =  K_{\rel} d (\phi^\ast_{\rel} \omega) + d K_{\rel} (\phi^\ast_{\rel}
  \omega) \\
& = & i^\ast_{1,\rel} (\phi^\ast_{\rel} \omega) - 
        i^\ast_{0,\rel} (\phi^\ast_{\rel} \omega) 
 =  \phi^\ast_1 (\omega) - \omega.
\end{eqnarray*}
Thus $L$ is a cochain homotopy between $\phi^\ast_1$ and the
identity.
\end{proof}

\begin{prop} \label{prop.rel}
The inclusion $\Omega^\bullet_{\rel} (N) \subset \Omega^\bullet_c (N)$
induces an isomorphism
\[ H^\bullet (\Omega^\bullet_{\rel} (N)) \cong H^\bullet_c (N), \]
that is, $\Omega^\bullet_{\rel} (N)$ computes the cohomology 
with compact supports of $N$.
\end{prop} 
\begin{proof}
Set
$N_{<-3/2} = N - ([-\frac{3}{2},1) \times \partial M)$
and
\[ \Omega^\bullet_{-2,\rel} (N) = \{ \omega \in \Omega^\bullet (N)
  ~|~ \omega|_{(-2,1)\times \partial M} =0 \}. \]
Suppose that $x\in N$ lies in $(-2,1)\times \partial M$ and
$\omega \in \Omega^\bullet_{\rel} (N).$ Then, as
\[ (\phi^\ast_1 \omega)_x (v_1, \ldots, v_p) = \omega_{\phi_1 (x)}
 (\phi_{1\ast} v_1, \ldots, \phi_{1\ast} v_p) \]
and $\phi_1 (x)\in E,$ we have that $\phi^\ast_1 \omega \in
\Omega^\bullet_{-2,\rel} (N).$ Therefore, the map
$\phi^\ast_1: \Omega^\bullet_{\rel} (N) \rightarrow \Omega^\bullet_{\rel}(N)$
of Lemma \ref{lem.phionestarhtpcid} factors as
\[ \phi^\ast_1: \Omega^\bullet_{\rel}(N) \rightarrow
 \Omega^\bullet_{-2,\rel}(N) \hookrightarrow
  \Omega^\bullet_c (N_{<-3/2}) \stackrel{\rho}{\hookrightarrow}
  \Omega^\bullet_{\rel} (N), \]
where $\rho$ is extension by zero.
Let us denote the composition of the first two maps by
$\phi^\ast_{1,c}: \Omega^\bullet_{\rel}(N) \rightarrow
  \Omega^\bullet_c (N_{<-3/2}).$
By Lemma \ref{lem.phionestarhtpcid}, $\rho \circ \phi^\ast_{1,c}$ is
homotopic to $\id_{\Omega^\bullet_{\rel}(N)}$. Thus, the induced
composition on cohomology,
\[ H^\bullet_{\rel} (N) \stackrel{\phi^\ast_{1,c}}{\longrightarrow}
  H^\bullet_c (N_{<-3/2}) \stackrel{\rho}{\longrightarrow}
  H^\bullet_{\rel} (N) \]
is equal to the identity. We deduce that
$\rho: H^\bullet_c (N_{<-3/2}) \twoheadrightarrow
  H^\bullet_{\rel} (N)$
is surjective. Since
$H^\bullet_c (N - N_{<-3/2}) = 
  H^\bullet_c ([-\frac{3}{2},1)\times \partial M) =0,$
the long exact sequence
\[ \cdots \longrightarrow
  H^\bullet_c (N_{<-3/2}) \stackrel{\gamma}{\longrightarrow}
  H^\bullet_c (N) \longrightarrow
  H^\bullet_c (N - N_{<-3/2}) \longrightarrow \cdots \]
implies that the map $\gamma$ induced by the inclusion
$\Omega^\bullet_c (N_{<-3/2}) \subset \Omega^\bullet_c (N)$
(extension by zero) is an isomorphism. 
Let $\alpha: H^\bullet_{\rel} (N) \rightarrow H^\bullet_c (N)$ be the
map induced by the inclusion $\Omega^\bullet_{\rel} (N) \subset
\Omega^\bullet_c (N).$ The inclusion
$\Omega^\bullet_c (N_{<-3/2}) \subset \Omega^\bullet_c (N)$
factors through $\Omega^\bullet_{\rel} (N)$. Thus there is a commutative
diagram
\[ \xymatrix{
H^\bullet_c (N_{<-3/2}) \ar[rd]^{\gamma}_{\cong}
  \ar@{->>}[d]_{\rho} & \\
H^\bullet_{\rel} (N) \ar[r]_\alpha & H^\bullet_c (N). 
} \]
Since $\gamma$ is an isomorphism, $\rho$ is injective, hence an
isomorphism. Thus $\alpha$ is an isomorphism as well.
\end{proof}

\section{A Complex of Multiplicatively Structured Forms on Flat Bundles}
\label{sec.fibharmflat}

Let $F$ be a closed, oriented, Riemannian manifold and $p: E\rightarrow B$
a flat, smooth fiber bundle over the closed smooth base manifold $B^n$ with
fiber $F$. An open cover of an
$n$-manifold is called \emph{good}, if all nonempty finite intersections of
sets in the cover are diffeomorphic to $\real^n$. Every smooth manifold
has a good cover and if the manifold is compact, then the cover can be
chosen to be finite. Let $\mathfrak{U} = \{ U_\alpha \}$ be a finite good open
cover of the base $B$ such that $p$ trivializes with respect to $\mathfrak{U}$.
Let $\{ \phi_\alpha \}$ be a system of local trivializations, that is, the
$\phi_\alpha$ are diffeomorphisms such that
\[ \xymatrix{
p^{-1} (U_\alpha) \ar[rr]^{\phi_\alpha} \ar[rd]_{p|} 
& & U_\alpha \times F \ar[ld]^{\pi_1} \\
& U_\alpha & } \]
commutes for every $\alpha$. Flatness means that the transition functions
\[ \rho_{\beta \alpha} = \phi_\beta| \circ \phi_\alpha|^{-1}:
 (U_\alpha \cap U_\beta)\times F \longrightarrow
  p^{-1} (U_\alpha \cap U_\beta) \longrightarrow 
  (U_\alpha \cap U_\beta)\times F \]
are of the form
$\rho_{\beta \alpha} (t,x) = (t, g_{\beta \alpha} (x)).$
If $X$ is a topological
space, let $\pi_2: X\times F \rightarrow F$ denote the second-factor
projection. Let $V \subset B$ be a $\mathfrak{U}$-small open subset
and suppose that $V\subset U_{\alpha}$.
\begin{defn}
A differential form $\omega \in \Omega^q (p^{-1} (V))$ is called
\emph{$\alpha$-multiplicatively structured}, if it has the form
\[ \omega = \phi^\ast_\alpha \sum_j \pi^\ast_1 \eta_j \wedge
  \pi^\ast_2 \gamma_j,~ 
 \eta_j \in \Omega^\bullet (V),~ \gamma_j \in \Omega^\bullet (F) \]
(finite sums).
\end{defn}

Flatness is crucial for the following basic lemma.
\begin{lemma} \label{lem.abinvariant}
Suppose $V \subset U_\alpha \cap U_\beta.$ Then $\omega$ is
$\alpha$-multiplicatively structured if, and only if, $\omega$ is $\beta$-multiplicatively structured.
\end{lemma}
\begin{proof}
The flatness allows us to construct a commutative diagram
\[ \xymatrix{
(U_\alpha \cap U_\beta)\times F \ar[r]^{\rho_{\alpha \beta}}
 \ar[d]_{\pi_2} & (U_\alpha \cap U_\beta)\times F 
 \ar[d]_{\pi_2} \\
F \ar[r]^{g_{\alpha \beta}} & F. } \]
If the form is $\alpha$-multiplicatively structured, then, using the equations
\[ \pi_1 \rho_{\alpha \beta} = \pi_1,~ \pi_2 \rho_{\alpha \beta} =
 g_{\alpha \beta} \pi_2, \]
we derive the \emph{transformation law}
\begin{eqnarray*}
\omega & = & \phi^\ast_\alpha \sum_j \pi^\ast_1 \eta_j \wedge
  \pi^\ast_2 \gamma_j 
 =  \phi^\ast_\beta (\phi^{-1}_\beta)^\ast \phi^\ast_\alpha \sum_j \pi^\ast_1 \eta_j \wedge
  \pi^\ast_2 \gamma_j \\
& = & \phi^\ast_\beta \sum_j \rho^\ast_{\alpha \beta}\pi^\ast_1 \eta_j \wedge
  \rho^\ast_{\alpha \beta} \pi^\ast_2 \gamma_j 
 =  \phi^\ast_\beta \sum_j \pi^\ast_1 \eta_j \wedge
  \pi^\ast_2 (g^\ast_{\alpha \beta} \gamma_j).
\end{eqnarray*}
Thus $\omega$ is $\beta$-multiplicatively structured.
The converse implication follows from symmetry.
\end{proof}
The lemma shows that the property of being
multiplicatively structured over $V$ is invariantly defined, independent
of the choice of $\alpha$ such that $V\subset U_\alpha$. 
We will use the shorthand notation
\[ U_{\alpha_0 \ldots \alpha_k} = U_{\alpha_0} \cap \cdots \cap U_{\alpha_k} \]
for multiple intersections. (Repetitions are allowed.)
Since $\mathfrak{U}$ is a good cover, every
$U_{\alpha_0 \ldots \alpha_k}$ is diffeomorphic to $\real^n,$ $n = \dim B$.
A linear subspace, the subspace of \emph{multiplicatively structured forms}, of $\Omega^q (E)$
is obtained by setting
\[ \Omega^q_{\ms} (B) = \{ \omega \in \Omega^q (E) ~|~ \omega|_{p^{-1} U_\alpha} \text{ is }
   \alpha \text{-multiplicatively structured for all } \alpha \}. \]
The Leibniz rule applied to a term of the form $\pi^\ast_1 \eta \wedge \pi^\ast_2 \gamma$ shows:
\begin{lemma} \label{lem.domfh}
The de Rham differential $d: \Omega^q (E) \rightarrow \Omega^{q+1}
(E)$ restricts to a differential
\[ d: \Omega^q_\ms (B) \longrightarrow \Omega^{q+1}_\ms (B). \]
\end{lemma}
This lemma shows that $\Omega^\bullet_\ms (B) \subset \Omega^\bullet (E)$ is a 
subcomplex. We shall eventually see that this inclusion is a quasi-isomorphism, that is,
induces isomorphisms on cohomology. For any $\alpha$, set
\[ \Omega^\bullet_{\ms} (U_\alpha) = \{ \omega \in \Omega^\bullet (p^{-1} U_\alpha) 
  ~|~ \omega \text{ is } \alpha \text{-multiplicatively structured} \}. \]
Let $r$ denote the obvious restriction map
\[ r: \Omega^\bullet_\ms (B) \longrightarrow \prod_\alpha \Omega^\bullet_\ms (U_\alpha). \]
If $k$ is positive, then we set
\[ \Omega^\bullet_{\ms} (U_{\alpha_0 \ldots \alpha_k}) = \{ \omega \in \Omega^\bullet (p^{-1} U_{\alpha_0 \ldots \alpha_k}) 
  ~|~ \omega \text{ is } \alpha_0 \text{-multiplicatively structured} \}. \]
Lemma \ref{lem.abinvariant} implies that for any $1\leq j \leq k,$
\[ \Omega^\bullet_{\ms} (U_{\alpha_0 \ldots \alpha_k}) = \{ \omega \in \Omega^\bullet (p^{-1} U_{\alpha_0 \ldots \alpha_k}) 
  ~|~ \omega \text{ is } \alpha_j \text{-multiplicatively structured} \}. \]
In particular, if $\sigma$ is any permutation of $0,1,\ldots, k$, then
\[ \Omega^\bullet_\ms (U_{\alpha_{\sigma (0)} \ldots \alpha_{\sigma (k)}}) =
  \Omega^\bullet_\ms (U_{\alpha_0 \ldots \alpha_k}). \]
The components of an element 
\[ \xi \in \prod_{\alpha_0, \ldots, \alpha_k} \Omega^\bullet_\ms (U_{\alpha_0 \ldots
\alpha_k}) \]
will be written as
$\xi_{\alpha_0 \ldots \alpha_k} \in \Omega^\bullet_\ms (U_{\alpha_0 \ldots
\alpha_k}).$ We impose the antisymmetry restriction 
$\xi_{\ldots \alpha_i \ldots \alpha_j \ldots} = - \xi_{\ldots \alpha_j \ldots \alpha_i \ldots}$
upon interchange of two indices. In particular, if $\alpha_0,\ldots, \alpha_k$ contains a repetition,
then $\xi_{\alpha_0 \ldots \alpha_k}=0.$
The difference operator
\[ \delta: \prod \Omega^\bullet (p^{-1} U_{\alpha_0 \ldots \alpha_k}) \longrightarrow
\prod \Omega^\bullet (p^{-1} U_{\alpha_0 \ldots \alpha_{k+1}}), \]
defined by
\[ (\delta \xi)_{\alpha_0 \ldots \alpha_{k+1}} = \sum_{j=0}^{k+1} (-1)^j
 \xi_{\alpha_0 \ldots \hat{\alpha}_j \ldots \alpha_{k+1}}|_{p^{-1}
  U_{\alpha_0 \ldots \alpha_{k+1}}} \]
and satisfying $\delta^2 =0,$ restricts to a difference operator
\[ \delta: \prod \Omega^\bullet_\ms (U_{\alpha_0 \ldots \alpha_k}) \longrightarrow
\prod \Omega^\bullet_\ms (U_{\alpha_0 \ldots \alpha_{k+1}}). \]
Since the de Rham differential $d$ commutes with restriction to open subsets, we have
$d\delta = \delta d$. Thus
\[ C^k (\mathfrak{U}; \Omega^q_\ms) = \prod \Omega^q_\ms (U_{\alpha_0 \ldots
 \alpha_k}) \]
is a double complex with horizontal differential $\delta$ and vertical differential $d$.
The associated simple complex $C^\bullet_\ms (\mathfrak{U})$ has groups
\[ C^j_\ms (\mathfrak{U}) = \bigoplus_{k+q=j} C^k (\mathfrak{U}; \Omega^q_\ms) \]
in degree $j$ and differential $D = \delta + (-1)^k d$ on 
$C^k (\mathfrak{U}; \Omega^q_\ms)$.
We shall refer to the double
complex $(C^\bullet (\mathfrak{U}; \Omega^\bullet_\ms), \delta, d)$ as the
\emph{multiplicatively structured \v{C}ech-de Rham complex}. Let us explicitly record
the following standard tool:
\begin{lemma} \label{lem.ext}
Let $M$ be a smooth manifold, $U\subset M$ an open subset and $\omega \in
\Omega^\bullet (U)$. If $f\in \Omega^0 (M)$ is a function with $\operatorname{supp}
(f)\subset U,$ then
\[ \overline{\omega}(x) = \begin{cases} f(x)\cdot \omega (x), & x\in U \\
0 & x\in M-U \end{cases} \]
defines a smooth form $\overline{\omega} \in \Omega^\bullet (M)$ on all of $M$.
\end{lemma}

\begin{lemma} \label{lem.gmv} (Generalized Mayer-Vietoris sequence.)
The sequence
\[ 0 \longrightarrow \Omega^\bullet_\ms (B)
 \stackrel{r}{\longrightarrow} C^0 (\mathfrak{U}; \Omega^\bullet_\ms)
 \stackrel{\delta}{\longrightarrow} C^1 (\mathfrak{U}; \Omega^\bullet_\ms)
 \stackrel{\delta}{\longrightarrow} C^2 (\mathfrak{U}; \Omega^\bullet_\ms)
 \stackrel{\delta}{\longrightarrow} \cdots \]
is exact.
\end{lemma}
\begin{proof}
The injectivity of $r$ is clear.
If $\{ \omega_{\alpha_0} \},$ $\omega_{\alpha_0} \in \Omega^\bullet_\ms (U_{\alpha_0})
\subset \Omega^\bullet (p^{-1} U_{\alpha_0}),$ is a family of forms which agree on
overlaps $p^{-1} (U_{\alpha_0 \alpha_1})$, then there exists a unique global differential
form $\omega \in \Omega^\bullet (E),$ which restricts to $\omega_{\alpha_0}$ on
$p^{-1} (U_{\alpha_0})$ for every $\alpha_0$. By definition of $\Omega^\bullet_\ms (B)$,
$\omega$ actually lies in $\Omega^\bullet_\ms (B)\subset \Omega^\bullet (E)$.
Thus the sequence is exact at $C^0 (\mathfrak{U}; \Omega^\bullet_\ms)$.
Now let $k$ be positive. Let $\{ \rho_\alpha \}$ be a smooth partition of unity on $B$
subordinate to $\mathfrak{U},$ $\supp (\rho_\alpha)\subset U_\alpha$.
The family of inverses $p^{-1}\mathfrak{U} = 
\{ p^{-1} U_\alpha \}$ is an open cover of $E$. The family $\{ \overline{\rho}_\alpha \}$
of functions $\overline{\rho}_\alpha = \rho_\alpha \circ p:E \rightarrow [0,\infty )$ is a 
smooth partition of unity subordinate to $p^{-1}\mathfrak{U}.$
Let $\omega \in C^k (\mathfrak{U}; \Omega^\bullet_\ms)$ be a cocycle,
$\delta \omega =0.$ This implies that
\begin{equation} \label{equ.delcocyc}
0 = (\delta \omega)_{\alpha \alpha_0 \ldots \alpha_k} =
\omega_{\alpha_0 \ldots \alpha_k} + \sum_j (-1)^{j+1}
\omega_{\alpha \alpha_0 \ldots \hat{\alpha}_j \ldots \alpha_k}.
\end{equation}
Applying Lemma \ref{lem.ext} with $M= p^{-1}(U_{\alpha_0 \ldots \alpha_{k-1}}),$
$U = p^{-1}(U_{\alpha \alpha_0 \ldots \alpha_{k-1}}),$ to the form
$\omega_{\alpha \alpha_0 \ldots \alpha_{k-1}} \in \Omega^\bullet (U),$ and taking
$f = \overline{\rho}_\alpha| \in \Omega^0 (M)$ with
\[ \operatorname{supp}(\overline{\rho}_\alpha|) \subset p^{-1}(U_\alpha)\cap M =
 p^{-1}(U_{\alpha \alpha_0 \ldots \alpha_{k-1}}) = U, \]
we receive a smooth form $\overline{\omega}_{\alpha \alpha_0 \ldots \alpha_{k-1}} \in
\Omega^\bullet (p^{-1}U_{\alpha_0 \ldots \alpha_{k-1}}),$ obtained from
$\overline{\rho}_\alpha \cdot \omega_{\alpha \alpha_0 \ldots \alpha_{k-1}}$ by
extension by zero. We shall show that in fact
\[ \overline{\omega}_{\alpha \alpha_0 \ldots \alpha_{k-1}} \in
\Omega^\bullet_\ms (U_{\alpha_0 \ldots \alpha_{k-1}}) \subset
\Omega^\bullet (p^{-1}U_{\alpha_0 \ldots \alpha_{k-1}}). \]
Since $\omega_{\alpha \alpha_0 \ldots \alpha_{k-1}} \in \Omega^\bullet_\ms
(U_{\alpha \alpha_0 \ldots \alpha_{k-1}}),$ it is $\alpha$-multiplicatively structured and thus, by
Lemma \ref{lem.abinvariant}, $\alpha_0$-multiplicatively structured. Hence it has the form
\[ \omega_{\alpha \alpha_0 \ldots \alpha_{k-1}} = \phi^\ast_{\alpha_0}
 \sum_j \pi^\ast_1 \eta_j \wedge \pi^\ast_2 \gamma_j, \]
for some $\eta_j \in \Omega^\bullet (U_{\alpha \alpha_0 \ldots \alpha_{k-1}}),$
$\gamma_j \in \Omega^\bullet (F).$ Therefore,
\begin{eqnarray*}
\overline{\omega}_{\alpha \alpha_0 \ldots \alpha_{k-1}} & = &
\overline{\rho}_\alpha \cdot \phi^\ast_{\alpha_0} \sum_j 
  \pi^\ast_1 \eta_j \wedge \pi^\ast_2 \gamma_j 
 =  p^\ast (\rho_\alpha) \wedge \phi^\ast_{\alpha_0} \sum_j 
  \pi^\ast_1 \eta_j \wedge \pi^\ast_2 \gamma_j \\
& = & \phi^\ast_{\alpha_0} (\pi^\ast_1 \rho_\alpha) \wedge \phi^\ast_{\alpha_0} \sum_j 
  \pi^\ast_1 \eta_j \wedge \pi^\ast_2 \gamma_j 
 =  \phi^\ast_{\alpha_0} (\pi^\ast_1 \rho_\alpha \wedge \sum_j 
  \pi^\ast_1 \eta_j \wedge \pi^\ast_2 \gamma_j) \\
& = & \phi^\ast_{\alpha_0} \sum_j 
  \pi^\ast_1 \rho_\alpha \wedge \pi^\ast_1 \eta_j \wedge \pi^\ast_2 \gamma_j 
 =  \phi^\ast_{\alpha_0} \sum_j 
  \pi^\ast_1 (\rho_\alpha \eta_j) \wedge \pi^\ast_2 \gamma_j.
\end{eqnarray*}
Again by Lemma \ref{lem.ext}, extension by zero allows us to regard $\rho_\alpha \eta_j$ as
a smooth form on $U_{\alpha_0 \ldots \alpha_{k-1}}$. We have thus exhibited
$\overline{\omega}_{\alpha \alpha_0 \ldots \alpha_{k-1}}$ as an element of
$\Omega^\bullet_\ms (U_{\alpha_0 \ldots \alpha_{k-1}})$. Define an element $\tau \in
C^{k-1} (\mathfrak{U}; \Omega^\bullet_\ms)$ by
\[ \tau_{\alpha_0 \ldots \alpha_{k-1}} = \sum_{\alpha} \overline{\omega}_{\alpha
\alpha_0 \ldots \alpha_{k-1}} \in \Omega^\bullet_\ms (U_{\alpha_0 \ldots \alpha_{k-1}}). \]
The calculation
\begin{eqnarray*}
(\delta \tau)_{\alpha_0 \ldots \alpha_k}
& = & \sum_j (-1)^j \tau_{\alpha_0 \ldots \hat{\alpha}_j \ldots \alpha_k} 
 =  \sum_j (-1)^j \sum_\alpha \overline{\omega}_{\alpha \alpha_0 \ldots 
  \hat{\alpha}_j \ldots \alpha_k} \\
& = & \sum_j (-1)^j \sum_\alpha \overline{\rho}_\alpha  \omega_{\alpha \alpha_0 \ldots 
  \hat{\alpha}_j \ldots \alpha_k} 
 =  \sum_\alpha \overline{\rho}_\alpha \sum_j (-1)^j \omega_{\alpha \alpha_0 \ldots 
  \hat{\alpha}_j \ldots \alpha_k} \\
& = & \sum_\alpha \overline{\rho}_\alpha \cdot \omega_{\alpha_0 \ldots 
  \alpha_k} \hspace{1cm} \text{ (by (\ref{equ.delcocyc}))} \\
& = & \omega_{\alpha_0 \ldots \alpha_k}
\end{eqnarray*}
shows that $\delta \tau = \omega.$ Since $\delta^2 =0,$ the exactness of the $\delta$-sequence
follows.
\end{proof}
We recall a fundamental fact about double complexes.
\begin{prop} \label{prop.rowex}
If all the rows of an augmented double complex are exact, then the augmentation map
induces an isomorphism from the cohomology of the augmentation column to the
cohomology of the simple complex associated to the double complex.
\end{prop}
This fact is applied in showing:
\begin{prop} \label{prop.simp1}
The restriction map $r: \Omega^\bullet_\ms (B) \rightarrow
C^0 (\mathfrak{U}; \Omega^\bullet_\ms)$ induces an isomorphism
\[ r^\ast: H^\bullet (\Omega^\bullet_\ms (B)) \stackrel{\cong}{\longrightarrow}
H^\bullet (C^\bullet_\ms (\mathfrak{U}), D). \]
\end{prop}
\begin{proof}
The map $r$ makes $C^\bullet (\mathfrak{U}; \Omega^\bullet_\ms)$ into an
augmented double complex. By the generalized Mayer-Vietoris sequence, Lemma \ref{lem.gmv},
all rows of this augmented complex are exact. According to Proposition \ref{prop.rowex},
$r^\ast$ is an isomorphism.
\end{proof}

Let us recall next that the double complex $(C^\bullet (p^{-1} \mathfrak{U}; \Omega^\bullet),
\delta, d)$ given by
\[ C^k (p^{-1} \mathfrak{U}; \Omega^q) = \prod \Omega^q (p^{-1}
  U_{\alpha_0 \ldots \alpha_k}) \]
can be used to compute the cohomology of the total space $E$. The restriction map
\[ \overline{r}: \Omega^\bullet (E) \longrightarrow \prod_\alpha \Omega^\bullet
 (p^{-1} U_\alpha) = C^0 (p^{-1} \mathfrak{U}; \Omega^\bullet) \]
makes $C^\bullet (p^{-1} \mathfrak{U}; \Omega^\bullet)$ into an augmented
double complex. By the standard generalized Mayer-Vietoris sequence, 
\cite{botttu}, the rows of this augmented double complex are exact. From
Proposition \ref{prop.rowex}, we thus deduce:
\begin{prop} \label{prop.simp2}
The restriction map $\overline{r}: \Omega^\bullet (E) \rightarrow
C^0 (p^{-1}\mathfrak{U}; \Omega^\bullet)$ induces an isomorphism
\[ \overline{r}^\ast: H^\bullet (E) = H^\bullet (\Omega^\bullet (E)) \stackrel{\cong}{\longrightarrow}
H^\bullet (C^\bullet (p^{-1}\mathfrak{U}), D), \]
where $(C^\bullet (p^{-1}\mathfrak{U}), D)$ is the simple complex of
$(C^\bullet (p^{-1}\mathfrak{U}; \Omega^\bullet), \delta, d)$.
\end{prop}

Regarding $\real^n \times F$ as a trivial fiber bundle over $\real^n$ with projection $\pi_1$, the
multiplicatively structured complex $\Omega^\bullet_\ms (\real^n)$ is defined as
\[ \Omega^\bullet_\ms (\real^n) = \{ \omega \in \Omega^\bullet (\real^n \times F) ~|~
 \omega = \sum_j \pi^\ast_1 \eta_j \wedge \pi^\ast_2 \gamma_j,~
 \eta_j \in \Omega^\bullet (\real^n),~ \gamma_j \in \Omega^\bullet (F) \}. \]
Let $s: \real^{n-1} \hookrightarrow \real \times \real^{n-1} = \real^n$ be the standard inclusion
$s(u) = (0,u),$ $u\in \real^{n-1}.$ Let $q: \real^n = \real \times \real^{n-1} 
\rightarrow \real^{n-1}$
be the standard projection $q(t,u) =u,$ so that
$qs = \id_{\real^{n-1}}.$
Set
\[ S = s\times \id_F: \real^{n-1} \times F \hookrightarrow \real^n \times F,~ 
 Q = q\times \id_F: \real^{n} \times F \rightarrow \real^{n-1} \times F \]
so that
$QS = \id_{\real^{n-1} \times F}.$
The equations
\[ \pi_1 \circ S = s\circ \pi_1,~ \pi_2 \circ S = \pi_2,~
\pi_1 \circ Q = q\circ \pi_1,~ \pi_2 \circ Q = \pi_2 \]
hold. The induced map
$S^\ast: \Omega^\bullet (\real^n \times F) \rightarrow \Omega^\bullet (\real^{n-1} \times F)$
restricts to a map
\[ S^\ast: \Omega^\bullet_\ms (\real^n) \rightarrow \Omega^\bullet_\ms (\real^{n-1}), \]
since
$S^\ast (\pi^\ast_1 \eta \wedge \pi^\ast_2 \gamma) = S^\ast \pi^\ast_1 \eta 
\wedge S^\ast \pi^\ast_2 \gamma = \pi^\ast_1 (s^\ast \eta)\wedge \pi^\ast_2 \gamma,$
$s^\ast \eta \in \Omega^\bullet (\real^{n-1}),$ $\gamma \in \Omega^\bullet (F)$. The induced
map $Q^\ast: \Omega^\bullet (\real^{n-1} \times F) \rightarrow \Omega^\bullet (\real^n \times F)$
restricts to a map
\[ Q^\ast: \Omega^\bullet_\ms (\real^{n-1}) \rightarrow \Omega^\bullet_\ms (\real^{n}), \]
since
$Q^\ast (\pi^\ast_1 \eta \wedge \pi^\ast_2 \gamma) = Q^\ast \pi^\ast_1 \eta 
\wedge Q^\ast \pi^\ast_2 \gamma = \pi^\ast_1 (q^\ast \eta)\wedge \pi^\ast_2 \gamma,$
$q^\ast \eta \in \Omega^\bullet (\real^{n}),$ $\gamma \in \Omega^\bullet (F)$.
\begin{prop} \label{prop.rnfrnm1f}
The maps
\begin{equation} \label{equ.restrSQ}
\xymatrix{\Omega^\bullet_{\ms}(\real^n) & \Omega^\bullet_\ms (\real^{n-1}) 
\ar@<1ex>[l]^{Q^\ast}  \ar@<1ex>[l];[]^{S^\ast}
} \end{equation}
are chain homotopy inverses of each other and thus induce mutually inverse
isomorphisms
\[ \xymatrix{H^\bullet (\Omega^\bullet_{\ms}(\real^n)) & H^\bullet (\Omega^\bullet_\ms (\real^{n-1})) 
\ar@<1ex>[l]^{Q^\ast}  \ar@<1ex>[l];[]^{S^\ast}
} \]
on cohomology.
\end{prop}
\begin{proof}
We start out by defining a homotopy operator $K: \Omega^\bullet (\real^n \times F) \to
\Omega^{\bullet -1} (\real^n \times F)$ satisfying
\begin{equation} \label{equ.dkpluskd}
dK + Kd = \id - Q^\ast S^\ast. 
\end{equation}
Think of $\real^n \times F$ as $\real \times M,$ with $M = \real^{n-1} \times F$.
In this notation, $Q$ and $S$ are the canonical projections and inclusions
\[ \xymatrix{\real \times M & M.
\ar@<1ex>[l]^{S}  \ar@<1ex>[l];[]^{Q}
} \]
Let $(t, t_2, \ldots, t_n)$ be coordinates on $\real^n = \real \times \real^{n-1}$ and
let $y$ denote (local) coordinates on $F$. Then $x = (t_2, \ldots, t_n, y)$ are coordinates
on $M$. Every form on $\real \times M$ can be uniquely written as a linear combination
of forms that do not contain $dt$, that is, forms
$f(t,x) Q^\ast \alpha,$
where $\alpha \in \Omega^\bullet (M)$,
and forms that do contain $dt$, that is, forms
$f(t,x) dt \wedge Q^\ast \alpha.$
We define $K$ by $K(f(t,x) Q^\ast \alpha)=0$ and
\[ K(f(t,x) dt \wedge Q^\ast \alpha) = g(t,x)Q^\ast \alpha, 
\text{ with } 
 g(t,x) = \int_0^t f(\tau,x)d\tau. \]
Equation (\ref{equ.dkpluskd}) is verified by a standard calculation.
We shall show that $K$ restricts to a homotopy operator $K_\ms$:
\[ \xymatrix{
\Omega^\bullet (\real^n \times F) \ar[r]^K & \Omega^{\bullet -1} (\real^n \times F) \\
\Omega^\bullet_\ms (\real^n) \ar@{^{(}->}[u] \ar@{..>}[r]^{K_\ms} & 
  \Omega^{\bullet -1}_\ms (\real^n). \ar@{^{(}->}[u] 
} \]
We shall use the commutative diagrams
\[ \xymatrix{
\real^n \times F \ar[d]^Q \ar[r]^{\pi_1} & \real^n = \real \times \real^{n-1}
  \ar[d]^q \\
\real^{n-1} \times F \ar[r]_{\hat{\pi}_1} & \real^{n-1}
} \hspace{.6cm} \text{and} \hspace{.6cm}
 \xymatrix{
\real^n \times F \ar[rd]_{\pi_2} \ar[r]^Q & \real^{n-1} \times F \ar[d]^{\hat{\pi}_2} \\
& F. 
} \]
Any form in $\Omega^\bullet_\ms (\real^n)$ can be written as a sum of forms
$\omega = \pi^\ast_1 \eta \wedge \pi^\ast_2 \gamma.$ We have to demonstrate that
$K(\omega)$ again has this multiplicatively structured form.
The form $\eta \in \Omega^\bullet (\real^n)$
can be uniquely written as a linear combination
of forms that do not contain $dt$, that is, forms
$f(t, t_2,\ldots, t_n) q^\ast \eta_{n-1},$
where $\eta_{n-1} \in \Omega^\bullet (\real^{n-1})$,
and forms that do contain $dt$, that is, forms
$f(t, t_2,\ldots, t_n) dt \wedge q^\ast \eta_{n-1}.$
In the former case,
\begin{eqnarray*}
\omega & = & \pi^\ast_1 ( f(t, t_2,\ldots, t_n) q^\ast \eta_{n-1}) \wedge Q^\ast
   \hat{\pi}^\ast_2 \gamma 
 =  f(t, t_2, \ldots, t_n) (Q^\ast \hat{\pi}^\ast_1 \eta_{n-1})\wedge Q^\ast
  \hat{\pi}^\ast_2 \gamma \\
& = & f\cdot Q^\ast \alpha
\end{eqnarray*}
with $\alpha = \hat{\pi}^\ast_1 \eta_{n-1} \wedge \hat{\pi}^\ast_2 \gamma$.
Thus $K(\omega)=0$ in this case. In the case where $\eta$ contains $dt$,
\[
\omega =  \pi^\ast_1 ( f(t, t_2,\ldots, t_n) dt \wedge q^\ast \eta_{n-1}) \wedge Q^\ast
   \hat{\pi}^\ast_2 \gamma 
=   f(t, t_2, \ldots, t_n) dt \wedge Q^\ast (\hat{\pi}^\ast_1 \eta_{n-1} \wedge 
  \hat{\pi}^\ast_2 \gamma) 
\]
so that
\[
 K(\omega)  =  g(t, t_2, \ldots, t_n)\cdot Q^\ast (\hat{\pi}^\ast_1 \eta_{n-1}\wedge 
  \hat{\pi}^\ast_2 \gamma) 
 =  \pi^\ast_1 (g q^\ast \eta_{n-1}) \wedge \pi^\ast_2 \gamma, 
\]
which is multiplicatively structured. We have thus constructed a homotopy operator
$K_\ms:\Omega^\bullet_\ms (\real^n) \to 
  \Omega^{\bullet -1}_\ms (\real^n)$ satisfying equation 
(\ref{equ.dkpluskd}) for the restricted maps (\ref{equ.restrSQ}).
Since $S^\ast Q^\ast = \id,$
$S^\ast$ and $Q^\ast$ are thus chain homotopy inverse chain homotopy
equivalences through multiplicatively structured forms.
\end{proof}

Let $S_0: F = \{ 0 \} \times F \hookrightarrow \real^n \times F$ be the inclusion
at $0$. The equations
$\pi_1 \circ S_0 = c_0,$ $\pi_2 \circ S_0 = \id_F$
hold, where $c_0: F\rightarrow \real^n$ is the constant map $c_0 (y)=0$ for all
$y\in F$. Thus, if $\eta \in \Omega^\bullet (\real^n)$ and $\gamma \in
\Omega^\bullet (F),$ then
\[ S^\ast_0 (\pi^\ast_1 \eta \wedge \pi^\ast_2 \gamma) =
 c^\ast_0 \eta \wedge \gamma = \begin{cases}
 \eta (0)\gamma, & \text{ if } \deg \eta =0 \\
 0, & \text{ if } \deg \eta >0. \end{cases} \]
The inclusion $S_0$ induces a map
$S^\ast_0: \Omega^\bullet_\ms (\real^n) \longrightarrow \Omega^\bullet (F).$
The map $\pi^\ast_2: \Omega^\bullet (F) \rightarrow \Omega^\bullet
(\real^n \times F)$ restricts to a map
$\pi^\ast_2: \Omega^\bullet (F) \longrightarrow \Omega^\bullet_\ms
(\real^n),$
as
\[ \pi^\ast_2 \gamma = 1 \wedge \pi^\ast_2 \gamma = \pi^\ast_1 (1)\wedge
 \pi^\ast_2 \gamma. \]

\begin{prop} \label{prop.rntfharmf}
The maps
\[ \xymatrix{\Omega^\bullet_{\ms}(\real^n) & \Omega^\bullet (F) 
\ar@<1ex>[l]^{\pi^\ast_2}  \ar@<1ex>[l];[]^{S^\ast_0}
} \]
are chain homotopy inverses of each other and thus induce mutually inverse
isomorphisms
\[ \xymatrix{H^\bullet (\Omega^\bullet_{\ms}(\real^n)) & H^\bullet (F) 
\ar@<1ex>[l]^>>>>>{\pi^\ast_2}  \ar@<1ex>[l];[]^>>>>{S^\ast_0}
} \]
on cohomology.
\end{prop}
\begin{proof}
The statement holds for $n=0,$ since then $S_0: \{ 0 \} \times F \rightarrow
\real^0 \times F$ is the identity map, $\pi_2: \real^0 \times F \rightarrow F$
is the identity map, and $\Omega^\bullet_\ms (\real^0) = \Omega^\bullet (F).$
For positive $n$, we factor $S_0$ as
\[ F = \real^0 \times F \stackrel{S}{\hookrightarrow} 
\real^1 \times F \stackrel{S}{\hookrightarrow} \ldots
\stackrel{S}{\hookrightarrow} 
\real^n \times F \]
and $\pi_2$ as
\[ \real^n \times F \stackrel{Q}{\longrightarrow} 
\real^{n-1} \times F \stackrel{Q}{\longrightarrow} \ldots
\stackrel{Q}{\longrightarrow} 
\real^0 \times F=F. \]
The statement then follows from Proposition \ref{prop.rnfrnm1f} by an 
induction on $n$.
\end{proof}

\begin{prop} \label{prop.fhrnrntfiso}
The inclusion $\Omega^\bullet_\ms (\real^n)\subset \Omega^\bullet
(\real^n \times F)$ induces an isomorphism
\[ H^\bullet (\Omega^\bullet_\ms (\real^n)) \cong H^\bullet
(\real^n \times F) \]
on cohomology.
\end{prop}
\begin{proof}
The factorization
\[ \xymatrix{
\Omega^\bullet_\ms (\real^n) \ar@{^{(}->}[r] \ar[rd]_{S^\ast_0}
& \Omega^\bullet (\real^n \times F) \ar[d]_{S^\ast_0} \\
 & \Omega^\bullet (F)
} \]
induces the diagram
\[ \xymatrix{
H^\bullet (\Omega^\bullet_\ms (\real^n)) \ar[r] \ar[rd]_{S^\ast_0}
& H^\bullet (\real^n \times F) \ar[d]_{S^\ast_0} \\
 & H^\bullet (F)
} \]
on cohomology. The diagonal arrow is an isomorphism by
Proposition \ref{prop.rntfharmf}. The vertical arrow is an isomorphism
by the homotopy invariance (Poincar\'e Lemma) of de Rham cohomology.
Thus the horizontal arrow is an isomorphism as well.
\end{proof}

\begin{prop} \label{prop.lociso}
For any $U_{\alpha_0 \ldots \alpha_k},$ the inclusion
\[ \Omega^\bullet_\ms (U_{\alpha_0 \ldots \alpha_k}) \hookrightarrow
 \Omega^\bullet (p^{-1} U_{\alpha_0 \ldots \alpha_k}) \]
induces an isomorphism on cohomology (with respect to the de Rham differential $d$).
\end{prop}
\begin{proof}
Put $V = U_{\alpha_0 \ldots \alpha_k}$. Since $\mathfrak{U}$ is a good cover, there
exists a diffeomorphism
$\psi: V \stackrel{\cong}{\longrightarrow} \real^n.$
We obtain a commutative diagram
\[ \xymatrix{
p^{-1}(V) \ar[r]^{\phi_{\alpha_0}}_{\cong} \ar[d]_{p|} &
V \times F \ar[r]^{\psi \times \id_F}_{\cong} \ar[d]_{\pi_1} &
\real^n \times F \ar[d]_{\pi_1} \\
V \ar@{=}[r] & V \ar[r]^{\cong}_{\psi} & \real^n.
} \]
The induced isomorphism
\[ \xymatrix@C=40pt{ \Omega^\bullet (\real^n \times F) 
 \ar[r]^{\phi^\ast_{\alpha_0} \circ (\psi \times \id)^\ast}_{\cong} &
 \Omega^\bullet (p^{-1} (V)) } \]
restricts to a map
\[ \xymatrix@C=40pt{ \Omega^\bullet_\ms (\real^n) 
 \ar[r]^{\phi^\ast_{\alpha_0} \circ (\psi \times \id)^\ast} &
 \Omega^\bullet_\ms (V), } \]
as
\begin{eqnarray*}
\phi^\ast_{\alpha_0} (\psi \times \id)^\ast \sum \pi^\ast_1 \eta_j
\wedge \pi^\ast_2 \gamma_j & = &
\phi^\ast_{\alpha_0} \sum (\psi \times \id)^\ast \pi^\ast_1 \eta_j
\wedge (\psi \times \id)^\ast \pi^\ast_2 \gamma_j \\
& = & \phi^\ast_{\alpha_0} \sum \pi^\ast_1 (\psi^\ast \eta_j)
\wedge \pi^\ast_2 \gamma_j \in \Omega^\bullet_\ms (V).
\end{eqnarray*}
The restricted map is again an isomorphism, since an element
\[ \phi^\ast_{\alpha_0} \sum \pi^\ast_1 \eta_j
\wedge \pi^\ast_2 \gamma_j \in \Omega^\bullet_\ms (V), \]
$\eta_j \in \Omega^\bullet (V),$ $\gamma_j \in \Omega^\bullet (F),$
is the image
$\phi^\ast_{\alpha_0} (\psi \times \id)^\ast \sum \pi^\ast_1 
((\psi^{-1})^\ast \eta_j)
\wedge \pi^\ast_2 \gamma_j,$
with
\[ \sum \pi^\ast_1 ((\psi^{-1})^\ast \eta_j)
\wedge \pi^\ast_2 \gamma_j \in \Omega^\bullet_\ms (\real^n). \]
The commutative square
\[ \xymatrix@C=45pt{
\Omega^\bullet_\ms (\real^n) 
\ar[r]^{\phi^\ast_{\alpha_0} \circ (\psi \times \id)^\ast}_{\cong}
\ar@{^{(}->}[d] & \Omega^\bullet_\ms (V) \ar@{^{(}->}[d] \\
\Omega^\bullet (\real^n \times F) 
\ar[r]^{\phi^\ast_{\alpha_0} \circ (\psi \times \id)^\ast}_{\cong}
& \Omega^\bullet (p^{-1}V)
} \]
induces a commutative square
\[ \xymatrix{
H^\bullet (\Omega^\bullet_\ms (\real^n)) 
\ar[r]^{\cong}
\ar[d] & H^\bullet (\Omega^\bullet_\ms (V)) \ar[d] \\
H^\bullet (\real^n \times F) 
\ar[r]^{\cong} & H^\bullet (p^{-1}V)
} \]
on cohomology. By Proposition \ref{prop.fhrnrntfiso}, the left vertical
arrow is an isomorphism. Thus the right vertical arrow is an
isomorphism as well.
\end{proof}

Since $d$ and $\delta$ on $C^\bullet (\mathfrak{U}; \Omega^\bullet_\ms)$
were obtained by restricting $d$ and $\delta$ on $C^\bullet
(p^{-1} \mathfrak{U}; \Omega^\bullet),$ the natural inclusion
$C^\bullet (\mathfrak{U}; \Omega^\bullet_\ms) \hookrightarrow
  C^\bullet (p^{-1} \mathfrak{U}; \Omega^\bullet)$
is a morphism of double complexes.
\begin{thm} \label{thm.fhcomputescohtotspace}
The inclusion $\Omega^\bullet_\ms (B) \hookrightarrow \Omega^\bullet (E)$
induces an isomorphism
\[ H^\bullet (\Omega^\bullet_\ms (B)) \stackrel{\cong}{\longrightarrow}
H^\bullet (E) \]
on cohomology.
\end{thm}
\begin{proof}
By Proposition \ref{prop.lociso}, the morphism
$C^\bullet (\mathfrak{U}; \Omega^\bullet_\ms) \rightarrow
  C^\bullet (p^{-1} \mathfrak{U}; \Omega^\bullet)$ of double complexes
induces an isomorphism on vertical (i.e. $d$-) cohomology, since
\[ H^\bullet_d (C^k (\mathfrak{U}; \Omega^\bullet_\ms)) =
 H^\bullet_d (\prod \Omega^\bullet_\ms (U_{\alpha_0 \ldots \alpha_k})) =
 \prod H^\bullet (\Omega^\bullet_\ms (U_{\alpha_0 \ldots \alpha_k})) \]
and
\[ H^\bullet_d (C^k (p^{-1} \mathfrak{U}; \Omega^\bullet)) =
 H^\bullet_d (\prod \Omega^\bullet (p^{-1}U_{\alpha_0 \ldots \alpha_k})) =
 \prod H^\bullet (\Omega^\bullet (p^{-1}U_{\alpha_0 \ldots \alpha_k})). \]
Whenever a morphism of double complexes induces an isomorphism on
vertical ($d$-) cohomology, then it also induces an isomorphism of the
$D$-cohomology of the respective simple complexes. Thus 
$C^\bullet (\mathfrak{U}; \Omega^\bullet_\ms) \rightarrow
  C^\bullet (p^{-1} \mathfrak{U}; \Omega^\bullet)$
induces an isomorphism
$H^\bullet (C^\bullet_\ms (\mathfrak{U}),D) \stackrel{\cong}{\longrightarrow}
 H^\bullet (C^\bullet (p^{-1} \mathfrak{U}), D).$
Since the diagram 
\[ \xymatrix{
\Omega^\bullet_\ms (B) \ar[r]^r \ar@{^{(}->}[d] &
 C^0 (\mathfrak{U}; \Omega^\bullet_\ms) 
\ar@{^{(}->}[d] \\
\Omega^\bullet (E) \ar[r]^{\overline{r}} &
 C^0 (p^{-1} \mathfrak{U}; \Omega^\bullet)
} \]
commutes, we get a commutative diagram
\[ \xymatrix{
H^\bullet (\Omega^\bullet_\ms (B)) \ar[r]^{r^\ast} \ar[d] &
 H^\bullet (C^\bullet_\ms (\mathfrak{U}),D) 
\ar[d]_{\cong} \\
H^\bullet (E) \ar[r]^{\overline{r}^\ast} &
 H^\bullet (C^\bullet (p^{-1} \mathfrak{U}),D).
} \]
By Proposition \ref{prop.simp1}, $r^\ast$ is an isomorphism, while by
Proposition \ref{prop.simp2}, $\overline{r}^\ast$ is an isomorphism.
Consequently,
$H^\bullet (\Omega^\bullet_\ms (B)) \longrightarrow H^\bullet (E)$
is an isomorphism as well.
\end{proof}

\section{Truncation and Cotruncation Over a Point}
\label{sec.truncoverpoint}

Let $F$ be a closed, oriented, $m$-dimensional Riemannian manifold as in Section
\ref{sec.fibharmflat}. We shall use the Riemannian metric to define truncation
$\tau_{<k}$ and cotruncation $\tau_{\geq k}$ of the complex $\Omega^\bullet (F)$.
The bilinear form
\[ \begin{array}{rcl}
(\cdot, \cdot): \Omega^r (F) \times \Omega^r (F) & \longrightarrow & \real, \\
(\omega, \eta) & \mapsto & \int_F \omega \wedge *\eta,
\end{array} \]
where $*$ is the Hodge star, is symmetric and positive definite, thus defines an inner
product on $\Omega^\bullet (F)$. The Hodge star acts as an isometry with respect to
this inner product, $(*\omega, *\eta) = (\omega, \eta),$ and the codifferential
\[ d^\ast = (-1)^{m(r+1)+1} *d*: \Omega^r (F) \longrightarrow \Omega^{r-1} (F) \]
is the adjoint of the differential $d$, $(d\omega, \eta) = (\omega, d^\ast \eta).$
The classical Hodge decomposition theorem provides orthogonal splittings
\begin{eqnarray*}
\Omega^r (F) & = & \im d^\ast \oplus \Harm^r (F) \oplus \im d, \\
\ker d & = & \Harm^r (F) \oplus \im d, \\
\ker d^\ast & = & \im d^\ast \oplus \Harm^r (F),
\end{eqnarray*}
where $\Harm^r (F) = \ker d \cap \ker d^\ast$ are the closed and coclosed, i.e. harmonic,
forms on $F$. In particular,
\[ \Omega^r (F) = \im d^\ast \oplus \ker d = \ker d^\ast \oplus \im d. \]
Let $k$ be a nonnegative integer.
\begin{defn} \label{def.truncoverpoint}
The \emph{truncation} $\tau_{<k} \Omega^\bullet (F)$ of $\Omega^\bullet (F)$ is the
complex
\[ \tau_{<k} \Omega^\bullet (F) = \cdots \longrightarrow \Omega^{k-2}(F)
\longrightarrow \Omega^{k-1}(F) \stackrel{d^{k-1}}{\longrightarrow} \im d^{k-1}
 \longrightarrow 0 \longrightarrow 0 \longrightarrow \cdots, \]
where $\im d^{k-1} \subset \Omega^k (F)$ is placed in degree $k$.
\end{defn}
The inclusion $\tau_{<k} \Omega^\bullet (F) \subset \Omega^\bullet (F)$ is a 
morphism of complexes, since
\[ \xymatrix{
\Omega^k F \ar[r]^{d^k} & \Omega^{k+1} F \\
\im d^{k-1} \ar@{^{(}->}[u] \ar[r] & 0 \ar[u]
} \]
commutes. The induced map on cohomology, $H^r (\tau_{<k} \Omega^\bullet F)\to
H^r (F),$ is an isomorphism for $r<k,$ while $H^r (\tau_{<k} \Omega^\bullet F)=0$ for
$r\geq k$. Using the orthogonal projection
\[ \proj: \Omega^k (F) = \ker d^\ast \oplus \im d \twoheadrightarrow \im d, \]
we define a surjective morphism of complexes
\[ \xymatrix@C=12pt{
\Omega^\bullet (F) =\cdots \ar[d]_{\proj} \ar[r] & \Omega^{k-2} (F) \ar@{=}[d] \ar[r] &
\Omega^{k-1}(F) \ar@{=}[d] \ar[r]^{d^{k-1}} & \Omega^k (F) \ar[d]_{\proj} \ar[r] &
\Omega^{k+1}(F) \ar[d] \ar[r] & \cdots \\
\tau_{<k} \Omega^\bullet (F)  =\cdots \ar[r] & 
\Omega^{k-2} (F) \ar[r] &
\Omega^{k-1}(F) \ar[r]^{d^{k-1}} & \im d^{k-1} \ar[r] & 0 \ar[r] &  \cdots.
} \]
(Note that $\proj \circ d^{k-1} = d^{k-1}.$) The composition
\[ \tau_{<k} \Omega^\bullet (F) \hookrightarrow \Omega^\bullet (F)
 \stackrel{\proj}{\twoheadrightarrow} \tau_{<k} \Omega^\bullet (F) \]
is the identity. Taking cohomology, this implies in particular that
$\proj^\ast: H^r (F) \to H^r (\tau_{<k} \Omega^\bullet F)$ is an isomorphism
for $r<k$. We move on to cotruncation.
\begin{defn} \label{def.cotruncoverpoint}
The \emph{cotruncation} $\tau_{\geq k} \Omega^\bullet (F)$ of $\Omega^\bullet (F)$ is the
complex
\[ \tau_{\geq k} \Omega^\bullet (F) = \cdots \longrightarrow 0
\longrightarrow 0 \longrightarrow \ker d^\ast
 \stackrel{d^k|}{\longrightarrow} \Omega^{k+1}(F) 
  \stackrel{d^{k+1}}{\longrightarrow} \Omega^{k+2}(F) \longrightarrow \cdots, \]
where $\ker d^\ast \subset \Omega^k (F)$ is placed in degree $k$.
\end{defn}
The inclusion $\tau_{\geq k} \Omega^\bullet (F) \subset \Omega^\bullet (F)$ is a
morphism of complexes. By construction, $H^r (\tau_{\geq k} \Omega^\bullet F)=0$
for $r<k$. There are several ways to see that 
$\tau_{\geq k} \Omega^\bullet (F) \hookrightarrow \Omega^\bullet (F)$
induces an isomorphism $H^r (\tau_{\geq k} \Omega^\bullet F)
\stackrel{\cong}{\longrightarrow} H^r (F)$ in the range $r\geq k$. One way is to
compare $\tau_{\geq k} \Omega^\bullet (F)$ to the standard cotruncation
\[ \widetilde{\tau}_{\geq k} \Omega^\bullet (F) = \cdots \longrightarrow 0
\longrightarrow 0 \longrightarrow \cok d^{k-1}
 \stackrel{d^k}{\longrightarrow} \Omega^{k+1}(F) 
  \stackrel{d^{k+1}}{\longrightarrow} \Omega^{k+2}(F) \longrightarrow \cdots, \]
for which the canonical morphism $\Omega^\bullet (F)\to 
\widetilde{\tau}_{\geq k} \Omega^\bullet (F)$ induces an isomorphism
$H^r (F)\to H^r (\widetilde{\tau}_{\geq k} \Omega^\bullet F)$ when $r\geq k$.
The inclusion $\ker d^\ast \subset \Omega^k F$ induces an isomorphism
\[ \ker d^\ast \stackrel{\cong}{\longrightarrow}
\frac{\ker d^\ast \oplus \im d}{\im d} = \frac{\Omega^k F}{\im d} = \cok d^{k-1}, \]
which extends to an isomorphism of complexes
\[ \xymatrix@C=12pt{
\tau_{\geq k} \Omega^\bullet F =\cdots \ar[d]_{\cong} \ar[r] & 0 \ar@{=}[d] \ar[r] &
\ker d^\ast \ar[d]_{\cong} \ar[r]^{d^k} & \Omega^{k+1} (F) \ar@{=}[d] \ar[r] &
\Omega^{k+2}(F) \ar@{=}[d] \ar[r] & \cdots \\
\widetilde{\tau}_{\geq k} \Omega^\bullet F =\cdots \ar[r] & 0  \ar[r] &
\cok d^{k-1}  \ar[r]^{d^k} & \Omega^{k+1} (F) \ar[r] &
\Omega^{k+2}(F) \ar[r] & \cdots. 
} \]
The commutativity of
\[ \xymatrix{
\tau_{\geq k} \Omega^\bullet F \ar[rr]^{\cong} 
\ar@{^{(}->}[rd] & & \widetilde{\tau}_{\geq k} \Omega^\bullet F \\
& \Omega^\bullet F \ar[ru]&
} \]
shows that $\tau_{\geq k} \Omega^\bullet F  \hookrightarrow \Omega^\bullet F$
is a cohomology isomorphism in degrees $r\geq k$. Alternatively, one observes that
\[ H^k (\tau_{\geq k} \Omega^\bullet F) = \ker d \cap \ker d^\ast =
  \Harm^k (F) \cong H^k (F) \]
and
\[ H^{k+1} (\tau_{\geq k} \Omega^\bullet F) =
\frac{\ker d^{k+1}}{d^k (\ker d^\ast)} =
\frac{\ker d^{k+1}}{d^k (\ker d^\ast \oplus \im d^{k-1})} =
\frac{\ker d^{k+1}}{\im d^k} = H^{k+1} (F). \]
The kernel of $\proj: \Omega^\bullet (F) \twoheadrightarrow \tau_{<k} \Omega^\bullet F$
is precisely $\tau_{\geq k} \Omega^\bullet (F).$ Thus there is an exact sequence
\begin{equation} \label{equ.tgeqkomftlek}
0 \to \tau_{\geq k} \Omega^\bullet F \longrightarrow \Omega^\bullet F
\longrightarrow \tau_{<k} \Omega^\bullet F \to 0. 
\end{equation}
(The associated long exact cohomology sequence gives a third way to see that
$\tau_{\geq k} \Omega^\bullet F \hookrightarrow \Omega^\bullet F$
is a cohomology isomorphism in degrees $r\geq k$.)

A key advantage of cotruncation over truncation is that $\tau_{\geq k} \Omega^\bullet F$
is a subalgebra of $\Omega^\bullet F$, whereas $\tau_{<k} \Omega^\bullet F$ is not.
This property of cotruncation will entail that the cohomology theory $HI^\bullet_{\bar{p}} (X)$
has a $\bar{p}$-internal cup product for all $\bar{p}$, while intersection cohomology does not.
\begin{prop} \label{prop.cotruncsubdga}
The complex $\tau_{\geq k} \Omega^\bullet F$ is a sub-DGA of $(\Omega^\bullet (F), d, \wedge)$.
\end{prop}
\begin{proof}
It remains to be shown that if $\omega, \eta \in \tau_{\geq k} \Omega^\bullet F$, then
$\omega \wedge \eta \in \tau_{\geq k} \Omega^\bullet F$. Let $p\geq 0$ be the degree of
$\omega$ and $q\geq 0$ the degree of $\eta$. If $p+q>k,$ then
$(\tau_{\geq k} \Omega^\bullet F)^{p+q} = \Omega^{p+q}(F)$ and there is nothing
to prove. If $p+q<k,$ then both $p$ and $q$ are less than k. In this case,
$(\tau_{\geq k} \Omega^\bullet F)^p = 0 = (\tau_{\geq k} \Omega^\bullet F)^q$ and
$\omega \wedge \eta = 0 \in \tau_{\geq k} \Omega^\bullet F$. Suppose
$p+q=k$. If one of $p,q$ is less than $k$, then $\omega \wedge \eta =
0\wedge \eta =0$ or $\omega \wedge \eta = \omega \wedge 0=0$ and the assertion
follows as before. If $p,q\geq k,$ then $k=p+q\geq 2k$ implies $k=0=p=q$.
But for $k=0$, $d^\ast =0: \Omega^0 F \to \Omega^{-1}F=0$ so that
$\ker d^\ast = \Omega^0 F.$ Thus for functions $\omega, \eta \in \Omega^0 F,$
we have $\omega \wedge \eta \in \Omega^0 (F) = \ker d^\ast =
(\tau_{\geq k} \Omega^\bullet F)^{p+q}$.
\end{proof}

\begin{prop} \label{prop.indepriemmetric}
The isomorphism type of $\tau_{\geq k} \Omega^\bullet F$ in the category of cochain complexes
is independent of the Riemannian metric on $F$.
\end{prop}
\begin{proof}
Let $g$ and $g'$ be two Riemannian metrics on $F$, determining codifferentials
$d^\ast_g, d^\ast_{g'}$, harmonic forms $\Harm^\bullet_g (F), \Harm^\bullet_{g'} (F),$
and cotruncations $\tau^g_{\geq k} \Omega^\bullet F, \tau^{g'}_{\geq k} \Omega^\bullet F.$
We observe first that
$D:= d^k (\ker d^\ast_g) = d^k (\ker d^\ast_{g'}),$
as follows from
\begin{eqnarray*}
d^k (\ker d^\ast_g) & = & d^k  (\im d^{k-1} \oplus \ker d^\ast_g) = d^k (\Omega^k F) \\
& = & d^k  (\im d^{k-1} \oplus \ker d^\ast_{g'}) = d^k  (\ker d^\ast_{g'}).
\end{eqnarray*}
Furthermore, as harmonic forms are closed,
\begin{eqnarray*}
d^k (\im d^\ast_g) & = & d^k (\im d^\ast_g \oplus \Harm^k_g (F)) =
 d^k (\ker d^\ast_g) \\
& = & d^k (\ker d^\ast_{g'}) = d^k (\im d^\ast_{g'} \oplus \Harm^k_{g'} (F)) 
 =  d^k (\im d^\ast_{g'}).
\end{eqnarray*}
Let
\[ d_g: \im d^\ast_g \longrightarrow D,~
  d_{g'}: \im d^\ast_{g'} \longrightarrow D \]
be the restrictions of $d^k: \Omega^k F \to \Omega^{k+1} F$ to
$\im d^\ast_g$ and $\im d^\ast_{g'}$, respectively. By the above observations,
$d_g$ and $d_{g'}$ are surjective. Since the decomposition
$\Omega^k F = \im d^\ast_g \oplus \ker d^k$ is direct, $d_g$ and $d_{g'}$
are injective, thus both isomorphisms. Since $F$ is closed, the inclusions
$\Harm^\bullet_g (F), \Harm^\bullet_{g'} (F) \subset \Omega^\bullet (F)$
induce isomorphisms
\[ h_g: \Harm^k_g (F) \stackrel{\cong}{\longrightarrow} H^k (F),~
  h_{g'}: \Harm^k_{g'} (F) \stackrel{\cong}{\longrightarrow} H^k (F). \]
Define an isomorphism $\kappa: \ker d^\ast_g \longrightarrow \ker d^\ast_{g'}$ by
\[ \xymatrix@C=50pt{
\kappa: \ker d^\ast_g = \im d^\ast_g \oplus \Harm^k_g (F)
\ar[r]^{d^{-1}_{g'} d_g \oplus h^{-1}_{g'} h_g} &
\im d^\ast_{g'} \oplus \Harm^k_{g'} (F) = \ker d^\ast_{g'}.
} \]
For $\alpha \in \im d^\ast_g,$ $\beta \in \Harm^k_g (F),$ we have
\[
d^k \kappa (\alpha + \beta) = 
d^k d^{-1}_{g'} d_g (\alpha) + d^k h^{-1}_{g'} h_g (\beta) 
=  d_g (\alpha) = d^k (\alpha + \beta),
\]
since harmonic forms are closed,
which verifies that
\[ \xymatrix{
\ker d^\ast_g \ar[r]^{d^k} \ar[d]_{\kappa}^{\cong} &
\Omega^{k+1} F \ar@{=}[d] \\
\ker d^\ast_{g'} \ar[r]^{d^k} & \Omega^{k+1} F
} \]
commutes. This square can be embedded in an isomorphism of complexes
\[ \xymatrix{
\tau^g_{\geq k} \Omega^\bullet F = \cdots \ar[r] \ar[d]_{\cong} &
0 \ar[r] \ar@{=}[d] & \ker d^\ast_g \ar[r] \ar[d]_{\kappa}^{\cong} &
\Omega^{k+1} F \ar[r] \ar@{=}[d] & \Omega^{k+2}F \ar[r] \ar@{=}[d]
& \cdots \\
\tau^{g'}_{\geq k} \Omega^\bullet F  = \cdots \ar[r] &
0 \ar[r]  & \ker d^\ast_{g'} \ar[r]  &
\Omega^{k+1} F \ar[r]  & \Omega^{k+2}F \ar[r] & \cdots.
} \]
\end{proof}

\begin{lemma} \label{lem.isometrypreservcotrunc}
Let $f:F\to F$ be a smooth self-map. \\
\noindent (1) $f$ induces an endomorphism $f^\ast$ of $\tau_{<k} \Omega^\bullet F$. \\
\noindent (2) If $f$ is an isometry, then $f$ induces an automorphism $f^\ast$
 of $\tau_{\geq k}  \Omega^\bullet F$.
\end{lemma}
\begin{proof}
(1) Since $f^\ast: \Omega^\bullet F \to \Omega^\bullet F$ commutes with $d$,
$f^\ast$ restricts to a map $f^\ast|: \im d^{k-1} \to \im d^{k-1}.$ 

(2) If $f$ is an isometry, then it preserves the orthogonal splitting
$\Omega^k F = \im d^{k-1} \oplus \ker d^\ast$: For an isometry, one has
$f^\ast \circ * = \epsilon \cdot *\circ f^\ast$ with $\epsilon =1$ if $f$ is
orientation preserving and $\epsilon = -1$ if $f$ is orientation reversing. Thus
\begin{eqnarray*}
d^\ast \circ f^\ast & = & (-1)^{m(k+1)+1} *d*f^\ast =
 (-1)^{m(k+1)+1} \epsilon \cdot *d f^\ast *  \\
 & = & (-1)^{m(k+1)+1} \epsilon \cdot * f^\ast d* =
 (-1)^{m(k+1)+1} \epsilon^2 \cdot f^\ast *d* \\
& = & f^\ast \circ d^\ast,
\end{eqnarray*}
which implies $f^\ast (\ker d^\ast) \subset \ker d^\ast.$
The preservation of $\im d^{k-1}$ was discussed in (1).
The restriction $f^\ast|: \ker d^\ast \to \ker d^\ast$ continues to be injective, and
is also onto: Given $\omega \in \ker d^\ast,$ there exist $\alpha \in \im d$,
$\beta \in \ker d^\ast$ such that $f^\ast (\alpha + \beta)=\omega,$
since $f^\ast: \Omega^k F \to \Omega^k F$ is onto. Then $f^\ast \alpha =
\omega - f^\ast \beta \in \ker d^\ast$ and $f^\ast \alpha \in \im d$ so that
$f^\ast \alpha \in \ker d^\ast \cap \im d = 0.$ Therefore, $f^\ast \beta = \omega$
and $f^\ast|: \ker d^\ast \to \ker d^\ast$ is surjective.
\end{proof}

\section{Fiberwise Truncation and Poincar\'e Duality}
\label{sec.fibtruncandpd}

\subsection{Local Fiberwise Truncation and Cotruncation}
\label{ssec.locfibtrunc}

Let $F$ be a closed, oriented, $m$-dimensional Riemannian manifold as
in Section \ref{sec.fibharmflat}. Regarding $\real^n \times F$ as a trivial
fiber bundle over $\real^n$ with projection $\pi_1$ and fiber $F$, a subcomplex
$\oms^\bullet (\real^n)\subset \Omega^\bullet (\real^n \times F)$ of
multiplicatively structured forms was defined in Section \ref{sec.fibharmflat} as
\[ \oms^\bullet (\real^n) = \{ \omega \in \Omega^\bullet (\real^n \times F) ~|~
  \omega = \sum_j \pi^\ast_1 \eta_j \wedge \pi^\ast_2 \gamma_j,~
  \eta_j \in \Omega^\bullet (\real^n),~ \gamma_j \in \Omega^\bullet (F) \}. \]
We shall here define the \emph{fiberwise truncation}
$\ft_{<k} \oms^\bullet (\real^n) \subset \oms^\bullet (\real^n)$ and the 
\emph{fiberwise cotruncation} $\ft_{\geq k} \oms^\bullet (\real^n) \subset
\oms (\real^n),$ depending on an integer $k$. Analogous concepts for
forms with compact supports will be introduced as well.
In Section \ref{sec.truncoverpoint}, a truncation $\tau_{<k} \Omega^\bullet (F)$
and a cotruncation $\tau_{\geq k} \Omega^\bullet (F)$ were defined using the
Riemannian metric on $F$.
Define
\begin{eqnarray*}
\ft_{<k} \oms^\bullet (\real^n) & = &\{ \omega \in \Omega^\bullet (\real^n \times F) ~|~
  \omega = \sum_j \pi^\ast_1 \eta_j \wedge \pi^\ast_2 \gamma_j, \\
& & \hspace{2cm} \eta_j \in \Omega^\bullet (\real^n),~ 
\gamma_j \in \tau_{<k} \Omega^\bullet (F) \}. 
\end{eqnarray*}
The Leibniz rule
\begin{equation} \label{equ.prodrule}
d(\pi^\ast_1 \eta \wedge \pi^\ast_2 \gamma) = \pi^\ast_1 (d\eta)\wedge \pi^\ast_2
 \gamma \pm \pi^\ast_1 \eta \wedge \pi^\ast_2 (d\gamma)
\end{equation}
shows that $\ft_{<k} \oms^\bullet (\real^n)$ is a
subcomplex of $\oms^\bullet (\real^n).$ Define
\begin{eqnarray*}
\ft_{\geq k} \oms^\bullet (\real^n) & = & \{ \omega \in \Omega^\bullet (\real^n \times F) ~|~
  \omega = \sum_j \pi^\ast_1 \eta_j \wedge \pi^\ast_2 \gamma_j, \\
& & \hspace{2cm} \eta_j \in \Omega^\bullet (\real^n),~ 
\gamma_j \in \tau_{\geq k} \Omega^\bullet (F) \}. 
\end{eqnarray*}
Again, this is a subcomplex of $\oms^\bullet (\real^n)$. Similar complexes can be defined
using compact supports. We define the complex $\omsc^\bullet (\real^n)$ of
\emph{multiplicatively structured forms with compact supports} on $\real^n \times F$ to be
\[ \omsc^\bullet (\real^n) = \{ \omega \in \Omega^\bullet (\real^n \times F) ~|~
  \omega = \sum_j \pi^\ast_1 \eta_j \wedge \pi^\ast_2 \gamma_j,~
  \eta_j \in \Omega^\bullet_c (\real^n),~ \gamma_j \in \Omega^\bullet (F) \}. \]
Since $d\eta$ has compact support if $\eta$ does, formula (\ref{equ.prodrule})
implies that $\omsc^\bullet (\real^n)$ is a complex. It is in fact a subcomplex of
$\Omega^\bullet_c (\real^n \times F)$, as $\pi^\ast_1 \eta \wedge \pi^\ast_2 \gamma$
has compact support if $\eta$ has compact support in $\real^n$. As above, fiberwise
truncations and cotruncations
\[ \ft_{<k} \omsc^\bullet (\real^n) \subset \omsc^\bullet (\real^n) \supset
 \ft_{\geq k} \omsc^\bullet (\real^n) \]
are defined by requiring the $\gamma_j$ to lie in $\tau_{<k} \Omega^\bullet (F)$ and
$\tau_{\geq k} \Omega^\bullet (F)$, respectively.

\subsection{Poincar\'e Lemmas for Fiberwise Truncations}
\label{ssec.poinclemmasfibtrunc}

Let
\[ s: \real^{n-1} \hookrightarrow \real^n,~
   S: \real^{n-1} \times F \hookrightarrow \real^n \times F,~
   q: \real^n \longrightarrow \real^{n-1},~
   Q: \real^n \times F \longrightarrow \real^{n-1} \times F  \]
be the standard inclusion and projection maps used in Section \ref{sec.fibharmflat}.
The formula
$S^\ast (\pi^\ast_1 \eta \wedge \pi^\ast_2 \gamma) =
  \pi^\ast_1 (s^\ast \eta)\wedge \pi^\ast_2 \gamma,$
$\gamma \in \tau_{<k} \Omega^\bullet (F),$ shows that
$S^\ast: \oms^\bullet (\real^n)\rightarrow \oms^\bullet (\real^{n-1})$ restricts
to a map
\[ S^\ast: \ft_{<k} \oms^\bullet (\real^n)\longrightarrow 
   \ft_{<k} \oms^\bullet (\real^{n-1}). \]
The formula
$Q^\ast (\pi^\ast_1 \eta \wedge \pi^\ast_2 \gamma) =
  \pi^\ast_1 (q^\ast \eta)\wedge \pi^\ast_2 \gamma,$
shows that
$Q^\ast: \oms^\bullet (\real^{n-1})\rightarrow \oms^\bullet (\real^n)$ restricts
to a map
\[ Q^\ast: \ft_{<k} \oms^\bullet (\real^{n-1})\longrightarrow 
   \ft_{<k} \oms^\bullet (\real^n). \]

\begin{lemma} \label{lem.821}
The maps
\[ \xymatrix{\ft_{<k} \oms^\bullet (\real^n) & \ft_{<k} \oms^\bullet (\real^{n-1}) 
\ar@<1ex>[l]^{Q^\ast}  \ar@<1ex>[l];[]^{S^\ast}
} \]
are chain homotopy inverses of each other and thus induce mutually inverse
isomorphisms
\[ \xymatrix{H^\bullet (\ft_{<k} \oms^\bullet (\real^n)) & 
  H^\bullet (\ft_{<k} \oms^\bullet (\real^{n-1})) 
\ar@<1ex>[l]^{Q^\ast}  \ar@<1ex>[l];[]^{S^\ast}
} \]
on cohomology.
\end{lemma}
\begin{proof}
Let $K_\ms: \oms^\bullet (\real^n) \rightarrow \oms^{\bullet -1} (\real^n)$ be the 
homotopy operator defined in the proof of Proposition \ref{prop.rnfrnm1f}.
In that proof, we have seen that $K_\ms$ applied to a form
$\omega = \pi^\ast_1 \eta \wedge \pi^\ast_2 \gamma$ yields a result that can be
written as $\pi^\ast_1 \eta' \wedge \pi^\ast_2 \gamma$ for some $\eta'$.
Thus $K_\ms$ does not transform $\gamma$ and if $\gamma \in \tau_{<k} 
\Omega^\bullet F,$ then $\pi^\ast_1 \eta' \wedge \pi^\ast_2 \gamma =
K_\ms (\omega)$ again lies in $\ft_{<k} \Omega^\bullet_\ms (\real^n)$.
Thus $K_\ms$ restricts to a homotopy operator
\[ K_\ms: \ft_{<k} \oms^\bullet (\real^n) \longrightarrow 
    (\ft_{<k} \oms^\bullet (\real^n))^{\bullet -1} \]
satisfying
$K_\ms d + dK_\ms = \id - Q^\ast S^\ast.$
Thus $Q^\ast S^\ast$ is chain homotopic to the identity on $\ft_{<k} \oms^\bullet
(\real^n)$. Since $S^\ast Q^\ast = \id,$ $S^\ast$ and $Q^\ast$ are thus chain
homotopy inverse chain homotopy equivalences through fiberwise truncated, 
multiplicatively structured forms.
\end{proof}
As in Section \ref{sec.fibharmflat}, let $S_0: F = \{ 0 \} \times F \hookrightarrow
\real^n \times F$ be the inclusion at $0$. If $\gamma \in \tau_{<k} \Omega^\bullet (F),$
then
\[ S^\ast_0 (\pi^\ast_1 \eta \wedge \pi^\ast_2 \gamma) = \begin{cases}
 \eta (0)\gamma, & \text{ if } \deg \eta =0 \\ 0, & \text{ if } \deg \eta >0 \end{cases} \]
lies in $\tau_{<k} \Omega^\bullet (F)$ for any $\eta \in \Omega^\bullet (\real^n).$
Thus $S^\ast_0: \oms^\bullet (\real^n) \rightarrow \Omega^\bullet (F)$ restricts to a map
\[ S^\ast_0: \ft_{<k} \oms^\bullet (\real^n) \longrightarrow \tau_{<k} \Omega^\bullet (F). \]
The map $\pi^\ast_2: \Omega^\bullet (F) \rightarrow \oms^\bullet (\real^n)$ restricts
to a map
\[  \pi^\ast_2: \tau_{<k} \Omega^\bullet (F) \rightarrow \ft_{<k}\oms^\bullet (\real^n) \]
by the definition of $\ft_{<k} \oms^\bullet (\real^n).$

\begin{lemma}(Poincar\'e Lemma, truncation version.) \label{lem.822}
The maps
\[ \xymatrix{\ft_{<k} \oms^\bullet (\real^n) & \tau_{<k} \Omega^\bullet (F) 
\ar@<1ex>[l]^{\pi^\ast_2}  \ar@<1ex>[l];[]^{S^\ast_0}
} \]
are chain homotopy inverses of each other and thus induce mutually inverse
isomorphisms
\[ \xymatrix{H^r (\ft_{<k} \oms^\bullet (\real^n)) & 
  H^r (\tau_{<k} \Omega^\bullet (F)) \cong 
\mbox{$\begin{cases} H^r (F),& r<k \\ 0, & r\geq k \end{cases}$} 
\ar@<1ex>[l]^<<<{\pi^\ast_2}  \ar@<1ex>[l];[]^<<<<{S^\ast_0}
} \]
on cohomology.
\end{lemma}
\begin{proof}
The statement holds for $n=0,$ since then $S_0$ and $\pi_2$ are both the
identity map and $\ft_{<k} \oms^\bullet (\real^0)= \tau_{<k} \Omega^\bullet (F).$
For positive $n,$ the statement follows, as in the proof of Proposition
\ref{prop.rntfharmf}, from an induction on $n$, using Lemma \ref{lem.821}.
\end{proof}
An analogous argument, replacing $\tau_{<k} \Omega^\bullet (F)$ by
$\tau_{\geq k} \Omega^\bullet (F),$ proves a version for fiberwise cotruncation:
\begin{lemma}(Poincar\'e Lemma, cotruncation version.) \label{lem.822cotrunc}
The maps
\[ \xymatrix{\ft_{\geq k} \oms^\bullet (\real^n) & \tau_{\geq k} \Omega^\bullet (F) 
\ar@<1ex>[l]^{\pi^\ast_2}  \ar@<1ex>[l];[]^{S^\ast_0}
} \]
are chain homotopy inverses of each other and thus induce mutually inverse
isomorphisms
\[ \xymatrix{H^r (\ft_{\geq k} \oms^\bullet (\real^n)) & 
  H^r (\tau_{\geq k} \Omega^\bullet (F)) \cong 
\mbox{$\begin{cases} H^r (F),& r\geq k \\ 0, & r< k. \end{cases}$} 
\ar@<1ex>[l]^<<<{\pi^\ast_2}  \ar@<1ex>[l];[]^<<<<{S^\ast_0}
} \]
on cohomology.
\end{lemma}


In order to set up a Poincar\'e lemma for fiberwise cotruncation of multiplicatively
structured compactly supported forms, we need to discuss integration along the fiber.
Let $Y$ be a smooth manifold and $\pi_2: \real^k \times Y\to Y$ the
second-factor projection. Integration along the fiber $\real^k$ of $\pi_2$ is a map
$\pi_{2\ast}: \Omega^\bullet_c (\real^k \times Y)\to \Omega^{\bullet -k}_c (Y)$
of degree $-k,$ given as follows. Let $t = (t_1, \ldots, t_k)$ be the standard coordinates
on $\real^k$ and
let $dt$ denote the $k$-form $dt=dt_1 \wedge 
\cdots \wedge dt_k$. A compactly supported form on $\real^k \times Y$ is a linear
combination of two types of forms: those which do not contain $dt$ as a
factor and those which do.
The former can be written as
$f (t,y) dt_{i_1} \wedge \cdots \wedge
dt_{i_r} \wedge \pi^\ast_2 \gamma,$ $r<k,$ and the latter as 
$g(t,y) dt \wedge \pi^\ast_2 \gamma,$
where $\gamma \in \Omega^\bullet_c (Y)$,
$y$ is a (local) coordinate on $Y$, and $f$,$g$ have compact support. Define $\pi_{2\ast}$ by
\begin{eqnarray*}
\pi_{2\ast} (f (t,y) dt_{i_1} \wedge \cdots \wedge
dt_{i_r} \wedge \pi^\ast_2 \gamma) & = & 0 \hspace{2cm} (r<k), \\
\pi_{2\ast}(g(t,y) dt \wedge \pi^\ast_2 \gamma) & = &
(\int_{\real^k} g dt_1 \cdots dt_k)\cdot \gamma.
\end{eqnarray*}
This is a chain map $\pi_{2\ast}: \Omega^\bullet_c (\real^k \times Y)\to \Omega^{\bullet -k}_c (Y)$,
provided the shifted complex $\Omega^{\bullet -k}_c (Y)$ is given the differential
$d_{-k} = (-1)^k d$.
For $\omega \in \Omega^\bullet_c (\real^k \times Y),$ one has the \emph{projection formula}
\[ \pi_{2\ast} (\omega \wedge \pi^\ast_2 \gamma)= (\pi_{2\ast} \omega)\wedge \gamma. \]
In particular, for a multiplicatively structured form involving the pullback of 
$\eta \in \Omega^\bullet_c (\real^k),$ we obtain
\[ \pi_{2\ast} (\pi^\ast_1 \eta \wedge \pi^\ast_2 \gamma)=\pi_{2\ast} (\pi^\ast_1 \eta)\wedge
 \gamma. \]
Applying this concept to our $\pi_2: \real^n \times F\to F,$ we receive a map
$\pi_{2\ast}: \Omega^\bullet_c (\real^n \times F)\to \Omega^{\bullet -n}(F),$ and, by
restriction,
$\pi_{2\ast}: \omsc (\real^n) \to \Omega^{\bullet -n}(F).$
\begin{lemma} \label{lem.823}
For $\omega \in \omsc^r (\real^n)$ and $\gamma \in \Omega^{n+m-r} (F),$ the
integration formula
\[ \int_{\real^n \times F} \omega \wedge \pi^\ast_2 \gamma =
 \int_F (\pi_{2\ast} \omega) \wedge \gamma \]
holds.
\end{lemma}
Now suppose that $\gamma \in \tau_{\geq k} \Omega^\bullet (F)$ and $\deg \eta =n$,
so that $\pi^\ast_1 \eta \wedge \pi^\ast_2 \gamma$ lies in $\ft_{\geq k} \omsc^\bullet
(\real^n)$. Then
\[ \pi_{2\ast} (\pi^\ast_1 \eta \wedge \pi^\ast_2 \gamma) = \pm (\int_{\real^n} \eta)
\cdot \gamma \]
lies in $\tau_{\geq k} \Omega^\bullet (F)$ as well. Thus integration along the fiber
restricts to a map
\[ \pi_{2\ast}: \ft_{\geq k} \omsc^\bullet (\real^n) \longrightarrow
  (\tau_{\geq k} \Omega^\bullet (F))^{\bullet -n}. \]
Choose any compactly supported $1$-form $e_1 = \varepsilon (t)dt \in
\Omega^1_c (\real^1)$ with
\[ \int_{-\infty}^{+\infty} \varepsilon (t)dt =1. \]
Then
\[ e = e_1 \wedge e_1 \wedge \cdots \wedge e_1 =
\prod_{i=1}^n \varepsilon (t_i) dt_1 \wedge \cdots \wedge dt_n \]
is a compactly supported $n$-form on $\real^n$ with $\int_{\real^n} e =1.$
A chain map
\[ e_\ast: \Omega^{\bullet -n} (F) \longrightarrow \omsc^\bullet (\real^n) \]
is given by
$e_\ast (\gamma) = \pi^\ast_1 e \wedge \pi^\ast_2 \gamma,$
since
\[ de_\ast (\gamma) = d(\pi^\ast_1 e \wedge \pi^\ast_2 \gamma) =
\pi^\ast_1 (de) \wedge \pi^\ast_2 \gamma + (-1)^n \pi^\ast_1 e \wedge \pi^\ast_2 d\gamma =
  (-1)^n \pi^\ast_1 e \wedge \pi^\ast_2 d\gamma =
  e_\ast (d_{-n} \gamma). \]
By definition of $\ft_{\geq k} \omsc^\bullet (\real^n),$ $e_\ast$ restricts to a map
\[ e_\ast: (\tau_{\geq k} \Omega^\bullet (F))^{\bullet -n} \longrightarrow
 \ft_{\geq k} \omsc^\bullet (\real^n). \]

\begin{lemma} (Poincar\'e Lemma for Cotruncation with Compact Supports.) \label{lem.824}
The maps
\[ \xymatrix{\ft_{\geq k} \omsc^\bullet (\real^n) & 
  (\tau_{\geq k} \Omega^\bullet (F))^{\bullet -n} 
\ar@<1ex>[l]^{e_\ast}  \ar@<1ex>[l];[]^{\pi_{2\ast}}
} \]
are chain homotopy inverses of each other and thus induce mutually inverse
isomorphisms
\[ \xymatrix{H^r (\ft_{\geq k} \omsc^\bullet (\real^n)) & 
  H^r ((\tau_{\geq k} \Omega^\bullet (F))^{\bullet -n}) \cong 
\mbox{$\begin{cases} H^{r-n} (F),& r-n \geq k \\ 0, & r-n < k. \end{cases}$} 
\ar@<1ex>[l]^<<<{e_\ast}  \ar@<1ex>[l];[]^<<<<{\pi_{2\ast}}
} \]
on cohomology.
\end{lemma}
\begin{proof}
The plan is to factor $\pi_{2\ast}$ and $e_\ast$ by peeling off one $\real^1$-factor
at a time. Each map in the factorization will be shown to be a homotopy equivalence.
Let $M$ be the manifold $M = \real^{n-1} \times F$ so that
$\real^n \times F = \real^1 \times \real^{n-1} \times F = \real^1 \times M.$
The coordinate on the $\real^1$-factor is $t_1$, coordinates on the 
$\real^{n-1}$-factor will be $u = (t_2, \ldots, t_n)$ and coordinates on $F$
will be $y$. We shall also write $x = (u,y)$ for points in $M$. Let 
$\pi: \real^1 \times M \rightarrow M$ be the projection given by $\pi (t_1,x)=x.$ \\

\textbf{Step 1.} We shall show that integration along the fiber of $\pi,$ 
$\pi_\ast: \Omega^\bullet_c (\real^1 \times M)\to \Omega^{\bullet -1}_c (M)$,
restricts to the complex of fiberwise cotruncated multiplicatively structured forms.
Let $\pi_1^\ast \eta \wedge \pi^\ast_2 \gamma \in \ft_{\geq k} \omsc^\bullet (\real^n)$
be a multiplicatively structured form, $\eta \in \Omega^p_c (\real^n),$
$\gamma \in \tau_{\geq k} \Omega^\bullet (F).$ The $p$-form $\eta$ can be
uniquely decomposed as
\[ \eta = \sum_I f_I (t_1,u) du_I + \sum_J g_J (t_1,u) dt_1 \wedge du_J, \]
\[ du_I = dt_{i_1}\wedge \cdots \wedge dt_{i_p},~
   du_J = dt_{j_1}\wedge \cdots \wedge dt_{j_{p-1}}, \]
where $I$ ranges over all strictly increasing multi-indices
$2\leq i_1 < i_2 < \ldots < i_p \leq n$ and $J$ over
$2\leq j_1 < j_2 < \ldots < j_{p-1} \leq n.$ 
The functions $f_I$ and $g_J$ have compact support.
As the terms
$\pi^\ast_1 (f_I (t_1,u) du_I)\wedge \pi^\ast_2 \gamma$
do not contain $dt_1$, they are sent to $0$ by $\pi_\ast$. Let
\[ \real^{n-1} \stackrel{\widehat{\pi}_1}{\longleftarrow} \real^{n-1}\times F
\stackrel{\widehat{\pi}_2}{\longrightarrow} F \]
be the standard projections $\widehat{\pi}_1 (u,y)=u,$ $\widehat{\pi}_2 (u,y)=y,$
and set
\[ G_J (u) = \int_{-\infty}^{+\infty} g_J (t_1,u)dt_1. \]
The map $\pi_\ast$ sends the term
\[ \pi^\ast_1 (g_J (t_1,u) dt_1 \wedge du_J)\wedge \pi^\ast_2 \gamma =
 g_J (t_1,u)dt_1 \wedge \pi^\ast (\widehat{\pi}^\ast_1 du_J 
    \wedge \widehat{\pi}^\ast_2 \gamma) \]
to
\[ G_J (u)\cdot (\widehat{\pi}^\ast_1 du_J \wedge \widehat{\pi}^\ast_2 \gamma) =
  \widehat{\pi}^\ast_1 (G_J (u) du_J) \wedge \widehat{\pi}^\ast_2 \gamma, \]
which lies in $(\ft_{\geq k} \omsc^\bullet (\real^{n-1}))^{\bullet -1}$. Thus
$\pi_\ast$ restricts to a map
\[ \pi_\ast: \ft_{\geq k} \omsc^\bullet (\real^n) \longrightarrow
 (\ft_{\geq k} \omsc^\bullet (\real^{n-1}))^{\bullet -1}. \]

\textbf{Step 2.} We shall construct a candidate $e_{1\ast}$ for a homotopy inverse
for $\pi_\ast$ and show that it, too, restricts to the complex of fiberwise cotruncated
multiplicatively structured forms. We define a chain map
$e_{1\ast}: \Omega^{\bullet -1}_c (M) \longrightarrow \Omega^\bullet_c 
 (\real^1 \times M),$
that is,
\[ e_{1\ast}: \Omega^{\bullet -1}_c (\real^{n-1}\times F) \longrightarrow 
  \Omega^\bullet_c (\real^n \times F), \]
by
$e_{1\ast} (\omega) = e_1 \wedge \pi^\ast \omega.$
By construction,
$\pi_\ast \circ e_{1\ast} = \id.$
(Recall that $\int_{\real^1} e_1 =1$.) 
The equations
$\hat{\pi} \circ \pi_1 = \widehat{\pi}_1 \circ \pi,~ \widehat{\pi}_2 \circ \pi = \pi_2$
hold, where $\hat{\pi}: \real \times \real^{n-1} \rightarrow \real^{n-1}$ is
the standard projection $\hat{\pi} (t,u)=u.$ The image of a form
$\widehat{\pi}^\ast_1 \eta \wedge \widehat{\pi}^\ast_2 \gamma \in
(\ft_{\geq k} \omsc^\bullet (\real^{n-1}))^{\bullet -1},$
$\eta \in \Omega^\bullet_c (\real^{n-1}),$
$\gamma \in \tau_{\geq k} \Omega^\bullet (F),$ under $e_{1\ast}$ is
\begin{eqnarray*}
e_{1\ast} (\widehat{\pi}^\ast_1 \eta \wedge \widehat{\pi}^\ast_2 \gamma) & = &
  e_1 \wedge \pi^\ast (\widehat{\pi}^\ast_1 \eta \wedge \widehat{\pi}^\ast_2 \gamma) 
 =  e_1 \wedge \pi^\ast \widehat{\pi}^\ast_1 \eta \wedge  
   \pi^\ast \widehat{\pi}^\ast_2 \gamma \\
& = & e_1 \wedge \pi^\ast_1 \hat{\pi}^\ast \eta \wedge  
   \pi^\ast_2 \gamma 
 =  \pi^\ast_1 (e_1 \wedge \hat{\pi}^\ast \eta) \wedge  
   \pi^\ast_2 \gamma, 
\end{eqnarray*}
which lies in $\ft_{\geq k} \omsc^\bullet (\real^n)$. Thus $e_{1\ast}$ restricts to a map
\[ e_{1\ast}: (\ft_{\geq k} \omsc^\bullet (\real^{n-1}))^{\bullet -1} \longrightarrow
   \ft_{\geq k} \omsc^\bullet (\real^n). \]

\textbf{Step 3.} We shall show that $e_{1\ast} \pi_\ast$ is homotopic to the identity
by exhibiting a homotopy operator
$K: \ft_{\geq k} \omsc^\bullet (\real^n)  \longrightarrow 
   (\ft_{\geq k} \omsc^\bullet (\real^n))^{\bullet -1}$
such that
\begin{equation} \label{equ.htpyopkofhc}
\id - e_{1\ast} \pi_\ast = dK + Kd 
\end{equation}
on $\ft_{\geq k} \omsc^\bullet (\real^n).$ First, define
$K: \Omega^\bullet_c (\real^1 \times M) \longrightarrow
 \Omega^{\bullet -1}_c (\real^1 \times M),$
that is,
\[ K: \Omega^\bullet_c (\real^n \times F) \longrightarrow
 \Omega^{\bullet -1}_c (\real^n \times F), \]
by
\begin{eqnarray*}
K(f(t_1,x)\cdot \pi^\ast \mu) & = & 0, \\
K(g(t_1,x)dt_1 \wedge \pi^\ast \mu) & = &  
  (G(t_1,x) - E_1 (t_1)G(\infty, x)) \cdot \pi^\ast \mu
\end{eqnarray*}
where
\[ G(t_1,x) = \int_{-\infty}^{t_1} g(\tau, x) d\tau,~
  E_1 (t_1) = \int_{-\infty}^{t_1} e_1. \]
Equation (\ref{equ.htpyopkofhc}) holds on $\Omega^\bullet_c (\real^1 \times M)$.
Let $\pi^\ast_1 \eta \wedge \pi^\ast_2 \gamma \in \ft_{\geq k} \omsc^\bullet
(\real^n)$ be a multiplicatively structured form, $\eta \in \Omega^p_c (\real^n),$
$\gamma \in \tau_{\geq k} \Omega^\bullet (F).$ The basic form $\eta$ is again
decomposed as in Step 1. As the terms
$\pi^\ast_1 (f_I (t_1,u) du_I)\wedge \pi^\ast_2 \gamma$ do not contain $dt_1$,
they are sent to $0$ by $K$. With
\[ H_J (t_1,u) = G_J (t_1,u) - E_1 (t_1) G_J (\infty, u), \]
which has compact support,
$K$ maps the terms
\[ \pi^\ast_1 (g_J (t_1,u) dt_1 \wedge du_J)\wedge \pi^\ast_2 \gamma =
  g_J (t_1,u)dt_1 \wedge 
 \pi^\ast (\widehat{\pi}^\ast_1 du_J \wedge \widehat{\pi}^\ast_2 \gamma) \]
to
\[ H_J (t_1,u)\cdot \pi^\ast (\widehat{\pi}^\ast_1 du_J \wedge \widehat{\pi}^\ast_2 \gamma)
 = H_J (t_1,u)\cdot \pi^\ast_1 du_J \wedge \pi^\ast_2 \gamma
 = \pi^\ast_1 (H_J (t_1,u)du_J)\wedge \pi^\ast_2 \gamma, \]
which lie in $(\ft_{\geq k} \omsc^\bullet (\real^n))^{\bullet -1}.$ Consequently,
$K$ restricts to a map
\[ K: \ft_{\geq k} \omsc^\bullet (\real^n)  \longrightarrow 
   (\ft_{\geq k} \omsc^\bullet (\real^n))^{\bullet -1}. \]
By equation (\ref{equ.htpyopkofhc}), it is a homotopy operator between
\[ e_{1\ast} \pi_\ast: \ft_{\geq k} \omsc^\bullet (\real^n)  \longrightarrow 
   \ft_{\geq k} \omsc^\bullet (\real^n) \]
and the identity.\\

\textbf{Step 4.} By Step 3 and $\pi_\ast e_{1\ast} = \id,$ the maps
\[ \xymatrix{\ft_{\geq k} \omsc^\bullet (\real^n) & 
  (\ft_{\geq k} \omsc^\bullet (\real^{n-1}))^{\bullet -1} 
\ar@<1ex>[l]^{e_{1\ast}}  \ar@<1ex>[l];[]^{\pi_\ast}
} \]
are mutually chain homotopy inverse chain homotopy equivalences. As $n$ was
arbitrary, we may iterate the application of these maps and obtain homotopy
equivalences

\[ \xymatrix@R=14pt{
\ft_{\geq k} \omsc^\bullet (\real^n) \ar@<1ex>[d]^{\pi_\ast} & \\
(\ft_{\geq k} \omsc^\bullet (\real^{n-1}))^{\bullet -1} \ar@<1ex>[u]^{e_{1\ast}}
   \ar@<1ex>[d]^{\pi_\ast} & \\
(\ft_{\geq k} \omsc^\bullet (\real^{n-2}))^{\bullet -2} \ar@<1ex>[u]^{e_{1\ast}} 
    \ar@{..}[d] & \\
(\ft_{\geq k} \omsc^\bullet (\real^{1}))^{\bullet -n+1} 
   \ar@<1ex>[d]^{\pi_\ast} & \\
(\ft_{\geq k} \omsc^\bullet (\real^{0}))^{\bullet -n} \ar@<1ex>[u]^{e_{1\ast}} &
 = (\tau_{\geq k} \Omega^\bullet (F))^{\bullet -n}.
} \] 
Let $\pi^n_\ast$ denote this $n$-fold iteration of $\pi_\ast$ and $e^n_{1\ast}$
the $n$-fold iteration of $e_{1\ast}$. Since, as is readily checked, 
$\pi^n_\ast = \pi_{2\ast}$ and $e^n_{1\ast} = e_\ast,$ the lemma is proved.
\end{proof}

\subsection{Local Poincar\'e Duality for Truncated Structured Forms}
\label{ssec.locpdtfhforms}

The Poincar\'e Lemmas of the previous section, together with the
integration formula of Lemma \ref{lem.823} imply local Poincar\'e duality
between fiberwise truncated multiplicatively structured forms and fiberwise
cotruncated compactly supported multiplicatively structured forms, as we will
demonstrate in this section. Given complementary perversities
$\bar{p}$ and $\bar{q}$, and the dimension $m$ of $F$, we define
truncation values
\[ K = m- \bar{p}(m+1),~ K^\ast = m- \bar{q}(m+1). \]
The bilinear form
\[ \begin{array}{ccl}
\Omega^r (\real^n \times F)\times \Omega^{n+m-r}_c (\real^n \times
F) & \longrightarrow & \real \\
(\omega, \omega') & \mapsto & \int_{\real^n \times F} \omega \wedge \omega'
\end{array} \]
restricts to
$\int: \oms^r (\real^n)\times \omsc^{n+m-r} (\real^n) \longrightarrow  \real$
and further to
\begin{equation} \label{equ.localpairing}
\int: (\ft_{<K} \oms^\bullet (\real^n))^r \times 
 (\ft_{\geq K^\ast} \omsc^{\bullet} (\real^n))^{n+m-r} \longrightarrow  \real.
\end{equation}
Stokes' theorem implies:
\begin{lemma} \label{lem.831}
The bilinear forms (\ref{equ.localpairing}) induce bilinear forms
\[ \int: H^r (\ft_{<K} \oms^\bullet (\real^n)) \times 
 H^{n+m-r}(\ft_{\geq K^\ast} \omsc^{\bullet} (\real^n)) \longrightarrow  \real \]
on cohomology.
\end{lemma}

\begin{lemma} \label{lem.truncpd}
Integration induces a nondegenerate bilinear form
\[ H^r (\tau_{<K} \Omega^\bullet (F)) \times
  H^{m-r} (\tau_{\geq K^\ast} \Omega^\bullet (F))
 \longrightarrow \real. \]
\end{lemma}
\begin{proof}
If $r\geq K,$ then $H^r (\tau_{<K} \Omega^\bullet (F))=0.$
The inequality $r\geq K$ implies the inequality
$m-r < K^\ast$. Thus
$H^{m-r} (\tau_{\geq K^\ast} \Omega^\bullet (F))=0$ as well
and the lemma is proved for $r\geq K$. When $r<K,$ then
$H^r (\tau_{<K} \Omega^\bullet (F)) = H^r (F).$
The inequality $r<K$ implies $m-r \geq K^\ast$. Hence
$H^{m-r} (\tau_{\geq K^\ast} \Omega^\bullet (F))
 = H^{m-r} (F).$
Classical Poincar\'e duality for the closed, oriented $m$-manifold
$F$ asserts that the bilinear form
\[ \begin{array}{ccc}
H^r (F) \times H^{m-r} (F) & \longrightarrow &
 \real \\
([\omega], [\eta]) & \mapsto & \int_{F} \omega \wedge \eta
\end{array} \]
is nondegenerate.
\end{proof}

\begin{lemma} \label{lem.832} (Local Poincar\'e Duality.)
The bilinear form
\[ \int: H^r (\ft_{<K} \oms^\bullet (\real^n)) \times 
 H^{n+m-r}(\ft_{\geq K^\ast} \omsc^{\bullet} (\real^n)) \longrightarrow  \real \]
is nondegenerate.
\end{lemma}
\begin{proof}
By Lemma \ref{lem.truncpd}, the map
\[ \begin{array}{rcl}
H^r (\tau_{<K} \Omega^\bullet (F)) & \longrightarrow &
H^{m-r} (\tau_{\geq K^\ast} \Omega^\bullet (F))^\dagger, \\
\mbox{[} \omega \mbox{]} & \mapsto & \int_F - \wedge \omega, 
\end{array} \]
is an isomorphism. We have to show that the map
\[ \begin{array}{rcl}
H^r (\ft_{<K} \oms^\bullet (\real^n)) & \longrightarrow &
H^{n+m-r} (\ft_{\geq K^\ast} \omsc^\bullet (\real^n))^\dagger, \\
\mbox{[} \omega \mbox{]} & \mapsto & \int_{\real^n \times F} - \wedge \omega, 
\end{array} \]
is an isomorphism. By the Poincar\'e Lemma \ref{lem.822},
\[ \pi^\ast_2: H^r (\tau_{<K} \Omega^\bullet (F)) \longrightarrow
  H^r (\ft_{<K} \oms^\bullet (\real^n)) \]
is an isomorphism. According to the Poincar\'e lemma for cotruncation with compact supports,
Lemma \ref{lem.824},
\[ \pi_{2\ast}: H^{n+m-r} (\ft_{\geq K^\ast} \omsc (\real^n))
 \longrightarrow H^{m-r} (\tau_{\geq K^\ast} \Omega^\bullet (F)) \]
is an isomorphism. The desired conclusion will follow once we have
verified that the diagram
\[ \xymatrix{
H^r (\ft_{<K} \oms^\bullet (\real^n)) \ar[d]_{\int} &
 H^r (\tau_{<K} \Omega^\bullet (F)) 
 \ar[l]^{\pi^\ast_2}_{\cong} \ar[d]_{\int}^{\cong} \\
H^{n+m-r} (\ft_{\geq K^\ast} \omsc^\bullet (\real^n))^\dagger
& H^{m-r} (\tau_{\geq K^\ast} \Omega^\bullet (F))^\dagger
\ar[l]^<<<{\pi^\dagger_{2\ast}}_<<<<{\cong}
} \]
commutes. Commutativity means that for $\gamma \in \tau_{<K}
\Omega^\bullet (F)$ and $\omega \in \ft_{\geq K^\ast} \omsc^\bullet
(\real^n),$ the identity
\[ \int_{\real^n \times F} \omega \wedge \pi^\ast_2 \gamma =
 \int_F \pi_{2\ast} \omega \wedge \gamma \]
holds. This is precisely the integration formula of Lemma \ref{lem.823}.
\end{proof}

\subsection{Global Poincar\'e Duality for Truncated Structured Forms}
\label{ssec.globalpdtruncfibharm}

Let $F \rightarrow E \stackrel{p}{\rightarrow} B$ be a flat fiber bundle as in 
Section \ref{sec.fibharmflat}. The manifold $F$ is Riemannian and 
we now assume that the
structure group of the bundle are the isometries of $F$. The smooth,
compact base $B$ is covered by a finite good open cover
$\mathfrak{U} = \{ U_\alpha \}$ with respect to which the bundle
trivializes. The local trivializations are denoted by
$\phi_\alpha: p^{-1} (U_\alpha) \stackrel{\cong}{\longrightarrow}
U_\alpha \times F,$ as before. 
For an open subset $U\subset B$, we set
\[ \oms^\bullet (U) = \{ \omega \in \Omega^\bullet (p^{-1} U) ~|~
 \omega|_{p^{-1}(U\cap U_\alpha)} \text{ is $\alpha$-multiplicatively structured
for all } \alpha \}. \]
A compactly supported version $\omsc^\bullet (U)$ is obtained by
setting
\[ \omsc^\bullet (U) = \{ \omega \in \Omega^\bullet (p^{-1}U) ~|~
 \omega = \sum_\alpha \omega_\alpha,~ \operatorname{supp}
 (\omega_\alpha) \subset p^{-1}(U\cap U_\alpha), \]
\[ \hspace{4cm} \omega_\alpha = \phi^\ast_\alpha \sum_j \pi^\ast_1 \eta_j \wedge
 \pi^\ast_2 \gamma_j,~ \eta_j \in \Omega^\bullet_c (U\cap U_\alpha),~
 \gamma_j \in \Omega^\bullet (F) \}. \]
Note that this is consistent with our earlier definition of $\omsc^\bullet (\real^n)$
for $U = \real^n$. This complex is indeed a subcomplex of $\Omega^\bullet_c
(p^{-1} U)$, since $\supp (\sum \omega_\alpha)\subset \bigcup_\alpha 
\supp (\omega_\alpha),$ the finite union of compact sets is compact, and a closed
subset of a compact set is compact. For any integer $k$, a subcomplex
\[ \ft_{<k} \oms^\bullet (U) \subset \oms^\bullet (U) \]
of \emph{fiberwise truncated} multiplicatively structured forms on $p^{-1} (U)$ is given
by requiring, for all $\alpha$, every $\gamma_j$ to lie in $\tau_{<k} \Omega^\bullet
(F)$. This is well-defined by the transformation law of Lemma
\ref{lem.abinvariant} together with Lemma \ref{lem.isometrypreservcotrunc}(1).
A subcomplex
\[ \ft_{\geq k} \oms^\bullet (U) \subset \oms^\bullet (U) \]
of \emph{fiberwise cotruncated} multiplicatively structured forms on $p^{-1} (U)$ is given
by requiring, for all $\alpha$, every $\gamma_j$ to lie in $\tau_ {\geq k} \Omega^\bullet
(F)$. This is well-defined by the transformation law and
 Lemma \ref{lem.isometrypreservcotrunc}(2).
(At this point it is used that the transition functions of the bundle are isometries.) 
A subcomplex
\[ \ft_{\geq k} \omsc^\bullet (U) \subset \omsc^\bullet (U) \]
of \emph{fiberwise cotruncated} multiplicatively structured compactly supported forms on $p^{-1} (U)$ is given
by requiring, for all $\alpha$, every $\gamma_j$ to lie in $\tau_ {\geq k} \Omega^\bullet
(F)$. Again, this is well-defined. Let $K = m - \bar{p}(m+1),$
$K^\ast = m - \bar{q}(m+1)$ be the truncation values defined in Section
\ref{ssec.locpdtfhforms}. The bilinear form
\[ \begin{array}{rcl}
\Omega^r (p^{-1} U) \times \Omega^{n+m-r}_c (p^{-1} U) & \longrightarrow &
 \real \\ (\omega, \omega') & \mapsto & \int_{p^{-1}U} \omega \wedge \omega'
\end{array} \]
restricts to
$\int: \oms^r (U) \times \omsc^{n+m-r} (U) \longrightarrow \real$
and further to
\begin{equation} \label{equ.nablanabla}
\int: (\ft_{<K} \oms^\bullet (U))^r \times (\ft_{\geq K^\ast} \omsc^\bullet 
(U))^{n+m-r} \longrightarrow \real.
\end{equation}
Replacing $\real^n$ by $U$ and $\real^n \times F$ by $p^{-1} U$ in the proof
of Lemma \ref{lem.831}, we obtain a globalized version of that lemma:
\begin{lemma} \label{lem.842}
The bilinear forms (\ref{equ.nablanabla}) induce bilinear forms
\[ \int: H^r (\ft_{<K} \oms^\bullet (U)) \times 
 H^{n+m-r}(\ft_{\geq K^\ast} \omsc^{\bullet} (U)) \longrightarrow  \real \]
on cohomology.
\end{lemma}

\begin{lemma} \label{lem.bootstrap843} (Bootstrap.)
Let $U,V \subset B$ be open subsets such that
\begin{equation} \label{equ.pairingonw}
 \int: H^r (\ft_{<K} \oms^\bullet (W)) \times 
 H^{n+m-r}(\ft_{\geq K^\ast} \omsc^{\bullet} (W)) \longrightarrow  \real
\end{equation}
is nondegenerate for $W = U,V,U\cap V$. Then (\ref{equ.pairingonw}) is
nondegenerate for $W=U\cup V$.
\end{lemma}
\begin{proof}
We start out by showing that for any $k\in \intg$ the map
\[ \begin{array}{rcl}
\ft_{<k} \oms^\bullet (U) \oplus \ft_{<k} \oms^\bullet (V) & \longrightarrow &
  \ft_{<k} \oms^\bullet (U\cap V) \\
(\omega, \tau) & \mapsto & \tau|_{p^{-1}(U\cap V)} - \omega|_{p^{-1}(U\cap V)} 
\end{array} \]
is surjective. Let $\{ \rho_U, \rho_V \}$ be a partition of unity subordinate to
$\{ U,V \}$. Given $\omega$ in $\ft_{<k} \oms^\bullet (U\cap V),$ 
$p^\ast (\rho_V)\omega$ is a form on $U$ and $p^\ast (\rho_U)\omega$ is a 
form on $V$ such that the pair $(-p^\ast (\rho_V)\omega, p^\ast (\rho_U)\omega)$
maps to $\omega$. We have to check that $p^\ast (\rho_V)\omega \in
\ft_{<k} \oms^\bullet (U)$ and $p^\ast (\rho_U)\omega \in \ft_{<k} \oms^\bullet
(V)$. Since
\[ \omega|_{p^{-1} (U\cap V\cap U_\alpha)} = \phi^\ast_\alpha \sum_j
 \pi^\ast_1 \eta_j \wedge \pi^\ast_2 \gamma_j, \]
$\eta_j \in \Omega^\bullet (U\cap V\cap U_\alpha),$
$\gamma_j \in \tau_{<k} \Omega^\bullet (F),$ we have
\begin{eqnarray*}
(p^\ast (\rho_V)\omega)|_{p^{-1} (U\cap U_\alpha)} & = &
 p^\ast \rho_V \cdot \phi^\ast_\alpha \sum_j
 \pi^\ast_1 \eta_j \wedge \pi^\ast_2 \gamma_j 
  = 
 \phi^\ast_\alpha \pi^\ast_1 (\rho_V) \cdot \phi^\ast_\alpha \sum_j
 \pi^\ast_1 \eta_j \wedge \pi^\ast_2 \gamma_j \\
& = &
 \phi^\ast_\alpha \sum_j
 \pi^\ast_1 (\rho_V \cdot \eta_j) \wedge \pi^\ast_2 \gamma_j, 
\end{eqnarray*}
which implies that  $p^\ast (\rho_V)\omega \in
\ft_{<k} \oms^\bullet (U)$. The corresponding fact for $p^\ast (\rho_U)\omega$
follows from symmetry. Thus the difference map is surjective as claimed.

Let us proceed to demonstrate the exactness of the sequence
\begin{equation} \label{equ.exmvftlkofh}
0 \rightarrow \ft_{<k} \oms^\bullet (U\cup V) \longrightarrow
\ft_{<k} \oms^\bullet (U) \oplus \ft_{<k} \oms^\bullet (V) \longrightarrow
\ft_{<k} \oms^\bullet (U\cap V) \rightarrow 0
\end{equation}
at the middle group. Given $\omega \in \ft_{<k} \oms^\bullet (U)$ and
$\tau \in \ft_{<k} \oms^\bullet (V)$ such that
\[ \omega|_{p^{-1} (U\cap V)} = \tau|_{p^{-1} (U\cap V)}, \]
there exists a unique differential form $\delta \in \Omega^\bullet (p^{-1}(U\cup V))$ with
$\delta|_{p^{-1}U} = \omega,$ $\delta|_{p^{-1}V} = \tau.$
We must show that $\delta$ lies in $\ft_{<k} \oms^\bullet (U\cup V)\subset
\Omega^\bullet (p^{-1}(U\cup V))$. The restriction of $\omega$ to
$p^{-1} (U\cap U_\alpha)$ can be written as
\[ \omega|_{p^{-1} (U\cap U_\alpha)} = \phi^\ast_\alpha \sum_i
 \pi^\ast_1 \eta^U_i \wedge \pi^\ast_2 \gamma^U_i, \]
$\eta^U_i \in \Omega^\bullet (U\cap U_\alpha),$
$\gamma^U_i \in \tau_{<k} \Omega^\bullet (F).$ 
The restriction of $\tau$ to
$p^{-1} (V\cap U_\alpha)$ can be written as
\[ \tau|_{p^{-1} (V\cap U_\alpha)} = \phi^\ast_\alpha \sum_j
 \pi^\ast_1 \eta^V_j \wedge \pi^\ast_2 \gamma^V_j, \]
$\eta^V_j \in \Omega^\bullet (V\cap U_\alpha),$
$\gamma^V_j \in \tau_{<k} \Omega^\bullet (F).$ 
Therefore,
\begin{eqnarray*}
\delta|_{p^{-1} ((U\cup V)\cap U_\alpha)} & = &
((p^\ast \rho_U + p^\ast \rho_V)\cdot \delta)|_{p^{-1} (U\cap U_\alpha) \cup
  p^{-1} (V\cap U_\alpha)} \\
& = & p^\ast \rho_U \cdot \phi^\ast_\alpha \sum_i \pi^\ast_1 \eta^U_i \wedge
 \pi^\ast_2 \gamma^U_i + p^\ast \rho_V \cdot \phi^\ast_\alpha \sum_j 
  \pi^\ast_1 \eta^V_j \wedge \pi^\ast_2 \gamma^V_j \\ 
& = & \phi^\ast_\alpha (\sum_i \pi^\ast_1 (\rho_U \eta^U_i) \wedge
 \pi^\ast_2 \gamma^U_i + \sum_j 
  \pi^\ast_1 (\rho_V \eta^V_j) \wedge \pi^\ast_2 \gamma^V_j), \\ 
\end{eqnarray*}
which places $\delta$ in $\ft_{<k} \oms^\bullet (U\cup V).$ Thus
(\ref{equ.exmvftlkofh}) is exact at the middle group.
Since
\[ \ft_{<k} \oms^\bullet (U\cup V) \longrightarrow
 \ft_{<k} \oms^\bullet (U) \oplus \ft_{<k} \oms^\bullet (V) \]
is clearly injective, the sequence (\ref{equ.exmvftlkofh}) is exact.

Our next immediate objective is to create a similar sequence for cotruncated
multiplicatively structured forms with compact supports. The sum of $\sum \omega_\alpha
\in \ft_{\geq k} \omsc^\bullet (U)$ and $\sum \omega'_\alpha
\in \ft_{\geq k} \omsc^\bullet (V)$ can be written as
$\sum \omega_\alpha + \sum \omega'_\alpha = \sum (\omega_\alpha +
  \omega'_\alpha)$
with
\begin{eqnarray*}
\supp (\omega_\alpha + \omega'_\alpha) & \subset &
  \supp (\omega_\alpha) \cup \supp (\omega'_\alpha) \\
& \subset & p^{-1} (U\cap U_\alpha) \cup p^{-1} (V\cap U_\alpha) 
 =  p^{-1} ((U\cup V)\cap U_\alpha) 
\end{eqnarray*}
and
\[ \omega_\alpha + \omega'_\alpha = \phi^\ast_\alpha (\sum_i \pi^\ast_1
  \eta_i \wedge \pi^\ast_2 \gamma_i + \sum_j \pi^\ast_1 \eta'_j \wedge
  \pi^\ast_2 \gamma'_j), \]
$\eta_i \in \Omega^\bullet_c (U\cap U_\alpha),$
$\eta'_j \in \Omega^\bullet_c (V\cap U_\alpha);$
$\gamma_i, \gamma'_j \in \tau_{\geq k} \Omega^\bullet (F).$
Since by extension by zero
\[ \Omega^\bullet_c (U\cap U_\alpha) \subset
  \Omega^\bullet_c ((U\cup V)\cap U_\alpha) \supset 
  \Omega^\bullet_c (V\cap U_\alpha), \]
the forms $\eta_i$ and $\eta'_j$ all lie in $\Omega^\bullet_c ((U\cup V)\cap U_\alpha)$.
Consequently,
\[ \sum \omega_\alpha + \sum \omega'_\alpha \in \ft_{\geq k}
  \omsc^\bullet (U\cup V) \]
so that taking the sum of two forms defines a map
\[ \ft_{\geq k} \omsc^\bullet (U) \oplus \ft_{\geq k} \omsc^\bullet (V)
  \longrightarrow \ft_{\geq k} \omsc^\bullet (U\cup V). \]
We claim that this map is onto. Given a form $\omega \in \ft_{\geq k}
\omsc^\bullet (U\cup V),$ consider the forms $p^\ast (\rho_U)\omega \in
\Omega^\bullet_c (p^{-1} U)$ and $p^\ast (\rho_V)\omega \in
\Omega^\bullet_c (p^{-1} V)$. We have
$p^\ast (\rho_U)\omega = \sum p^\ast (\rho_U)\omega_\alpha$
with 
\begin{eqnarray*}
\supp (p^\ast (\rho_U)\omega_\alpha) & \subset &
\supp (p^\ast \rho_U)\cap \supp (\omega_\alpha) \\
& \subset & p^{-1}(U)\cap p^{-1} ((U\cup V)\cap U_\alpha) 
 =  p^{-1} (U\cap U_\alpha)
\end{eqnarray*}
and
\[ p^\ast (\rho_U)\omega_\alpha = p^\ast (\rho_U)\cdot \phi^\ast_\alpha
   \sum_j \pi^\ast_1 \eta_j \wedge \pi^\ast_2 \gamma_j =
\phi^\ast_\alpha
   \sum_j \pi^\ast_1 (\rho_U \eta_j) \wedge \pi^\ast_2 \gamma_j. \]
Since $\eta_j \in \Omega^\bullet_c ((U\cup V)\cap U_\alpha),$
\[ \supp (\rho_U \eta_j) \subset \supp (\rho_U)\cap \supp (\eta_j)
 \subset U\cap ((U\cup V)\cap U_\alpha) = U\cap U_\alpha \]
is compact. Thus $\rho_U \eta_j \in \Omega^\bullet_c (U\cap U_\alpha)$
and $p^\ast (\rho_U)\omega$ is an element in $\ft_{\geq k}
\omsc^\bullet (U)$. By symmetry, $p^\ast (\rho_V)\omega$ lies in $\ft_{\geq k}
\omsc^\bullet (V)$. The summation map sends the pair
$(p^\ast (\rho_U)\omega, p^\ast (\rho_V)\omega)$ to
$(p^\ast \rho_U + p^\ast \rho_V)\omega = \omega.$
The claim is verified. Given a form $\omega \in \ft_{\geq k} \omsc^\bullet
(U\cap V),$ extension by zero $\iota_\ast: \Omega^\bullet_c (p^{-1}
(U\cap V)) \rightarrow \Omega^\bullet_c (p^{-1}U)$ allows us to regard
$\omega$ as a form $\iota_\ast \omega \in \Omega^\bullet_c (p^{-1}U)$.
We claim that this form lies in fact in $\ft_{\geq k} \omsc^\bullet (U)$.
This is obvious as $\iota_\ast \omega = \sum \iota_\ast \omega_\alpha$ and
\[ \iota_\ast \omega_\alpha = \iota_\ast \phi^\ast_\alpha \sum_j
 \pi^\ast_1 \eta_j \wedge \pi^\ast_2 \gamma_j =
 \phi^\ast_\alpha \sum_j \pi^\ast_1 (\iota_\ast \eta_j)\wedge \pi^\ast_2
 \gamma_j, \]
where $\eta_j \in \Omega^\bullet_c (U\cap V \cap U_\alpha)$ and
$\iota_\ast \eta_j \in \Omega^\bullet_c (U\cap U_\alpha).$ Similarly, we
may regard $\omega$ as a form $\iota_\ast \omega \in \ft_{\geq k}
\omsc^\bullet (V)$. Extension by zero thus defines a map
\[ \begin{array}{rcl}
\ft_{\geq k} \omsc^\bullet (U\cap V) & \longrightarrow &
 \ft_{\geq k} \omsc^\bullet (U) \oplus \ft_{\geq k} \omsc^\bullet (V), \\
\omega & \mapsto & (-\iota_\ast \omega, \iota_\ast \omega),
\end{array} \]
which is clearly injective. We obtain a sequence
\begin{equation} \label{equ.mvftgrekfhc}
0 \rightarrow \ft_{\geq k} \omsc^\bullet (U\cap V) \longrightarrow
\ft_{\geq k} \omsc^\bullet (U) \oplus \ft_{\geq k} \omsc^\bullet (V)
\longrightarrow \ft_{\geq k} \omsc^\bullet (U\cup V)
\rightarrow 0.
\end{equation}
Exactness in the middle follows from the exactness of the standard
sequence
\[ 0 \rightarrow \Omega^\bullet_c (p^{-1} (U\cap V))
\longrightarrow \Omega^\bullet_c (p^{-1} U)\oplus
 \Omega^\bullet_c (p^{-1} V) \longrightarrow
 \Omega^\bullet_c (p^{-1} (U\cup V)) \rightarrow 0, \]
since the unique form $\tau \in \Omega^\bullet_c (p^{-1}
(U\cap V))$ which hits a given $(-\omega, \omega) \in
\ft_{\geq k} \omsc^\bullet (U) \oplus \ft_{\geq k} \omsc^\bullet (V)$
must actually lie in $\ft_{\geq k} \omsc^\bullet (U\cap V)$, which can
be seen as follows: We have compact $\supp (\omega) \subset
p^{-1}(U\cap V),$ and $\tau = \omega|_{p^{-1}(U\cap V)}.$
Let $f:B\to \real$ be a smooth function such that $f\equiv 1$
on the compact set $p(\supp \omega)$ and $\supp f \subset U\cap V$
is compact. Then $f\circ p \equiv 1$ on $\supp \omega$, so
$f\circ p\cdot \omega = \omega$. Thus
$\omega = p^\ast f \cdot \sum \omega_\alpha = \sum (p^\ast f)\cdot
\omega_\alpha$ with
$(p^\ast f)\cdot \omega_\alpha = \phi^\ast_\alpha \sum_j
\pi^\ast_1 (f\eta_j)\wedge \pi^\ast_2 \gamma_j.$
Since
\[ \supp (f\eta_j) \subset \supp f \cap \supp \eta_j \subset
 (U\cap V)\cap (U\cap U_\alpha) = U\cap V\cap U_\alpha \]
is compact, we have $f\eta_j \in \Omega^\bullet_c (U\cap V \cap U_\alpha).$
We have shown that the sequence (\ref{equ.mvftgrekfhc}) is
exact. The long exact cohomology sequences induced by
(\ref{equ.exmvftlkofh}) and (\ref{equ.mvftgrekfhc}) are dually
paired by the bilinear forms of Lemma \ref{lem.842}:
\[ \xymatrix@R=15pt@C=5pt{
H^r (\ft_{<K} \oms^\bullet (U\cup V)) \ar[d] & \otimes &
  H^{n+m-r} (\ft_{\geq K^\ast} \omsc^\bullet (U\cup V)) 
     \ar[rrr]^<<<<<<{\int_{p^{-1}(U\cup V)}} & & & \real \\
\mbox{$\begin{array}{c} H^r (\ft_{<K} \oms^\bullet (U)) \\ 
 \oplus H^r (\ft_{<K} \oms^\bullet (V)) \end{array}$ } \ar[d]
& \otimes &
  \mbox{$\begin{array}{c} H^{n+m-r} (\ft_{\geq K^\ast} \omsc^\bullet (U)) \\ 
 \oplus H^{n+m-r} (\ft_{\geq K^\ast} \omsc^\bullet (V)) \end{array}$}
 \ar[u] \ar[rrr]^<<<<<<{\mbox{$\begin{array}{c} 
 \int_{p^{-1}(U)} \\ + \int_{p^{-1}(V)} \end{array}$}} & & & \real \\
H^r (\ft_{<K} \oms^\bullet (U\cap V)) \ar[d]_{d^\ast} & \otimes &
  H^{n+m-r} (\ft_{\geq K^\ast} \omsc^\bullet (U\cap V)) \ar[u] 
  \ar[rrr]^<<<<<<{\int_{p^{-1}(U\cap V)}} & & & \real \\
H^{r+1} (\ft_{<K} \oms^\bullet (U\cup V)) & \otimes &
  H^{n+m-r-1} (\ft_{\geq K^\ast} \omsc^\bullet (U\cup V)) \ar[u]^{d_\ast} 
\ar[rrr]^<<<<<<{\int_{p^{-1}(U\cup V)}} & & & \real
} \]
The proof of Lemma 5.6 on page 45 of \cite{botttu} shows that this diagram
commutes up to sign. Since Poincar\'e duality holds over $U,V$ and $U\cap V$
by assumption, the $5$-lemma implies that it holds over $U\cup V$ as well.
\end{proof}

\begin{lemma} \label{lem.844}
For $U=B$, the identities
\[ \omsc^\bullet (B) = \oms^\bullet (B),~ 
  \ft_{\geq k} \omsc^\bullet (B) = \ft_{\geq k} \oms^\bullet (B) \]
hold.
\end{lemma}
\begin{proof}
Let $\omega = \sum \omega_\alpha$ be a form in $\ft_{\geq k} \omsc^\bullet (B)$.
Thus $\supp (\omega_\alpha) \subset p^{-1} U_\alpha,$
$\omega_\alpha = \phi^\ast_\alpha \sum_j \pi^\ast_1 \eta^\alpha_j \wedge
\pi^\ast_2 \gamma^\alpha_j,$ where $\eta^\alpha_j \in \Omega^\bullet_c 
(U_\alpha),$ $\gamma^\alpha_j \in \tau_{\geq k} \Omega^\bullet (F).$
Since the support of $\omega_\alpha$ is compact and contained in $p^{-1} U_\alpha$,
we may apply extension by zero $\iota^\alpha_\ast: \Omega^\bullet_c
(p^{-1} U_\alpha) \rightarrow \Omega^\bullet_c (E)$ to $\omega_\alpha$.
The result is a form $\iota^\alpha_\ast \omega_\alpha \in \ft_{\geq k}
\oms^\bullet (B).$ Then the finite sum $\sum_\alpha \iota^\alpha_\ast 
\omega_\alpha = \omega$ is in $\ft_{\geq k} \oms^\bullet (B)$ as well.

Let $\omega$ be a form in $\ft_{\geq k} \oms^\bullet (B)$. This means that
\[ \omega|_{p^{-1} U_\alpha} = \phi^\ast_\alpha
\sum_j \pi^\ast_1 \eta^\alpha_j \wedge \pi^\ast_2 \gamma^\alpha_j, \]
with $\eta^\alpha_j \in \Omega^\bullet (U_\alpha),$
$\gamma^\alpha_j \in \tau_{\geq k} \Omega^\bullet (F).$ Let $\{ \rho_\alpha \}$
be a partition of unity subordinate to $\mathfrak{U} = \{ U_\alpha \}$ such that
$\rho_\alpha$ has compact support contained in $U_\alpha$. Set
$\omega_\alpha = (p^\ast \rho_\alpha)\cdot \omega$. Then
$\omega = (\sum p^\ast \rho_\alpha)\omega = \sum \omega_\alpha,$
\[ \begin{array}{l}
\supp (\omega_\alpha) \subset \supp (p^\ast \rho_\alpha)\cap \supp (\omega)
  \subset p^{-1} (U_\alpha)\cap E = p^{-1} U_\alpha, \\
\omega_\alpha = \phi^\ast_\alpha \sum_j \pi^\ast_1 (\rho_\alpha \cdot
\eta^\alpha_j)\wedge \pi^\ast_2 \gamma^\alpha_j, 
\end{array} \]
with $\rho_\alpha \cdot \eta^\alpha_j$ having compact support
$\supp (\rho_\alpha \cdot \eta^\alpha_j) \subset \supp (\rho_\alpha)
 \subset U_\alpha.$
Hence $\omega \in \ft_{\geq k} \omsc^\bullet (B).$ Taking $k$ negative,
the first identity follows from the second.
\end{proof}

\begin{prop} \label{prop.globpdtruncfhforms}
(Global Poincar\'e Duality for Truncated Multiplicatively Structured Forms.)
Wedge product followed by integration induces a nondegenerate form
\[ H^r (\ft_{<K} \oms^\bullet (B)) \times
   H^{n+m-r} (\ft_{\geq K^\ast} \oms^\bullet (B)) \longrightarrow \real, \]
where $n=\dim B,$ $m=\dim F,$ $K=m-\bar{p}(m+1),$
$K^\ast = m-\bar{q}(m+1),$ and $\bar{p},\bar{q}$ are complementary
perversities. 
\end{prop}
\begin{proof}
By Lemma \ref{lem.844}, this is equivalent to proving that
\[ H^r (\ft_{<K} \oms^\bullet (B)) \times
   H^{n+m-r} (\ft_{\geq K^\ast} \omsc^\bullet (B)) \longrightarrow \real \]
is nondegenerate. We will in fact prove that
\[ H^r (\ft_{<K} \oms^\bullet (U)) \times
   H^{n+m-r} (\ft_{\geq K^\ast} \omsc^\bullet (U)) \longrightarrow \real \]
is nondegenerate for all open subsets $U\subset B$ of the form
\[ U = \bigcup_{i=1}^s U_{\alpha^i_0 \ldots \alpha^i_{p_i}} \]
by an induction on $s$. For $s=1,$ so that $U = U_{\alpha_0 \ldots \alpha_p}
\cong \real^n,$ the statement holds by Local Poincar\'e Duality, Lemma
\ref{lem.832}. Suppose the bilinear form is nondegenerate for all $U$ of
the form $U = \bigcup_{i=1}^{s-1} U_{\alpha^i_0 \ldots \alpha^i_{p_i}}$ .
Let $V$ be a set $V= U_{\alpha^s_0 \ldots \alpha^s_{p_s}}.$ By induction
hypothesis, the form is nondegenerate for $U$ and for
\[
U\cap V  =  (\bigcup_{i=1}^{s-1} U_{\alpha^i_0 \ldots \alpha^i_{p_i}})
 \cap U_{\alpha^s_0 \ldots \alpha^s_{p_s}} 
=  \bigcup_{i=1}^{s-1} U_{\alpha^i_0 \ldots \alpha^i_{p_i}
  \alpha^s_0 \ldots \alpha^s_{p_s}}.
\]
Since it also holds for $V$ by the induction basis, it follows from the
Bootstrap Lemma \ref{lem.bootstrap843} that the form is nondegenerate for
\[ U\cup V = \bigcup_{i=1}^s U_{\alpha^i_0 \ldots \alpha^i_{p_i}}. \]
The statement for $U=B$ follows as $B$ is the finite union
$B = \bigcup_\alpha U_\alpha.$
\end{proof}

\section{The Complex $\Omega I^\bullet_{\bar{p}}$}

Let $X^n$ be a stratified, compact pseudomanifold as in Section \ref{sec.abcomplexes}.
We continue to use the notation $(M,\partial M)$, $p: \partial M \to B=\Sigma,$ $F,$
$N = \interi (M)$ as introduced in that section. The link bundle $p$ is assumed 
to be flat and has structure group the isometries of $F$. Let $b=\dim B$.
The \emph{end}
$E = (-1,1)\times \partial M \subset N$ is defined using a collar. 
Let $j: E\hookrightarrow N$ be the inclusion
and $\pi: E\rightarrow \partial M$ the second-factor projection.
To the bundle $p$ one can associate a complex $\Omega^\bullet_\ms (B)
\subset \Omega^\bullet (\partial M)$ of multiplicatively structured forms as
shown in Section \ref{sec.fibharmflat}. We define forms on $N$ that
are multiplicatively structured near the end of $N$ (i.e. near the boundary of $M$)
as
\[ \Omega^r_\bms (N) = \{ \omega \in \Omega^r (N) ~|~
 j^\ast \omega = \pi^\ast \eta, \text{ some } \eta \in
 \Omega^r_\ms (B) \}.
\]
Then $\Omega^\bullet_\bms (N) \subset \Omega^\bullet (N)$ is a 
subcomplex, since
$j^\ast (d\omega) = dj^\ast  \omega = d\pi^\ast \eta = 
  \pi^\ast (d\eta)$
and $d\eta \in  \Omega^{r+1}_\ms (B).$ We shall show below that
this inclusion is a quasi-isomorphism. Cutoff values $K$ and
$K^\ast$ are defined by
\[ K = m- \bar{p}(m+1),~ K^\ast = m- \bar{q}(m+1), \]
with $\bar{p}, \bar{q}$ complementary perversities. In Section
\ref{sec.fibtruncandpd}, we defined and investigated a fiberwise cotruncation
$\ft_{\geq K} \Omega^\bullet_\ms (B)$. Using this complex, we define a 
complex $\oip^\bullet (N)$ by
\[ \label{page.oipdef}
\oip^\bullet (N) = \{ \omega \in \Omega^\bullet (N)
 ~|~ j^\ast \omega = \pi^\ast \eta, \text{ some } \eta \in
 \ft_{\geq K} \Omega^\bullet_\ms (B) \}. \]
It is obviously a subcomplex of $\Omega^\bullet_\bms (N)$.
The cohomology theory $HI^\bullet_{\bar{p}} (X)$ is the cohomology of this
complex.
\begin{defn}
The cohomology groups $HI^\bullet_{\bar{p}}(X)$ are defined to be
\[ HI^r_{\bar{p}}(X) = H^r (\oip^\bullet (N)). \]
\end{defn}
It follows from Proposition \ref{prop.indepriemmetric} that
the groups $HI^\bullet_{\bar{p}}(X)$ are independent of the Riemannian
metric on the link, where the metric is allowed to vary within all metrics
such that the transition functions of the link bundle are isometries.
Let $\Omega^\bullet_\bms (E)$ be defined as
$\Omega^\bullet_\bms (E) =
\{ \omega \in \Omega^\bullet (E) ~|~
 \omega = \pi^\ast \eta, \text{ some } \eta \in
 \Omega^\bullet_\ms (B) \}.$

\begin{lemma} \label{lem.13star}
The maps
\[ \xymatrix{\Omega^\bullet_{\bms}(E) & \Omega^\bullet_\ms (B) 
\ar@<1ex>[l]^{\pi^\ast}  \ar@<1ex>[l];[]^{\sigma^\ast_0}
} \]
are mutually inverse isomorphisms of cochain complexes, where
$\sigma_0: \partial M \hookrightarrow E = (-1,+1)\times \partial M$
is given by $\sigma_0 (x) = (0,x)$.
\end{lemma}
The proof of this is obvious.
In Section \ref{sec.bh}, a complex $\Omega^\bullet_{\bc} (N)$ was defined by
\[ \Omega^\bullet_{\bc} (N) = \{ \omega \in \Omega^\bullet (N) ~|~
 j^\ast \omega = \pi^\ast \eta, \text{ some } \eta \in \Omega^\bullet
 (\partial M) \}; \]
likewise, one has $\Omega^\bullet_{\bc} (E)$. 
In a similar vein as Lemma \ref{lem.13star}, we also have that
\[ \xymatrix{\Omega^\bullet_{\bc}(E) & \Omega^\bullet (\partial M) 
\ar@<1ex>[l]^{\pi^\ast}  \ar@<1ex>[l];[]^{\sigma^\ast_0}
} \]
are mutually inverse isomorphisms of cochain complexes.

\begin{prop} \label{prop.11star}
The inclusion $\Omega^\bullet_\bms (N)\subset \Omega^\bullet (N)$
induces an isomorphism
\[ H^\bullet (\Omega^\bullet_\bms (N)) \cong H^\bullet (N) \]
on cohomology.
\end{prop}
\begin{proof}
The restriction map $j^\ast: \Omega^\bullet_\bms (N)\rightarrow
\Omega^\bullet_\bms (E)$ is onto: Given a pullback $\pi^\ast \eta \in
\Omega^\bullet_\bms (E),$ extend a little further to
$E_{-2} = (-2,1) \times \partial M$ by taking $\pi^\ast_{-2}
\eta,$ where $\pi_{-2}: (-2,1)\times \partial M \rightarrow
\partial M$ is the second-factor projection, then multiply by a
cutoff function which is identically $1$ on $E$, zero on
$(-2, -\frac{3}{2})\times \partial M$ and depends only on the collar
coordinate, not on the coordinates in $\partial M$.
Since the kernel of $j^\ast$
is $\Omega^\bullet_{\rel} (N),$ we have an exact sequence
\[ 0 \rightarrow
 \Omega^\bullet_{\rel} (N) \longrightarrow
 \Omega^\bullet_\bms (N) \longrightarrow
 \Omega^\bullet_\bms (E) \rightarrow 0. \]
Similarly, the restriction map $\Omega^\bullet_{\bc} (N)
\rightarrow \Omega^\bullet_{\bc} (E)$ is onto. Its kernel is also
$\Omega^\bullet_{\rel} (N),$ and we get a commutative diagram
\[ \xymatrix{
0 \ar[r] & \Omega^\bullet_{\rel} (N) \ar[r] & \Omega^\bullet_{\bc} (N)
 \ar[r] & \Omega^\bullet_{\bc} (E) \ar[r] & 0 \\
0 \ar[r] & \Omega^\bullet_{\rel} (N) \ar[r] \ar@{=}[u] & 
\Omega^\bullet_{\bms} (N)
 \ar[r] \ar@{^{(}->}[u] & \Omega^\bullet_{\bms} (E) 
\ar[r] \ar@{^{(}->}[u] & 0.
} \]
By Lemma \ref{lem.13star}, $\sigma^\ast_0$ and $\pi^\ast$ induce
isomorphisms
\[ \Omega^\bullet_{\bms} (E)\cong \Omega^\bullet_\ms (B),~
 \Omega^\bullet_{\bc} (E) \cong \Omega^\bullet (\partial M), \]
and the square
\[ \xymatrix{
\Omega^\bullet_{\bc} (E) \ar@{=}[r]^{\sim} &
 \Omega^\bullet (\partial M) \\
\Omega^\bullet_{\bms} (E) \ar@{=}[r]^{\sim} \ar@{^{(}->}[u] &
 \Omega^\bullet_\ms (B) \ar@{^{(}->}[u]
} \]
commutes. On cohomology, we arrive at a commutative diagram with
long exact rows,
\[ \xymatrix{
H^\bullet_{\rel} (N) \ar[r] &
 H^\bullet_ {\bc} (N) \ar[r] & H^\bullet (\partial M) \ar[r] & 
  H^{\bullet +1}_{\rel} (N) \\
H^\bullet_{\rel} (N) \ar[r] \ar@{=}[u] &
 H^\bullet_{\bms} (N) \ar[r] \ar[u]
& H^\bullet (\Omega^\bullet_\ms (B)) \ar[r] \ar[u]_{\cong} & 
 H^{\bullet +1}_{\rel} (N). \ar@{=}[u] 
} \]
The vertical arrow $H^\bullet (\Omega^\bullet_\ms (B)) \rightarrow
H^\bullet (\partial M)$ is an isomorphism by Theorem
\ref{thm.fhcomputescohtotspace}. By the 5-lemma,
$H^\bullet_{\bms} (N) \rightarrow H^\bullet_{\bc} (N)$ is an
isomorphism. The inclusion $\Omega^\bullet_{\bc} (N)\subset
\Omega^\bullet (N)$ induces an isomorphism
$H^\bullet_{\bc} (N)\rightarrow H^\bullet (N)$ by 
Proposition \ref{lem.ombcinom}. Thus the composition
\[ \xymatrix{
H^\bullet_{\bms} (N) \ar[r]^{\cong} \ar[rd]_{\cong} &
H^\bullet_{\bc} (N) \ar[d]^{\cong} \\
& H^\bullet (N)
} \]
is an isomorphism as well.
\end{proof}
For an open subset $U\subset B,$ we set
\[ Q^\bullet (U) = \frac{\Omega^\bullet_\ms (U)}
 {\ft_{\geq K} \Omega^\bullet_\ms (U)}. \]
\begin{lemma} \label{lem.91}
Given open subsets $U,V \subset B,$ there is a Mayer-Vietoris
exact sequence
\[ \cdots \stackrel{\delta^\ast}{\rightarrow}
H^r Q^\bullet (U\cup V) \rightarrow
H^r Q^\bullet (U) \oplus H^r Q^\bullet (V) \rightarrow
H^r Q^\bullet (U\cap V) 
\stackrel{\delta^\ast}{\rightarrow}
H^{r+1} Q^\bullet (U\cup V) \rightarrow \cdots. \]
\end{lemma}
\begin{proof}
The arguments in the proof of Lemma \ref{lem.bootstrap843}
that establish the exactness of the fiberwise truncation
sequence (\ref{equ.exmvftlkofh}),
\[ 0 \rightarrow \ft_{<K} \Omega^\bullet_\ms (U\cup V) \rightarrow
\ft_{<K} \Omega^\bullet_\ms (U) \oplus \ft_{<K} \Omega^\bullet_\ms
(V) \rightarrow \ft_{<K} \Omega^\bullet_\ms (U\cap V) 
\rightarrow 0 \]
also apply to show that there is an analogous exact fiberwise
cotruncation sequence
\[ 0 \rightarrow \ft_{\geq K} \Omega^\bullet_\ms (U\cup V) \rightarrow
\ft_{\geq K} \Omega^\bullet_\ms (U) \oplus \ft_{\geq K} \Omega^\bullet_\ms
(V) \rightarrow \ft_{\geq K} \Omega^\bullet_\ms (U\cap V) 
\rightarrow 0 \]
because the fiber forms $\gamma_j, \gamma^U_i, \gamma^V_j$ appearing
in these arguments may just as well come from
$\tau_{\geq k} \Omega^\bullet (F)$ instead of $\tau_{<k} \Omega^\bullet (F)$.
There is a unique map $Q^\bullet (U\cup V)\rightarrow Q^\bullet (U)
\oplus Q^\bullet (V)$ such that
\[ \xymatrix@C=10pt{
0 \ar[r] & \ft_{\geq K} \Omega^\bullet_\ms (U\cup V) \ar[r] \ar@{^{(}->}[d] &
 \Omega^\bullet_\ms (U\cup V) \ar[r] \ar@{^{(}->}[d] &
 Q^\bullet (U\cup V) \ar[r] \ar@{..>}[d] & 0 \\
0 \ar[r] & \ft_{\geq K} \Omega^\bullet_\ms (U) 
 \oplus \ft_{\geq K} \Omega^\bullet_\ms (V)
\ar[r] &
 \Omega^\bullet_\ms (U) \oplus \Omega^\bullet_\ms (V) \ar[r]  &
 Q^\bullet (U) \oplus Q^\bullet (V) \ar[r]  & 0
} \]
commutes and a unique map $Q^\bullet (U)\oplus Q^\bullet (V)
\rightarrow Q^\bullet (U\cap V)$ such that
\[ \xymatrix@C=10pt{
0 \ar[r] & \ft_{\geq K} \Omega^\bullet_\ms (U) 
 \oplus \ft_{\geq K} \Omega^\bullet_\ms (V)
\ar[r] \ar@{->>}[d] &
 \Omega^\bullet_\ms (U) \oplus \Omega^\bullet_\ms (V) \ar[r]  \ar@{->>}[d] &
 Q^\bullet (U) \oplus Q^\bullet (V) \ar[r]  \ar@{..>}[d] & 0 \\
0 \ar[r] & \ft_{\geq K} \Omega^\bullet_\ms (U\cap V) \ar[r] &
 \Omega^\bullet_\ms (U\cap V) \ar[r] &
 Q^\bullet (U\cap V) \ar[r] & 0
} \]
commutes.
We receive a commutative $3\times 3$-diagram
\[ \xymatrix@C=10pt@R=16pt{
& 0 \ar[d] & 0 \ar[d] & 0 \ar[d] & \\
0 \ar[r] & \ft_{\geq K} \Omega^\bullet_\ms (U\cup V) \ar[r] \ar[d] &
\ft_{\geq K} \Omega^\bullet_\ms (U) \oplus \ft_{\geq K} \Omega^\bullet_\ms
(V) \ar[r] \ar[d] &
\ft_{\geq K} \Omega^\bullet_\ms (U\cap V) \ar[r] \ar[d] & 0 \\
0 \ar[r] & \Omega^\bullet_\ms (U\cup V) \ar[r] \ar[d] &
\Omega^\bullet_\ms (U) \oplus \Omega^\bullet_\ms
(V) \ar[r] \ar[d] &
\Omega^\bullet_\ms (U\cap V) \ar[r] \ar[d] & 0 \\
0 \ar[r] & Q^\bullet (U\cup V) \ar[r] \ar[d] &
Q^\bullet (U) \oplus Q^\bullet
(V) \ar[r] \ar[d] &
Q^\bullet (U\cap V) \ar[r] \ar[d] & 0 \\
& 0 & 0 & 0 &
} \]
with all columns and the top two rows exact. By the $3\times 3$-lemma,
the bottom row is exact as well. By the standard zig-zag construction,
the bottom row induces a long exact sequence on cohomology.
\end{proof}
For every open subset $U\subset B,$ we define a canonical map
\[ \gamma_U: \ft_{<K} \Omega^\bullet_\ms (U) \longrightarrow
  Q^\bullet (U) \]
by composing
\[ \ft_{<K} \Omega^\bullet_\ms (U) \stackrel{\operatorname{incl}}
{\hookrightarrow} \Omega^\bullet_\ms (U)
 \stackrel{\operatorname{quot}}{\longrightarrow} 
Q^\bullet (U). \]
Our next goal is to show that $\gamma_B$ is a quasi-isomorphism.
To prove this, we will use the following bootstrap principle:
\begin{lemma} \label{lem.92}
Let $U,V \subset B$ be open subsets. If $\gamma_U, \gamma_V$
and $\gamma_{U\cap V}$ are quasi-isomorphisms, then
$\gamma_{U\cup V}$ is a quasi-isomorphism as well.
\end{lemma}
\begin{proof}
In the proof of Lemma \ref{lem.bootstrap843}, we had developed an
exact Mayer-Vietoris sequence
\[ H^r (\ft_{<K} \Omega^\bullet_\ms (U\cup V)) \to
  H^r (\ft_{<K} \Omega^\bullet_\ms (U)) \oplus
 H^r (\ft_{<K} \Omega^\bullet_\ms (V)) \to \]
\[ H^r (\ft_{<K} \Omega^\bullet_\ms (U\cap V))
\stackrel{d^\ast}{\to} \cdots. \]
Mapping this sequence to the Mayer-Vietoris sequence of 
Lemma \ref{lem.91} via $\gamma$, we obtain a commutative
diagram
\[ \xymatrix@C=50pt@R=18pt{
H^r (\ft_{<K} \Omega^\bullet_\ms (U\cup V)) \ar[r]^{\gamma_{U\cup V}}
 \ar[d] & H^r Q^\bullet (U\cup V) \ar[d] \\
H^r (\ft_{<K} \Omega^\bullet_\ms (U)) \oplus
 H^r (\ft_{<K} \Omega^\bullet_\ms (V)) 
 \ar[r]^>>>>>>>>>>{\gamma_U \oplus \gamma_V}_>>>>>>>>>{\cong}
\ar[d] & H^r Q^\bullet (U) \oplus H^r Q^\bullet (V)
 \ar[d] \\
H^r (\ft_{<K} \Omega^\bullet_\ms (U\cap V)) 
\ar[r]^{\gamma_{U\cap V}}_{\cong}
\ar[d]_{d^\ast} & H^r Q^\bullet (U\cap V) \ar[d]^{d^\ast} \\
H^{r+1} (\ft_{<K} \Omega^\bullet_\ms (U\cup V)) 
\ar[r]^{\gamma_{U\cup V}}
 & H^{r+1} Q^\bullet (U\cup V) 
} \]
The $5$-lemma concludes the proof.
\end{proof}

\begin{lemma} \label{lem.93}
The map $\gamma_B: \ft_{<K} \Omega^\bullet_\ms (B) \to Q^\bullet (B)$
induces an isomorphism
\[ H^\bullet (\ft_{<K} \Omega^\bullet_\ms (B)) \longrightarrow
 H^\bullet Q^\bullet (B) \]
on cohomology.
\end{lemma}
\begin{proof}
We shall show that $\gamma_U$ is a quasi-isomorphism for all open
$U$ of the form
\[ U = \bigcup_{i=1}^s U_{\alpha^i_0 \ldots \alpha^i_{p_i}} \]
by an induction on $s$, where $\{ U_\alpha \}$ is a finite good cover
of $B$ with respect to which the link bundle trivializes.
Let $s=1$ so that $U = U_{\alpha_0 \ldots \alpha_p} \cong
\real^b$. The inclusion $\im d^{K-1} \subset \Omega^K F$ induces
an isomorphism
\[ \im d^{K-1} \stackrel{\cong}{\longrightarrow}
\frac{\ker d^\ast \oplus \im d^{K-1}}{\ker d^\ast} =
\frac{\Omega^K F}{(\tau_{\geq K} \Omega^\bullet F)^K}, \]
which can be extended to an isomorphism of complexes
\[ \xymatrix@C=12pt@R=16pt{
\tau_{<K} \Omega^\bullet (F) =\cdots \ar[d]_{\gamma}^{\cong} \ar[r] & \Omega^{K-2} (F) \ar@{=}[d] \ar[r] &
\Omega^{K-1}(F) \ar@{=}[d] \ar[r] & \im d^{K-1} \ar[d]_{\cong} \ar[r] &
0 \ar[d] \ar[r] & \cdots \\
\Omega^\bullet F/\tau_{\geq K} \Omega^\bullet F  =\cdots \ar[r] & 
\Omega^{K-2} (F) \ar[r] &
\Omega^{K-1}(F) \ar[r] & \frac{\Omega^K F}{(\tau_{\geq K} \Omega^\bullet F)^K} 
\ar[r] & 0 \ar[r] &  \cdots.
} \]
This isomorphism can be factored as
\[ \gamma: \tau_{<K} \Omega^\bullet (F) 
 \stackrel{\operatorname{incl}}{\hookrightarrow}
 \Omega^\bullet (F) \stackrel{\operatorname{quot}}{\longrightarrow}
\frac{\Omega^\bullet (F)}{\tau_{\geq K} \Omega^\bullet (F)}. \]
According to the Poincar\'e Lemmas
\ref{lem.822} and \ref{lem.822cotrunc}, the restriction 
$S^\ast_0$ of a form on $\real^b \times F$ to
$\{ 0 \} \times F = F$ provides a homotopy equivalence
$S^\ast_0: \ft_{<K} \Omega^\bullet_\ms (\real^b) 
 \stackrel{\simeq}{\longrightarrow} \tau_{<K} \Omega^\bullet (F)$
and a homotopy equivalence
$S^\ast_0: \ft_{\geq K} \Omega^\bullet_\ms (\real^b) 
 \stackrel{\simeq}{\longrightarrow} \tau_{\geq K} \Omega^\bullet (F).$
Taking $K$ negative in the latter homotopy equivalence (or $K$ larger
than $m$ in the former), we get in particular a homotopy equivalence
\[ S^\ast_0: \Omega^\bullet_\ms (\real^b) 
 \stackrel{\simeq}{\longrightarrow} \Omega^\bullet (F). \]
The map $S^\ast_0$ induces a unique map
\[ Q^\bullet (\real^b) \longrightarrow
 \frac{\Omega^\bullet (F)}{\tau_{\geq K} \Omega^\bullet (F)} \]
such that
\[ \xymatrix{
0 \ar[r] & \ft_{\geq K} \Omega^\bullet_\ms (\real^b) \ar[r]
 \ar[d]_{\simeq}^{S^\ast_0} &
 \Omega^\bullet_\ms (\real^b) \ar[r] \ar[d]_{\simeq}^{S^\ast_0} &
 Q^\bullet (\real^b) \ar[r] \ar@{..>}[d] & 0 \\
0 \ar[r] & \tau_{\geq K} \Omega^\bullet (F) \ar[r] &
 \Omega^\bullet (F) \ar[r] &
\frac{\Omega^\bullet (F)}{\tau_{\geq K} \Omega^\bullet (F)}
\ar[r] & 0
} \]
commutes. This map is a quasi-isomorphism by the $5$-lemma.
By the commutativity of
\[ \xymatrix{
H^\bullet (\ft_{<K} \Omega^\bullet_\ms (\real^b))
 \ar[r]^\cong_{S^\ast_0} \ar[d]^{\operatorname{incl}^\ast} 
 \ar@/_3pc/[dd]_{\gamma^\ast_{\real^b}}  &
 H^\bullet (\tau_{<K} \Omega^\bullet (F)) 
 \ar[d]_{\operatorname{incl}^\ast} 
 \ar@/^3pc/[dd]^{\gamma^\ast,~ \cong}  \\
H^\bullet (\Omega^\bullet_\ms (\real^b))
 \ar[r]^\cong_{S^\ast_0} \ar[d]^{\operatorname{quot}^\ast} &
 H^\bullet (\Omega^\bullet (F)) 
 \ar[d]_{\operatorname{quot}^\ast} \\
H^\bullet Q^\bullet (\real^b) \ar[r]^>>>>>>>{\cong}_>>>>>>>{S^\ast_0} &
H^\bullet (\Omega^\bullet (F)/ \tau_{\geq K}
 \Omega^\bullet (F)),
} \]
the map $\gamma_{\real^b}$ is a quasi-isomorphism. This furnishes the
induction basis. Suppose $\gamma_U$ is a quasi-isomorphism for all $U$
of the form
$U = \bigcup_{i=1}^{s-1} U_{\alpha^i_0 \ldots \alpha^i_{p_i}}.$
Let $V$ be a set $V = U_{\alpha^s_0 \ldots \alpha^s_{p_s}}.$ By the
induction hypothesis, $\gamma_U$ is a quasi-isomorphism and
$\gamma_{U\cap V}$ is a quasi-isomorphism, as
$U\cap V = \bigcup_{i=1}^{s-1} 
  U_{\alpha^i_0 \ldots \alpha^i_{p_i} \alpha^s_0 \ldots \alpha^s_{p_s}}.$
Since $\gamma_V$ is a quasi-isomorphism as well ($s=1$), the
bootstrap Lemma \ref{lem.92} implies that $\gamma_{U\cup V}$
is a quasi-isomorphism,
$U \cup V = \bigcup_{i=1}^{s} U_{\alpha^i_0 \ldots \alpha^i_{p_i}}.$
The statement for $U=B$ follows as $B$ is the finite union
$B = \bigcup_\alpha U_\alpha.$
\end{proof}

Let $\Da (\real)$ denote the derived category of complexes of real vector
spaces. The exact sequence
\[ 0 \longrightarrow \ft_{\geq K} \Omega^\bullet_\ms (B)
 \longrightarrow \Omega^\bullet_\ms (B)
 \longrightarrow Q^\bullet (B) \longrightarrow 0 \]
induces a distinguished triangle 
\[ \xymatrix@C=10pt{
\ft_{\geq K} \Omega^\bullet_\ms (B) \ar[rr] & &
 \Omega^\bullet_\ms (B) \ar[ld] \\
& Q^\bullet (B) \ar[lu]^{+1} &
} \]
in $\Da (\real)$. Using the quasi-isomorphism $\gamma_B$ of Lemma
\ref{lem.93}, we may replace $Q^\bullet (B)$ in the triangle by
$\ft_{<K} \Omega^\bullet_\ms (B)$ and thus arrive at a distinguished
triangle
\begin{equation} \label{equ.dt1}
\xymatrix@C=10pt{
\ft_{\geq K} \Omega^\bullet_\ms (B) \ar[rr] & &
 \Omega^\bullet_\ms (B) \ar[ld] \\
& \ft_{<K} \Omega^\bullet_\ms (B). \ar[lu]^{+1} &
} \end{equation}
On the basis of this triangle, we shall next construct a distinguished
triangle
\begin{equation} \label{equ.dt2}
\xymatrix@C=10pt{
\oip^\bullet (N) \ar[rr] & &
 \Omega^\bullet_\bms (N) \ar[ld] \\
& \ft_{<K} \Omega^\bullet_\ms (B). \ar[lu]^{+1} &
} \end{equation}
Since $\oip^\bullet (N)$ is a subcomplex of
$\Omega^\bullet_\bms (N),$ there is an exact sequence
\[ 0 \longrightarrow \oip^\bullet (N)
 \longrightarrow \Omega^\bullet_\bms (N) \longrightarrow
 \frac{\Omega^\bullet_\bms (N)}{\oip^\bullet (N)}
 \longrightarrow 0. \]
The inclusion $j: E \hookrightarrow N$ induces a restriction map
$j^\ast: \Omega^\bullet_\bms (N) \longrightarrow
  \Omega^\bullet_\bms (E),$
which is surjective (cf. the proof of Proposition \ref{prop.11star}).
This map restricts further to a map
$j^\ast_{\bar{p}}: \oip^\bullet (N) \longrightarrow 
 \oip^\bullet (E),$
which is also surjective. Based on Lemma \ref{lem.13star}, there are
isomorphisms
\[ \sigma^\ast_0: \Omega^\bullet_{\bms} (E) \stackrel{\cong}{\longrightarrow} 
\Omega^\bullet_\ms (B),~
 \sigma^\ast_0: \oip^\bullet (E) \stackrel{\cong}{\longrightarrow}
 \ft_{\geq K} \Omega^\bullet_\ms (B), \]
which induce a unique isomorphism
\[ \frac{\Omega^\bullet_\bms (E)}{\oip^\bullet (E)}
 \stackrel{\cong}{\longrightarrow}
 \frac{\Omega^\bullet_\ms (B)}{\ft_{\geq K} \Omega^\bullet_\ms (B)}
 = Q^\bullet (B) \]
such that
\[ \xymatrix@R=16pt{
0 \ar[r] & \oip^\bullet (E) \ar[d]_{\sigma^\ast_0}^{\cong} \ar[r]
& \Omega^\bullet_\bms (E) \ar[d]_{\sigma^\ast_0}^{\cong} \ar[r] &
\frac{\Omega^\bullet_\bms (E)}{\oip^\bullet (E)} 
  \ar[r] \ar@{..>}[d]^{\cong}
& 0 \\
0 \ar[r] & \ft_{\geq K} \Omega^\bullet_\ms (B) \ar[r] &
\Omega^\bullet_\ms (B) \ar[r] & Q^\bullet (B) \ar[r] & 0
} \]
commutes. The surjective maps $j^\ast$ induce a unique surjective map
\[ \bar{\jmath}^\ast: \frac{\Omega^\bullet_\bms (N)}{\oip^\bullet (N)}
 \twoheadrightarrow \frac{\Omega^\bullet_\bms (E)}{\oip^\bullet (E)} \]
such that
\[ \xymatrix@R=16pt{
0 \ar[r] & \oip^\bullet (N) \ar@{->>}[d]_{j^\ast_{\bar{p}}} \ar[r]
& \Omega^\bullet_\bms (N) \ar@{->>}[d]_{j^\ast} \ar[r] &
\frac{\Omega^\bullet_\bms (N)}{\oip^\bullet (N)} 
  \ar[r] \ar@{..>>}[d]_{\bar{\jmath}^\ast}
& 0 \\
0 \ar[r] & \oip^\bullet (E) \ar[r]
& \Omega^\bullet_\bms (E) \ar[r] &
\frac{\Omega^\bullet_\bms (E)}{\oip^\bullet (E)} \ar[r]
& 0 
} \]
commutes. Composition yields surjective maps
\[ J^\ast = \sigma^\ast_0 j^\ast,~ J^\ast_{\bar{p}} = \sigma^\ast_0 j^\ast_{\bar{p}},~
  \overline{J}^\ast = \sigma^\ast_0 \bar{\jmath}^\ast \]
such that
\[ \xymatrix@R=16pt{
0 \ar[r] & \oip^\bullet (N) \ar@{->>}[d]_{J^\ast_{\bar{p}}} \ar[r]
& \Omega^\bullet_\bms (N) \ar@{->>}[d]_{J^\ast} \ar[r] &
\frac{\Omega^\bullet_\bms (N)}{\oip^\bullet (N)} 
  \ar[r] \ar@{->>}[d]_{\overline{J}^\ast}
& 0 \\
0 \ar[r] & \ft_{\geq K} \Omega^\bullet_\ms (B) \ar[r] &
\Omega^\bullet_\ms (B) \ar[r] & Q^\bullet (B) \ar[r] & 0
} \]
commutes. The kernel of both $J^\ast$ and $J^\ast_{\bar{p}}$ is
\[ \ker J^\ast = \ker j^\ast = \Omega^\bullet_{\rel} (N) =
 \ker j^\ast_{\bar{p}} = \ker J^\ast_{\bar{p}}. \]
We obtain a commutative $3\times 3$-diagram
\[ \xymatrix@R=17pt{
& 0 \ar[d] & 0 \ar[d] & 0 \ar[d] & \\
0 \ar[r] & \Omega^\bullet_{\rel} (N) \ar@{=}[r] \ar[d]
 & \Omega^\bullet_{\rel} (N) \ar[r] \ar[d] & 0 \ar[r] \ar[d]
& 0 \\
0 \ar[r] & \oip^\bullet (N) \ar[r] \ar[d]_{J^\ast_{\bar{p}}} &
 \Omega^\bullet_\bms (N) \ar[r] \ar[d]_{J^\ast} &
 \frac{\Omega^\bullet_\bms (N)}{\oip^\bullet (N)}
 \ar[r] \ar[d]_{\overline{J}^\ast} & 0 \\
0 \ar[r] & \ft_{\geq K} \Omega^\bullet_\ms (B) \ar[r] \ar[d]
& \Omega^\bullet_\ms (B) \ar[r] \ar[d] & Q^\bullet (B)
 \ar[r] \ar[d] & 0 \\
& 0 & 0 & 0 & 
} \]
with exact rows. Since the left hand and middle columns are
also exact, the $3\times 3$-lemma implies that the right hand
column is exact, too. This proves that $\overline{J}^\ast$ is
an isomorphism. Using the isomorphism
\[ \gamma^{-1}_B \circ \overline{J}^\ast:
\frac{\Omega^\bullet_\bms (N)}{\oip^\bullet (N)}
\stackrel{\cong}{\longrightarrow} \ft_{<K} \Omega^\bullet_\ms (B) \]
in $\Da (\real)$ to replace the quotient in the distinguished triangle
\[
\xymatrix@C=10pt{
\oip^\bullet (N) \ar[rr] & &
 \Omega^\bullet_\bms (N) \ar[ld] \\
& \frac{\Omega^\bullet_\bms (N)}{\oip^\bullet (N)} \ar[lu]^{+1} &
} \]
by $\ft_{<K} \Omega^\bullet_\ms (B),$ we arrive at the desired triangle
(\ref{equ.dt2}). As the kernel of the surjective map
$J^\ast_{\bar{p}}: \oip^\bullet (N) \twoheadrightarrow \ft_{\geq K}
\Omega^\bullet_\ms (B)$ is $\Omega^\bullet_{\rel} (N),$ there is also
a distinguished triangle
\begin{equation} \label{equ.dt3}
\xymatrix@C=10pt{
\Omega^\bullet_{\rel} (N) \ar[rr]^{\operatorname{incl}} & &
 \oip^\bullet (N) \ar[ld]^{J^\ast_{\bar{p}}} \\
& \ft _{\geq K} \Omega^\bullet_\ms (B). \ar[lu]^{+1} &
} \end{equation}
These triangles will be used in proving Poincar\'e duality for $HI^\bullet (X)$.

\section{Integration on $\Omega I^\bullet_{\bar{p}}$}
\label{sec.intnonisolbcomplex}

\begin{lemma}
Integration defines bilinear forms
\[ \begin{array}{rcl}
\int: \Omega^r_\bms (N) \times \Omega^{n-r}_\bms (N)
& \longrightarrow & \real \\
(\omega, \eta) & \mapsto & \int_N \omega \wedge \eta. 
\end{array} \]
\end{lemma}
\begin{proof}
Let $\omega \in \Omega^r_\bms (N),$
$\eta \in \Omega^{n-r}_\bms (N).$ By definition, there exists an
$r$-form $\omega_0 \in \Omega^r_\ms (B)$ and an $(n-r)$-form
$\eta_0 \in \Omega^{n-r}_\ms (B)$ such that
$j^\ast \omega = \pi^\ast \omega_0,$
$j^\ast \eta = \pi^\ast \eta_0.$ Note that
\[ j^\ast (\omega \wedge \eta) = j^\ast \omega \wedge j^\ast \eta
 = \pi^\ast \omega_0 \wedge \pi^\ast \eta_0 = \pi^\ast
 (\omega_0 \wedge \eta_0) =0, \]
as $\omega_0 \wedge \eta_0$ is an $n$-form on the
$(n-1)$-dimensional manifold $\partial M$. Consequently,
\[ \int_N \omega \wedge \eta = \int_{N-E} \omega \wedge \eta + 
 \int_E j^\ast (\omega \wedge \eta) =
 \int_{N-E} \omega \wedge \eta  \]
is finite, since $N-E$ is compact and $\omega \wedge \eta$ is smooth
on a neighborhood of $N-E$.
\end{proof}
Since $\oip^\bullet (N)$ is  a subcomplex of 
$\Omega^\bullet_\bms (N),$ we obtain in particular:
\begin{cor} \label{cor.10.2}
Integration defines bilinear forms
\[ 
\int: \oip^r (N) \times \oiq^{n-r} (N)
 \longrightarrow \real. \]
\end{cor}

\begin{lemma} \label{lem.10.3}
For forms $\nu_0 \in (\ft_{\geq K} \Omega^\bullet_\ms
(B))^{r-1}$ and $\eta_0 \in (\ft_{\geq K^\ast} \Omega^\bullet_\ms
(B))^{n-r},$ the vanishing result
$\int_{\partial M} \nu_0 \wedge \eta_0 =0$
holds.
\end{lemma}
\begin{proof}
Let $\{ \rho_\alpha \}$ be a partition of unity subordinate to
$\mathfrak{U} = \{ U_\alpha \},$ $\supp (\rho_\alpha)\subset U_\alpha$
compact. Then $\{ \overline{\rho}_\alpha \},$ $\overline{\rho}_\alpha =
\rho_\alpha \circ p,$ is a partition of unity subordinate to
$p^{-1} \mathfrak{U} = \{ p^{-1} U_\alpha \}.$ Since 
\[ \int_{\partial M} \nu_0 \wedge \eta_0 = \int_{\partial M}
 (\sum \overline{\rho}_\alpha)\cdot \nu_0 \wedge \eta_0 =
 \sum \int_{\partial M} \overline{\rho}_\alpha \nu_0 \wedge \eta_0
 = \sum \int_{p^{-1} U_\alpha} \overline{\rho}_\alpha \nu_0
\wedge \eta_0, \]
it suffices to show that
\[ \int_{p^{-1} U_\alpha} \overline{\rho}_\alpha \nu_0
\wedge \eta_0 =0 \]
for all $\alpha$. Let $\phi_\alpha: p^{-1} U_\alpha 
\stackrel{\cong}{\longrightarrow} U_\alpha \times F$ be the
trivialization over $U_\alpha$. Over $U_\alpha,$ $\nu_0$ has
the form
\[ \nu_0|_{p^{-1} U_\alpha} = \phi^\ast_\alpha \sum_{i=1}^k
 \pi^\ast_1 \nu_i \wedge \pi^\ast_2 \gamma_i, \]
with $\nu_i \in \Omega^\bullet (U_\alpha),$
$\gamma_i \in \tau_{\geq K} \Omega^\bullet (F),$
for $1\leq i\leq k,$ $\deg \nu_i + \deg \gamma_i = r-1,$
and $\eta_0$ has the local form
\[ \eta_0|_{p^{-1} U_\alpha} = \phi^\ast_\alpha \sum_{j=1}^l
 \pi^\ast_1 \eta_j \wedge \pi^\ast_2 \overline{\gamma}_j, \]
with $\eta_j \in \Omega^\bullet (U_\alpha),$
$\overline{\gamma}_j \in \tau_{\geq K^\ast} \Omega^\bullet (F),$
$\deg \eta_j + \deg \overline{\gamma}_j = n-r,$ for 
$1\leq j \leq l$. We have
\[ (\overline{\rho}_\alpha \nu_0)|_{p^{-1} U_\alpha} =
\phi^\ast_\alpha \sum_i \pi^\ast_1 (\rho_\alpha \nu_i)
 \wedge \pi^\ast_2 \gamma_i, \]
where $\rho_\alpha \nu_i \in \Omega^\bullet_c (U_\alpha)$
has compact support in $U_\alpha$. Thus
\begin{eqnarray*}
\int_{p^{-1} U_\alpha} \overline{\rho}_\alpha \nu_0 \wedge \eta_0
& = & \int_{p^{-1} U_\alpha} \phi^\ast_\alpha \sum_{i,j}
 \pi^\ast_1 (\rho_\alpha \nu_i)\wedge \pi^\ast_2 \gamma_i
 \wedge \pi^\ast_1 \eta_j \wedge \pi^\ast_2 \overline{\gamma}_j \\
& = & \sum_{i,j} (\pm) \int_{U_\alpha \times F}
 \pi^\ast_1 (\rho_\alpha \nu_i \wedge \eta_j)\wedge
 \pi^\ast_2 (\gamma_i \wedge \overline{\gamma}_j) \\
& = & \sum_{i,j} (\pm) \int_{U_\alpha}
 \rho_\alpha \nu_i \wedge \eta_j \cdot
 \int_F \gamma_i \wedge \overline{\gamma}_j.
\end{eqnarray*}
We claim that
$\int_F \gamma_i \wedge \overline{\gamma}_j =0,$
which will finish the proof. Let $D$ denote the degree of $\gamma_i$;
we may assume that $\deg \overline{\gamma}_j = m-D$
($m=\dim F$). If $D<K,$ then $\gamma_i =0,$ so the claim is
verified for this case. Suppose that $D\geq K.$ Since
$K=m - \bar{p}(m+1),$ $K^\ast =m - \bar{q}(m+1),$ and
$\bar{p}(m+1) + \bar{q}(m+1) = m-1,$
the inequality $D\geq K$ implies that $m-D < K^\ast.$ Hence
$\overline{\gamma}_j =0$ and the claim is correct in the case
$D\geq K$ as well.
\end{proof}
The next lemma would immediately follow from Stokes' theorem
if we knew that $\nu \wedge \eta$ has compact support in $N$.
\begin{lemma} \label{lem.10.4}
If $\nu$ is a form in $\oip^{r-1} (N)$ and $\eta$
is a form in $\oiq^{n-r}(N),$ then
\[ \int_N d(\nu \wedge \eta)=0. \]
\end{lemma}
\begin{proof}
Set
$E_{>0} = (0,+1)\times \partial M \subset N,$
$N_{\leq 0} = N - E_{>0}.$
The compact manifold $N_{\leq 0}$ has boundary $0 \times \partial M.$
There is a form $\nu_0 \in (\ft_{\geq K} \Omega^\bullet_\ms (B))^{r-1}$
and a form $\eta_0 \in (\ft_{\geq K^\ast} \Omega^\bullet_\ms (B))^{n-r}$
such that $j^\ast \nu = \pi^\ast \nu_0,$ $j^\ast \eta = \pi^\ast \eta_0$.
Splitting the integral into integration over $N_{\leq 0}$ and $E_{>0},$ and
using Stokes' theorem for $N_{\leq 0}$ followed by an application of
Lemma \ref{lem.10.3}, we obtain
\begin{eqnarray*}
\int_N d(\nu \wedge \eta)
& = & \int_{N_{\leq 0}} d(\nu \wedge \eta) + \int_{E_{>0}} d(\nu \wedge
 \eta) \\
& = & \int_{0\times \partial M} \sigma^\ast_0 j^\ast (\nu \wedge \eta) + 
  \int_{E_{>0}} d \pi^\ast (\nu_0 \wedge \eta_0) \\
& = & \int_{\partial M} \nu_0 \wedge \eta_0 + 
  \int_{E_{>0}} d \pi^\ast (\nu_0 \wedge \eta_0) \\
& = & \int_{E_{>0}} \pi^\ast d(\nu_0 \wedge \eta_0), 
\end{eqnarray*}
$\sigma_0: \partial M = 0\times \partial M \hookrightarrow E,$
$\pi \sigma_0 = \id.$ Now $d(\nu_0 \wedge \eta_0) \in \Omega^n (\partial M)$
is an $n$-form on the $(n-1)$-dimensional manifold $\partial M$, thus
$d (\nu_0 \wedge \eta_0)=0$ and
$\int_{E_{>0}} \pi^\ast d(\nu_0 \wedge \eta_0)=0.$
\end{proof}

\section{Poincar\'e Duality for $HI^\bullet_{\bar{p}}$}
\label{sec.pdforhi}

\begin{prop} \label{prop.10.5}
The bilinear form of Corollary \ref{cor.10.2} induces a bilinear form
\[ \begin{array}{rcl}
\int: HI^r_{\bar{p}} (X) \times
 HI^{n-r}_{\bar{q}} (X) & \longrightarrow & \real \\
([\omega], [\eta]) & \mapsto & \int_N \omega \wedge \eta
\end{array} \]
on cohomology.
\end{prop}
\begin{proof}
Let $\omega \in \oip^r (N)$ be a closed form, 
let $\eta \in \oiq^{n-r} (N)$ be a closed form,
let $\omega' \in \oip^{r-1} (N)$ and
$\eta' \in \oiq^{n-r-1} (N)$ be any forms.
Then
$\int_N d(\omega' \wedge \eta)=0$
by Lemma \ref{lem.10.4}. Since $\eta$ is closed,
$d(\omega' \wedge \eta) = (d\omega')\wedge \eta.$ Thus
\[ \int_N (\omega + d\omega')\wedge \eta =
 \int_N \omega \wedge \eta + \int_N (d\omega')\wedge \eta =
 \int_N \omega \wedge \eta. \]
By symmetry,
$\int_N \omega \wedge (\eta + d\eta') = \int_N \omega \wedge \eta$
as well.
\end{proof}

\begin{thm} \label{thm.10.6}
(Generalized Poincar\'e Duality.)
The bilinear form 
\[ 
\int: HI^r_{\bar{p}} (X) \times
 HI^{n-r}_{\bar{q}} (X) \longrightarrow \real 
\]
of Proposition \ref{prop.10.5} is nondegenerate.
\end{thm}
\begin{proof}
By Proposition \ref{prop.11star}, the inclusion $\Omega^\bullet_\bms (N)
\subset \Omega^\bullet (N)$ induces an isomorphism
$H^r_\bms (N) \stackrel{\cong}{\longrightarrow} H^r (N).$
Classical Poincar\'e duality asserts that
\[ 
H^r (N)  \longrightarrow  H^{n-r}_c (N)^\dagger,~ 
\mbox{[}\omega\mbox{]}  \mapsto  \int_N \omega \wedge -
\]
is an isomorphism. By Proposition \ref{prop.rel}, the inclusion
$\Omega^\bullet_{\rel} (N) \subset \Omega^\bullet_c (N)$
induces an isomorphism
$H^{n-r}_c (N)^\dagger \stackrel{\cong}{\longrightarrow}
 H^{n-r}_{\rel} (N)^\dagger.$
Composing these three isomorphisms, we obtain an isomorphism
\begin{equation} \label{equ.10.6}
H^r_\bms (N)  \stackrel{\cong}{\longrightarrow} 
 H^{n-r}_{\rel} (N)^\dagger,~
\mbox{[}\omega \mbox{]}  \mapsto  \int_N \omega \wedge -.
\end{equation}
The nondegenerate form of Proposition 
\ref{prop.globpdtruncfhforms} can be rewritten as an isomorphism
\begin{equation} \label{equ.10.7}
H^r (\ft_{<K} \Omega^\bullet_\ms (B)) \stackrel{\cong}{\longrightarrow}
 H^{n-r-1} (\ft_{\geq K^\ast} \Omega^\bullet_\ms (B))^\dagger,
\end{equation}
while the bilinear form of Proposition \ref{prop.10.5} can be
rewritten as a map
\begin{equation} \label{equ.10.8}
H^r (\oip^\bullet (N)) \longrightarrow
 H^{n-r} (\oiq^\bullet (N))^\dagger. 
\end{equation}
The distinguished triangle (\ref{equ.dt2}) induces a long exact
cohomology sequence
\small
\[ \cdots \rightarrow 
 H^{r-1} (\ft_{<K} \Omega^\bullet_\ms (B)) \rightarrow
 H^r (\oip^\bullet (N)) \rightarrow
 H^r (\Omega^\bullet_\bms (N)) \rightarrow
 H^r (\ft_{<K} \Omega^\bullet_\ms (B)) \rightarrow \cdots. \]
\normalsize
The distinguished triangle (\ref{equ.dt3}) induces a long exact
cohomology sequence
\[ \cdots \rightarrow
H^{n-r} (\ft_{\geq K^\ast} \Omega^\bullet_\ms (B))^\dagger
 \stackrel{(J^\ast_{\bar{p}})^\dagger}{\longrightarrow}
 H^{n-r} (\oiq^\bullet (N))^\dagger
 \stackrel{\operatorname{incl}^{\ast \dagger}}{\longrightarrow}
 H^{n-r} (\Omega^\bullet_{\rel} (N))^\dagger \longrightarrow \]
\[ H^{n-r-1} (\ft_{\geq K^\ast} \Omega^\bullet_\ms (B))^\dagger
 \longrightarrow \cdots. \]
Using the maps (\ref{equ.10.6}), (\ref{equ.10.7}) and 
(\ref{equ.10.8}), we map the former sequence to the latter:
\begin{equation} \label{equ.threestars}
\xymatrix@R=16pt{
H^{r-1} (\ft_{<K} \Omega^\bullet_\ms (B)) \ar[r]^{\cong} \ar[d] & 
   H^{n-r} (\ft_{\geq K^\ast} \Omega^\bullet_\ms (B))^\dagger 
\ar[d]^{(J^\ast_{\bar{p}})^\dagger} \\
H^r (\oip^\bullet (N)) \ar[r] \ar[d] & 
   H^{n-r} (\oiq^\bullet (N))^\dagger 
\ar[d]^{\operatorname{incl}^{\ast \dagger}} \\
H^r (\Omega^\bullet_\bms (N)) \ar[r]^{\cong} \ar[d] & 
H^{n-r} (\Omega^\bullet_{\rel} (N))^\dagger \ar[d] \\
H^r (\ft_{<K} \Omega^\bullet_\ms (B)) \ar[r]^{\cong} & 
H^{n-r-1} (\ft_{\geq K^\ast} \Omega^\bullet_\ms (B))^\dagger \\ 
}
\end{equation}
Let us denote the top square, middle square and bottom
square of this diagram by (TS), (MS), (BS), respectively.
We shall verify that all three squares commute up to sign.
Let us start with (TS). We begin by describing the map
\[ \delta: H^{r-1} (\ft_{<K} \Omega^\bullet_\ms (B))
 \longrightarrow H^r (\oip^\bullet (N)). \]
Let $\iota: \oip^\bullet (N) \hookrightarrow
 \Omega^\bullet_\bms (N)$ denote the subcomplex inclusion and
$C^\bullet (\iota)$ the algebraic mapping cone of $\iota$, that is,
$C^r (\iota) = \oip^{r+1} (N) \oplus
  \Omega^r_\bms (N)$
and $d: C^r (\iota) \to C^{r+1} (\iota)$ is given by
$d(\tau, \sigma) = (-d\tau, \tau + d\sigma).$
Let
\[ \begin{array}{rcl}
P: C^\bullet (\iota) & \longrightarrow & \oip^{\bullet +1} (N) \\
P(\tau,\sigma) & = & \tau
\end{array} \]
be the standard projection and
\[ f: C^\bullet (\iota) \longrightarrow \frac{\Omega^\bullet_\bms (N)}
 {\oip^\bullet (N)} \]
be the map given by
$f(\tau, \sigma) = q(\sigma),$
where
\[ q: \Omega^\bullet_\bms (N) \longrightarrow
\frac{\Omega^\bullet_\bms (N)}
 {\oip^\bullet (N)} \]
is the canonical quotient map. The map $f$ is a quasi-isomorphism.
Recall that
\[ \overline{J}^\ast: \frac{\Omega^\bullet_\bms (N)}
 {\oip^\bullet (N)}  
 \stackrel{\cong}{\longrightarrow} 
 \frac{\Omega^\bullet_\ms (B)}{\ft_{\geq K} \Omega^\bullet_\ms (B)} \]
is an isomorphism given by restriction of a form from $N$ to
$\{ 0 \} \times \partial M = \partial M.$ The quasi-isomorphism
\[ \gamma_B: \ft_{<K} \Omega^\bullet_\ms (B) \longrightarrow
\frac{\Omega^\bullet_\ms (B)}{\ft_{\geq K} \Omega^\bullet_\ms (B)} \]
was defined to be the composition
\[ \ft_{<K} \Omega^\bullet_\ms (B)
 \stackrel{\operatorname{incl}}{\hookrightarrow}
 \Omega^\bullet_\ms (B)
 \stackrel{\operatorname{quot}}{\longrightarrow}
\frac{\Omega^\bullet_\ms (B)}{\ft_{\geq K} \Omega^\bullet_\ms (B)}. \]
Let $\omega \in (\ft_{<K} \Omega^\bullet_\ms (B))^{r-1}$ be a closed
form. Then $d(\gamma_B \omega)=0$ as well.
As $\overline{J}^\ast$ is an isomorphism, there exists a unique element
$w\in \Omega^\bullet_\bms (N) / \oip^\bullet (N)$
such that $\overline{J}^\ast (w) = \gamma_B (\omega)$ and
$\overline{J}^\ast (dw) = d(\overline{J}^\ast w) = d\gamma_B (\omega)=0.$
The injectivity of $\overline{J}^\ast$ implies that
$dw=0 \in \Omega^\bullet_\bms (N) / \oip^\bullet (N).$
Let $\overline{\omega} \in \Omega^{r-1}_\bms (N)$ be a representative for 
$w$ so that $q(\overline{\omega})=w.$ From
$q(d\overline{\omega}) = dq(\overline{\omega}) = dw=0$
we conclude that
$d\overline{\omega} \in \oip^r (N).$
The element
\[ c = (-d\overline{\omega}, \overline{\omega}) \in
  C^{r-1} (\iota) = \oip^r (N) \oplus \Omega^{r-1}_\bms (N) \]
is a cocycle, since
$dc = (d^2 \overline{\omega}, -d\overline{\omega} + d\overline{\omega})
 = (0,0).$
Furthermore,
$f(c) = q(\overline{\omega}) =w$
and hence
$\overline{J}^\ast f(c) = \overline{J}^\ast w = \gamma_B (\omega),$
i.e. $c$ is a lift of $\gamma_B (\omega)$ to a cocycle in the mapping cone.
Since $P(c) = -d\overline{\omega} \in \oip^r (N),$
the element $\delta (\omega)$ can be described as
\[ \delta (\omega) = -d\overline{\omega}. \]
(Note that this does of course not mean that $\delta (\omega)$ represents the
zero class in cohomology, since only $d\overline{\omega}$ is known to
lie in $\oip^\bullet (N)$, but $\overline{\omega}$ itself lies only in
$\Omega^\bullet_\bms (N),$ not necessarily in $\oip^\bullet (N).$)
Since the restriction $\sigma^\ast_0 j^\ast (\overline{\omega})$ of
$\overline{\omega}$ to $\{ 0 \} \times \partial M$ satisfies
\[ [\sigma^\ast_0 j^\ast (\overline{\omega})] =
 [J^\ast \overline{\omega}] = \overline{J}^\ast q(\overline{\omega}) =
 \gamma_B (\omega) \in \frac{\Omega^\bullet_\ms (B)}
 {\ft_{\geq K} \Omega^\bullet_\ms (B)}, \]
we have
\[ \alpha := \sigma^\ast_0 j^\ast (\overline{\omega}) - \omega
 \in \ft_{\geq K} \Omega^\bullet_\ms (B). \]
Thus the restriction of $\overline{\omega}$ to $\{ 0 \} \times \partial M$
equals $\omega$ up to an element in $\ft_{\geq K} \Omega^\bullet_\ms (B).$

As $\overline{\omega} \in \Omega^{r-1}_\bms (N),$ there exists an
$\overline{\omega}_0 \in \Omega^{r-1}_\ms (B) \subset
\Omega^{r-1} (\partial M)$ such that $j^\ast \overline{\omega} =
\pi^\ast \overline{\omega}_0.$ Let $\eta \in \oiq^{n-r} (N)$
be a closed form. There exists an
$\eta_0 \in (\ft_{\geq K^\ast} \Omega^\bullet_\ms (B))^{n-r} \subset
\Omega^{n-r} (\partial M)$ with $j^\ast \eta =
\pi^\ast \eta_0.$ In order to verify the commutativity of (TS), we must
show that
\[ \int_N \delta (\omega)\wedge \eta = \pm \int_{0\times \partial M}
 \omega \wedge J^\ast_{\bar{q}} (\eta). \]
Since $\eta$ is closed,
$(d\overline{\omega})\wedge \eta = d(\overline{\omega} \wedge \eta)$
and
\[ \int_N \delta (\omega)\wedge \eta = - \int_N (d\overline{\omega})\wedge
 \eta = - \int_N d(\overline{\omega} \wedge \eta) 
 = -\int_{N_{\leq 0}} d(\overline{\omega} \wedge \eta) -
  \int_{E_{>0}} d(\overline{\omega} \wedge \eta), \]
where $E_{>0} = (0,1)\times \partial M \subset E,$
$N_{\leq 0} = N - E_{>0},$ $\partial N_{\leq 0} = 0\times \partial M$.
The integral over $E_{>0}$ vanishes, as on $E_{>0}$,
$d(\overline{\omega} \wedge \eta)|_{E_{>0}} = 
 \pi^\ast d (\overline{\omega}_0 \wedge \eta_0)=0,$
$d (\overline{\omega}_0 \wedge \eta_0)$ being an $n$-form on the
$(n-1)$-dimensional manifold $\partial M$. By Stokes' theorem
\[ \int_{N_{\leq 0}} d(\overline{\omega} \wedge \eta) =
 \int_{0 \times \partial M} \overline{\omega}|_{0\times \partial M}
 \wedge \eta|_{0 \times \partial M} =
 \int_{0\times \partial M} \sigma^\ast_0 j^\ast \overline{\omega}
 \wedge \sigma^\ast_0 j^\ast_{\bar{q}} \eta \]
\[ = \int_{\partial M} \omega \wedge J^\ast_{\bar{q}} \eta +
 \int_{\partial M} \alpha \wedge J^\ast_{\bar{q}} \eta. \]
From $\alpha \in (\ft_{\geq K} \Omega^\bullet_\ms (B))^{r-1},$
$J^\ast_{\bar{q}} \eta = \sigma^\ast_0 j^\ast_{\bar{q}} \eta = \sigma^\ast_0
\pi^\ast \eta_0 = \eta_0 \in (\ft_{\geq K^\ast} \Omega^\bullet_\ms
(B))^{n-r}$ and Lemma \ref{lem.10.3} it follows that
$\int_{\partial M} \alpha \wedge J^\ast_{\bar{q}} \eta =0.$
Thus (TS) commutes. \\

Let us move on to (BS). We begin by describing the map
\[ D: H^{n-r-1} (\ft_{\geq K^\ast} \Omega^\bullet_\ms (B))
 \longrightarrow H^{n-r} (\Omega^\bullet_{\rel} (N)). \]
Let $\rho: \Omega^\bullet_{\rel} (N) \hookrightarrow
 \oiq^\bullet (N)$ be the subcomplex inclusion
and $C^\bullet (\rho)$ its algebraic mapping cone. Let
$P: C^\bullet (\rho) \longrightarrow \Omega^{\bullet +1}_{\rel} (N),$
$P(\tau, \sigma) = \tau,$
be the projection and
\[ f: C^\bullet (\rho) \longrightarrow \ft_{\geq K^\ast} \Omega^\bullet_\ms
(B) \]
the quasi-isomorphism given by
$f(\tau, \sigma) = J^\ast_{\bar{q}} (\sigma).$
Recall that the kernel of $J^\ast_{\bar{q}}: \oiq^\bullet (N) 
\twoheadrightarrow \ft_{\geq K^\ast} \Omega^\bullet_\ms (B)$ is
$\im \rho = \Omega^\bullet_{\rel} (N)$.
Let $\eta \in (\ft_{\geq K^\ast} \Omega^\bullet_\ms (B))^{n-r-1}$ be a
closed form. Since $J^\ast_{\bar{q}}$ is surjective, there exists an
$\overline{\eta} \in \oiq^{n-r-1} (N)$ such that
$J^\ast_{\bar{q}} (\overline{\eta})=\eta$. We have
$J^\ast_{\bar{q}} (d\overline{\eta}) = dJ^\ast_{\bar{q}} (\overline{\eta})= d\eta =0.$
Thus $d\overline{\eta} \in \ker J^\ast_{\bar{q}} = \Omega^{n-r}_{\rel} (N).$
The element
\[ c = (-d\overline{\eta}, \overline{\eta})\in \Omega^{n-r}_{\rel} (N)
\oplus \oiq^{n-r-1} (N) = C^{n-r-1}(\rho) \]
is a cocycle, for
$dc = (d^2 \overline{\eta}, -d\overline{\eta} + d\overline{\eta}) = (0,0).$
Moreover,
$f(c) = J^\ast_{\bar{q}} (\overline{\eta})=\eta$ and $P(c) = -d\overline{\eta}.$
We conclude that the image $D(\eta)$ can be described as
\[ D(\eta) = -d\overline{\eta}. \]
We shall next describe the map
\[ Q: H^r (\Omega^\bullet_\bms (N)) \longrightarrow
  H^r (\ft_{<K} \Omega^\bullet_\ms (B)). \]
Let $\omega \in \Omega^r_\bms (N)$ be a closed form. Its image under
\[ \xymatrix{
\Omega^r_\bms (N) \ar@{->>}[r]^{q} &
\frac{\Omega^r_\bms (N)}{\oip^r (N)}
\ar[r]^{\cong}_{\overline{J}^\ast} &
\frac{\Omega^r_\ms (B)}{(\ft_{\geq K} \Omega^\bullet_\ms
 (B))^r}
} \]
is represented by $\omega|_{0 \times \partial M},$
\[ \overline{J}^\ast q(\omega) =
[ \omega|_{0\times \partial M} ] \in
\frac{\Omega^r_\ms (B)}{\ft_{\geq K} \Omega^\bullet_\ms
 (B))^r}. \]
Let $\lsem \overline{J}^\ast q (\omega) \rsem \in H^r (Q^\bullet (B))$ denote
the cohomology class determined by $\overline{J}^\ast q(\omega).$
Since $\gamma_B$ is a quasi-isomorphism, there exists a unique 
class $\lsem \overline{\omega} \rsem \in H^r (\ft_{<K} \Omega^\bullet_\ms
(B)),$ represented by a closed form $\overline{\omega} \in
(\ft_{<K} \Omega^\bullet_\ms (B))^r,$ with $\gamma^\ast_B \lsem
\overline{\omega} \rsem = \lsem \overline{J}^\ast q(\omega) \rsem.$
Consequently, there exists a form $\xi \in \Omega^{r-1}_\ms (B),$
representing an element $[\xi ] \in Q^{r-1} (B)$ with
\[ \gamma_B (\overline{\omega}) - \overline{J}^\ast q(\omega)
 = d[\xi ]. \]
We deduce that
$\alpha = \overline{\omega} - \omega|_{0\times \partial M} -
 d\xi \in \ft_{\geq K} \Omega^\bullet_\ms (B).$
The map $Q$ is described by
\[ Q(\omega) = \overline{\omega}. \]
In order to verify the commutativity of (BS), we must show that
\[ \int_N \omega \wedge D(\eta) = \pm\int_{\partial M}
 Q(\omega) \wedge \eta. \]
Using $d\omega =0,$ we split the left integral as
\[ - \int_N \omega \wedge d\overline{\eta} =
 (-1)^{r+1}\int_{N_{\leq 0}} d(\omega \wedge \overline{\eta}) -
 \int_{E_{>0}} \omega \wedge d\overline{\eta}. \]
The integral over $E_{>0}$ vanishes as
$d\overline{\eta} \in \Omega^{n-r}_{\rel} (N),$ so that
$d\overline{\eta}|_{E_{>0}}=0.$ By Stokes' theorem on
$N_{\leq 0},$ we are reduced to showing
\[ \int_{\partial M} \omega \wedge \overline{\eta} = \pm
 \int_{\partial M} \overline{\omega}\wedge \eta. \]
Rewriting the integrand on the left-hand side as
\[
\omega|_{0\times \partial M} \wedge \overline{\eta}|_{0\times \partial M}
 =  (\overline{\omega} -\alpha -d\xi)\wedge J^\ast_{\bar{q}} (\overline{\eta}) 
=  \overline{\omega} \wedge \eta - \alpha \wedge \eta -
  (d\xi)\wedge \eta,
\]
it remains to show that
\[ \int_{\partial M} \alpha \wedge \eta =0 \text{ and }
 \int_{\partial M} d\xi \wedge \eta =0. \]
The former statement is implied by Lemma \ref{lem.10.3}, as
$\alpha \in (\ft_{\geq K} \Omega^\bullet_\ms (B))^r$ and
$\eta \in (\ft_{\geq K^\ast} \Omega^\bullet_\ms (B))^{n-r-1}$.
The latter follows from Stokes' theorem, observing that
$(d\xi)\wedge \eta = d(\xi \wedge \eta)$ since $\eta$ is closed. \\

Finally (MS) commutes, since the map
$H^r (\oip^\bullet (N)) \longrightarrow
  H^r (\Omega^\bullet_\bms (N))$
is induced by the subcomplex inclusion
$\oip^\bullet (N) \subset \Omega^\bullet_\bms (N),$
and
$H^{n-r} (\Omega^\bullet_{\rel} (N)) \to
  H^{n-r} (\oiq^\bullet (N))$
is induced by the subcomplex inclusion
$\Omega^\bullet_{\rel} (N) \subset
  \oiq^\bullet (N),$
whence the two integrals whose equality has to be demonstrated are
both just
$\int_N \omega \wedge \eta,$
$\omega \in \oip^r (N),$
$\eta \in \Omega^{n-r}_{\rel} (N).$ Since the diagram
(\ref{equ.threestars}) is now known to commute (up to sign), the statement of the theorem
is implied by the $5$-lemma.
\end{proof}

\section{The de Rham Theorem to the Cohomology of Intersection Spaces}

\subsection{Partial Smoothing} \label{ssec.partialsmooth}

Our method to establish the de Rham isomorphism between $HI^\bullet_{\bar{p}}$
and the cohomology of the corresponding intersection space 
requires building an interface between smooth objects
and techniques, such as smooth differential forms and smooth singular
chains in a smooth manifold, and nonsmooth objects, such as the intersection
space, which arises from a homotopy-theoretic construction and is a
CW-complex, not generally a manifold. The interface will be provided by a
certain \emph{partial smoothing} technique that we shall now develop.

For a topological space $X$, let $S_{\bullet} (X)$ denote its singular
chain complex with real coefficients. Homology
$H_{\bullet} (X)$ will mean singular homology, $H_\bullet (S_{\bullet} (X))$.
For a smooth manifold $V$ (which is allowed to have a boundary),
let $\si_{\bullet} (V)$ denote its smooth singular chain complex with
real coefficients, generated by smooth singular simplices
$\Delta^k \rightarrow V$. For a continuous map $g: X\rightarrow V$,
we shall define the \emph{partially smooth chain complex}
$\sps_{\bullet} (g)$. In degree $k$, we set
\[ \sps_k (g) = H_{k-1} (X) \oplus \si_k (V). \]
Let $\iota: \si_{\bullet} (V) \hookrightarrow S_{\bullet} (V)$ be the
inclusion and
$s: S_{\bullet} (V) \longrightarrow \si_{\bullet} (V)$
Lee's smoothing operator, \cite{lee}, pp. 416 -- 424.
The map $s$ is a chain map such that $s \circ \iota$ is the identity
and $\iota \circ s$ is chain homotopic to the identity. Thus $s$ and
$\iota$ induce mutually inverse isomorphisms on homology. If $V$ has
a nonempty boundary $\partial V$ and $J: \partial V \hookrightarrow V$
is the inclusion, then a continuous singular simplex that lies in the boundary
can be smoothed within the boundary. Thus, we can assume that $s$ has been
arranged so that the square
\begin{equation} \label{equ.naturalbndry}
\xymatrix{
S_\bullet (\partial V) \ar[r]^s \ar[d]_{J_\ast} 
& \si_\bullet (\partial V) \ar[d]_{J_\ast} \\
S_\bullet (V) \ar[r]^s & \si_\bullet (V)
} \end{equation}
commutes.
Let $Z_k$ denote the subspace of $k$-cycles in $S_k (X)$ and
$B_k = \partial S_{k+1}(X)$ the subspace of $k$-boundaries.
Choosing direct sum decompositions
\[ S_k (X) = Z_k \oplus B'_k,~ Z_k = B_k \oplus H'_k, \]
we obtain a quasi-isomorphism
$q: H_{\bullet} (X) = H_{\bullet} (S_{\bullet} (X)) \longrightarrow
  S_{\bullet} (X),$
which is given in degree $k$ by the composition
\[ H_k (X) = \frac{Z_k}{B_k} = \frac{B_k \oplus H'_k}{B_k}
 \stackrel{\cong}{\longrightarrow} H'_k \hookrightarrow Z_k \hookrightarrow
  S_k (X).\]
Here, we regard $H_\bullet (X)$ as a chain
complex with zero boundary operators.
By construction, the formula
\begin{equation} \label{equ.bqxbisx}
[ q(x) ] = x
\end{equation}
holds for a homology class $x\in H_k (X),$ that is, $q(x)$ is a
cycle representative for $x$.
Let $x\in H_{k-1}(X)$ be a homology class in $X$ and $v: \Delta^k
\rightarrow V$ be a smooth singular simplex $v\in \si_k (V)$.
We define the boundary operator
$\partial: \sps_k (g) \longrightarrow \sps_{k-1} (g)$
by
\[ \partial (x,v) = (0, \partial v - s g_\ast q(x)), \]
where 
$g_\ast: S_{k-1}(X)\rightarrow S_{k-1}(V)$ is the chain map induced by $g$.
The equation $\partial^2 (x,v)=0$ holds.
The algebraic mapping cone $C_\bullet (g_\ast)$ of $g_\ast$ is given by
\[ C_k (g_\ast) = S_{k-1} (X) \oplus S_k (V),~
 \partial (x,v) = (-\partial x, \partial v - g_\ast (x)). \]
The homology $H_\bullet (g)$ of the map $g$ is
$H_{\bullet} (g) = H_{\bullet} (C_{\bullet} (g_\ast)).$
We wish to show that the partially smooth chain complex
$\sps_{\bullet} (g)$ computes $H_{\bullet} (g)$. To do this, we construct
an intermediate complex $U_{\bullet} (g)$, which underlies both
complexes,
\[ \xymatrix@C=7pt@R=15pt{
C_{\bullet}(g_\ast) \ar[rd] & & \sps_{\bullet} (g) \ar[ld] \\
& U_{\bullet} (g) &
} \]
such that the two maps are quasi-isomorphisms.
Set
\[ U_k (g) = S_{k-1} (X) \oplus \si_k (V),~
 \partial (x,v) = (-\partial x, \partial v - sg_\ast (x)). \]
The property $\partial^2 (x,v)=0$ is readily verified; thus
$U_{\bullet} (g)$ is a chain complex. The map
$\id \oplus s: C_{\bullet} (g_\ast) \longrightarrow U_{\bullet}(g)$
is a chain map.
\begin{lemma} \label{lem.muqis}
The map $\id \oplus s$ is a quasi-isomorphism.
\end{lemma}
\begin{proof}
The inclusions
\[ 
\si_k (V)  \longrightarrow  S_{k-1} (X) \oplus \si_k (V),~
v  \mapsto  (0,v),
\]
form an injective chain map $\si_{\bullet}(V) \rightarrow U_{\bullet}(g).$
The projections
\[
S_{k-1}(X) \oplus \si_k (V)  \longrightarrow S_{k-1}(X),~ 
(x,v)  \mapsto  x,
\]
form a surjective chain map $U_\bullet (g) \rightarrow S_{\bullet -1} (X).$
(Recall that the shifted complex $S_{\bullet -1}(X)$ has boundary operator $-\partial$.)
We obtain an exact sequence
\[ 0 \rightarrow \si_\bullet (V) \longrightarrow U_\bullet (g)
 \longrightarrow S_{\bullet -1} (X) \rightarrow 0. \]
Similarly, we have the standard exact sequence
\[ 0 \rightarrow S_\bullet (V) \longrightarrow C_\bullet (g_\ast)
 \longrightarrow S_{\bullet -1} (X) \rightarrow 0. \]
The morphism of exact sequences
\[ \xymatrix@R=16pt{
0 \ar[r] & \si_\bullet (V) \ar[r] &
U_\bullet (g) \ar[r] & S_{\bullet -1} (X) \ar[r] & 0 \\
0 \ar[r] & S_\bullet (V) \ar[r] \ar[u]_s &
C_\bullet (g_\ast) \ar[r] \ar[u]_{\id \oplus s}
& S_{\bullet -1} (X) \ar[r] \ar@{=}[u] & 0
} \]
induces a commutative diagram on homology with exact rows:
\[ \xymatrix@R=16pt{
H_\bullet (X) \ar[r] & H_\bullet (\si_\bullet (V)) \ar[r] &
H_\bullet (U_\bullet (g)) \ar[r] & H_{\bullet -1} (X) \ar[r] & 
  H_{\bullet -1} (\si_\bullet (V)) \\
H_\bullet (X) \ar[r] \ar@{=}[u] & 
 H_\bullet (V) \ar[r] \ar[u]_{s_\ast}^{\cong} &
H_\bullet (g) \ar[r] \ar[u]_{(\id \oplus s)_\ast}
& H_{\bullet -1} (X) \ar[r] \ar@{=}[u] & H_{\bullet -1}(V)
 \ar[u]_{s_\ast}^{\cong}
} \]
The lemma follows from the $5$-lemma.
\end{proof}

The map
$q \oplus \id: \sps_{\bullet} (g) \longrightarrow U_{\bullet}(g)$
is a chain map, in fact:
\begin{lemma} \label{lem.spsuqis}
The map $q \oplus \id$ is a quasi-isomorphism.
\end{lemma}
\begin{proof}
The inclusions
\[ 
\si_k (V) \longrightarrow H_{k-1} (X) \oplus \si_k (V),~ 
v \mapsto  (0,v),
\]
form an injective chain map $\si_{\bullet}(V) \rightarrow \sps_{\bullet}(g).$
The projections
\[
H_{k-1}(X) \oplus \si_k (V) \longrightarrow H_{k-1}(X),~
(x,v) \mapsto x,
\]
form a surjective chain map $\sps_\bullet (g) \rightarrow H_{\bullet -1} (X).$
We obtain an exact sequence
\[ 0 \rightarrow \si_\bullet (V) \longrightarrow \sps_\bullet (g)
 \longrightarrow H_{\bullet -1} (X) \rightarrow 0. \]
Recall that we had constructed an exact sequence
\[ 0 \rightarrow \si_\bullet (V) \longrightarrow U_\bullet (g)
 \longrightarrow S_{\bullet -1} (X) \rightarrow 0 \]
in the proof of Lemma \ref{lem.muqis}.
The morphism of exact sequences
\[ \xymatrix@R=16pt{
0 \ar[r] & \si_\bullet (V) \ar[r] &
U_\bullet (g) \ar[r] & S_{\bullet -1} (X) \ar[r] & 0 \\
0 \ar[r] & \si_\bullet (V) \ar[r] \ar@{=}[u] &
\sps_\bullet (g) \ar[r] \ar[u]_{q \oplus \id}
& H_{\bullet -1} (X) \ar[r] \ar[u]^q & 0
} \]
induces a commutative diagram on homology with exact rows:
\[ \xymatrix@R=16pt{
H_\bullet (X) \ar[r] & H_\bullet (\si_\bullet (V)) \ar[r] &
H_\bullet (U_\bullet (g)) \ar[r] & H_{\bullet -1} (X) \ar[r] & 
  H_{\bullet -1} (\si_\bullet (V)) \\
H_\bullet (X) \ar[r] \ar@{=}[u] & 
 H_\bullet (\si_\bullet (V)) \ar[r] \ar@{=}[u] &
H_\bullet (\sps_\bullet (g)) \ar[r] \ar[u]_{(q \oplus \id)_\ast}
& H_{\bullet -1} (X) \ar[r] \ar@{=}[u] & H_{\bullet -1}(\si_\bullet (V)),
 \ar@{=}[u]
} \]
using equation (\ref{equ.bqxbisx}), which implies that $q_\ast = \id$ on
homology.
The lemma follows from the $5$-lemma.
\end{proof}
Lemma \ref{lem.muqis} and Lemma \ref{lem.spsuqis} imply:
\begin{prop}(Partial Smoothing.) \label{prop.partialsmooth}
The maps $\id \oplus s$ and $q \oplus \id$ induce an isomorphism
\[ H_\bullet (\sps_\bullet (g)) \cong H_\bullet (g). \]
\end{prop}
This concludes the construction of the partially smooth model to
compute the homology of the map $g$. \\

\subsection{Background on Intersection Spaces}
\label{ssec.backinterspace}

We provide a quick review
of the construction of intersection spaces. For more details,
we ask the reader to consult \cite{banagl-intersectionspaces}. Let $k$ be an integer and
let $C_\bullet (K)$ denote the integral cellular chain complex of a CW-complex $K$.
\begin{defn}
The category $\CWkcb$ of
\emph{$k$-boundary-split CW-complexes} consists of the following
objects and morphisms: Objects are pairs $(K,Y)$, where
$K$ is a simply connected CW-complex and $Y \subset C_k (K)$ is
a subgroup that arises
as the image $Y = s(\im \partial)$ of some splitting
$s: \im \partial \rightarrow C_k (K)$ of the boundary map
$\partial: C_k (K) \rightarrow \im \partial (\subset C_{k-1} (K))$.
(Given $K$, such a splitting always exists, 
since $\im \partial$ is free abelian.) A morphism
$(K,Y_K) \rightarrow (L,Y_L)$ is a cellular map
$f: K \rightarrow L$ such that $f_\ast (Y_K)\subset Y_L$.
\end{defn}
Let $\HoCW_{k-1}$ denote the category whose objects are
CW-complexes and whose morphisms are rel $(k-1)$-skeleton homotopy
classes of cellular maps. Let
\[ t_{<\infty}: \CWkcb \longrightarrow
   \HoCW_{k-1} \]
be the natural projection functor, that is,
$t_{<\infty} (K,Y_K) = K$ for an object $(K, Y_K)$ in 
$\CWkcb$, and $t_{<\infty} (f) = [f]$ for a morphism
$f: (K,Y_K) \rightarrow (L,Y_L)$ in $\CWkcb$.
The following theorem is proved in \cite{banagl-intersectionspaces}.
\begin{thm} \label{thm.generaltruncation}
Let $k\geq 3$ be an integer.
There is a covariant assignment 
$t_{<k}: \CWkcb 
 \longrightarrow \HoCW_{k-1}$ of objects and morphisms
together with a natural transformation
$\emb_k: t_{<k} \rightarrow t_{<\infty}$ such that for an object
$(K, Y)$ of $\CWkcb,$ one has
$H_r (t_{<k} (K,Y);\intg)=0$ for $r\geq k,$ and
\[ \emb_k (K,Y)_\ast: H_r (t_{<k} (K,Y);\intg) \stackrel{\cong}{\longrightarrow}
  H_r (K;\intg) \]
is an isomorphism for $r<k.$ 
\end{thm}
This means in particular that given a morphism $f$, one has squares
\[ \xymatrix{
t_{<k}(K,Y_K)  \ \ \ar[r]^{\emb_k (K,Y_K)} \ar[d]_{t_{<k}(f)}  & 
 \  \ t_{<\infty} (K,Y_K) \ar[d]^{t_{<\infty}(f)} \\
t_{<k}(L,Y_L) \ \ \ar[r]^{\emb_k (L,Y_L)} & \  \
t_{<\infty} (L,Y_L)
} \]
that commute in $\HoCW_{k-1}$.
If $k\leq 2$ (and the CW-complexes are simply connected), then it is of course
a trivial matter to construct such truncations. \\

Let $X$ be an $n$-dimensional pseudomanifold with one isolated singularity.
For a given perversity $\bar{p},$ set $c = n-1-\bar{p}(n)$.
As usual, $M$ denotes the complement of an open cone neighborhood of the
singularity and $N$ continues to denote the interior of $M$. The notation
$E,j,\pi$ is as in Section \ref{sec.abcomplexes}.
To be able to apply
the general spatial homology truncation Theorem \ref{thm.generaltruncation}, we require
the link $L = \partial M$ to be simply connected. This assumption
is not always necessary, as in many non-simply connected situations, ad hoc truncation constructions can
be used.
If $c\geq 3,$ we can and do fix a
completion $(L,Y)$ of $L$ so that $(L,Y)$ is an object in
$\CWccb$. If $c\leq 2,$ no group $Y$ has to be chosen.
Applying the truncation $t_{<c}: \CWccb \rightarrow \HoCW_{c-1}$, we obtain a CW-complex
$t_{<c} (L,Y) \in Ob \HoCW_{c-1}$.
The natural transformation $\emb_c: t_{<c} \rightarrow t_{<\infty}$
of Theorem \ref{thm.generaltruncation} gives a homotopy class $\emb_c (L, Y)$ 
represented by a map
$f: t_{<c}(L, Y) \to L$
such that for $r<c,$ 
$f_{\ast}: H_r (t_{<c}(L, Y)) \cong H_r (L),$
while $H_r (t_{<c}(L, Y))=0$ for $r\geq c$.
The intersection space $\ip X$ is defined to be
\[ \ip X = \cone (g), \]
where $g$ is the composition
\[ \xymatrix{
t_{<c} (L,Y) \ar[r]^f \ar[dr]_g & L \ar@{^{(}->}[d]^J \\
& M.
} \]
Thus, to form the intersection space, we attach the cone on a suitable spatial homology
truncation of the link to the exterior of the singularity along the
boundary of the exterior. Let us briefly write $t_{<c} L$ for $t_{<c} (L,Y)$.
More generally, $\ip X$ has at present been constructed, and
Poincar\'e duality established, 
for the following classes of $X$, where all links are generally assumed
to be simply connected: \\

\noindent $\bullet$ $X$ has stratification depth $1$ and every connected component of
the singular set $\Sigma$ has trivializable link bundle (\cite{banagl-intersectionspaces}).
This includes all $X$ with only isolated singularities (and simply connected links). \\

\noindent $\bullet$ $X$ has depth $1$ and $\Sigma$ is a simply connected sphere,
whose link either has no odd-degree homology or has a cellular chain complex all of
whose boundary operators vanish (\cite{gaisenphdthesis}, the link bundle may be
twisted here), \\

\noindent $\bullet$ $X$ has depth $2$ with one-dimensional $\Sigma$ such that the
links of the components of the pure one-dimensional stratum satisfy a condition similar
to Weinberger's \emph{antisimplicity} condition \cite{weinbhigherrho}, which itself is an
algebraic version of a somewhat stronger geometric condition due to Hausmann, requiring a
manifold to have a handlebody without middle-dimensional handles.

\subsection{$\Omega I^\bullet_{\bar p}$ in the Isolated Singularity Case}

In the isolated singularity case,
\[ \Omega^k_{\bms} (N) = \{ \omega \in \Omega^k (N) ~|~
   j^\ast \omega = \pi^\ast \eta ,~ \text{ some } 
  \eta \in \Omega^k (\partial M) \} \]
and
\[ \oip^k (N) = \{ \omega \in \Omega^k (N) ~|~
   j^\ast \omega = \pi^\ast \eta, \text{ some } \eta \in \tau_{\geq c} \Omega^k (\partial M) \}. \]
Let $\sigma_0: \partial M \hookrightarrow E = (-1,+1)\times \partial M$
be given by $\sigma_0 (x) = (0,x) \in E$. The identity
$\pi \sigma_0 = \id_{\partial M}$
holds. We recall:
\begin{lemma} \label{lem.ombheisoharmbm}
The maps
\[ \xymatrix{\Omega^\bullet_{\bms}(E) & \Omega^\bullet (\partial M) 
\ar@<1ex>[l]^{\pi^\ast}  \ar@<1ex>[l];[]^{\sigma^\ast_0}
} \]
are mutually inverse isomorphisms of cochain complexes.
\end{lemma}
In Section \ref{sec.truncoverpoint}, an orthogonal projection 
$\proj: \Omega^\bullet (\partial M)\to \tau_{<c} \Omega^\bullet (\partial M)$
was defined. Composing, we obtain an epimorphism
$\proj \circ \sigma^\ast_0: \Omega^\bullet_{\bms} (E) \twoheadrightarrow
 \tau_{<c} \Omega^\bullet (\partial M).$
The inclusion $j:E \hookrightarrow N$ induces a surjective restriction map
$j^\ast: \Omega^\bullet_{\bms} (N)\twoheadrightarrow \Omega^\bullet_{\bms}(E).$
\begin{lemma} \label{lem.kernelprojsigmaj}
The kernel of
\[ \proj \circ \sigma^\ast_0 \circ j^\ast: \Omega^\bullet_{\bms} (N) 
\twoheadrightarrow \tau_{<c} \Omega^\bullet (\partial M) \]
is $\oip^\bullet (N)$.
\end{lemma}
\begin{proof}
Let $\omega \in \Omega^\bullet_{\bms} (N)$ be a form such that
$\proj \circ \sigma^\ast_0 \circ j^\ast (\omega)=0.$ There is an
$\eta \in \Omega^\bullet (\partial M)$ with $j^\ast \omega = \pi^\ast \eta.$
Thus
$0 = \proj \circ \sigma^\ast_0 j^\ast (\omega) = \proj \circ \sigma^\ast_0 \pi^\ast
\eta = \proj (\eta).$
The exact sequence (\ref{equ.tgeqkomftlek}) in Section \ref{sec.truncoverpoint},
\[
0 \to \tau_{\geq c} \Omega^\bullet \partial M \longrightarrow \Omega^\bullet \partial M
\longrightarrow \tau_{<c} \Omega^\bullet \partial M \to 0, 
\]
shows that $\eta \in \tau_{\geq c} \Omega^\bullet (\partial M)$.
Thus $\omega \in \oip (N)$. Conversely, every form in $\oip (N)$ is mapped to zero
by $\proj \circ \sigma^\ast_0 j^\ast.$
\end{proof}
By Lemma \ref{lem.kernelprojsigmaj}, we have an exact sequence
\begin{equation} \label{equ.Bbbhtlcbhe2}
0 \longrightarrow \oip^\bullet (N)
\longrightarrow \Omega^\bullet_{\bms} (N) \longrightarrow
\tau_{<c} \Omega^\bullet (\partial M) \longrightarrow 0. 
\end{equation}
In degrees less than $c$, the surjective map in this sequence is given
by restricting to the slice $0 \times \partial M \subset E \subset N.$

\subsection{The de Rham Theorem}

Let us define a map
\[ \Psi_L: H^{k-1} (\tau_{<c} \Omega^\bullet (L)) \longrightarrow
  H_{k-1} (t_{<c}L)^\dagger. \]
For $k-1\geq c,$ $\Psi_L =0,$ since both
$H^{k-1} (\tau_{<c} \Omega^\bullet (L))$ and
$H_{k-1} (t_{<c}L)$ are zero in this case. Suppose $k-1<c.$ Then
$H^{k-1} (\tau_{<c} \Omega^\bullet (L)) = H^{k-1} (L)$
and we define
\[ \widetilde{\Psi}_L: H^{k-1} (L) \longrightarrow
  H_{k-1} (\si_\bullet (L))^\dagger \]
by
\[ \widetilde{\Psi}_L [\omega][b] = \int_b \omega \]
for a smooth singular cycle $b\in \si_{k-1} (L)$. If $b' \in \si_{k-1} (L)$
is another chain such that $b - b' = \partial B$ for a smooth $k$-chain
$B\in \si_k (L),$ then
\[ \int_b \omega - \int_{b'} \omega = \int_{\partial B} \omega =
 \int_B d\omega = 0 \]
by Stokes' theorem for chains and
using $d\omega =0$. Adding an exact form does not change the integral either
because
$\int_b d\nu = \int_{\partial b} \nu = 0,$
as $b$ is a cycle.
Thus $\widetilde{\Psi}_L$
is well-defined. The smoothing operator $s$ induces on homology an
isomorphism
$s_\ast: H_\bullet (L) \stackrel{\cong}{\longrightarrow}
 H_{\bullet} (\si_\bullet (L)).$
The map $f$ induces an isomorphism 
$f_\ast: H_{k-1} (t_{<c}L) \stackrel{\cong}{\longrightarrow}
 H_{k-1} (L)$
since $k-1<c$. The map $\Psi_L$ is defined to be the composition
\[ \xymatrix{
H^{k-1} (L) \ar[r]^{\widetilde{\Psi}_L} &
 H_{k-1} (\si_\bullet (L))^\dagger \ar[r]^{\cong}_{s^\dagger_\ast} &
 H_{k-1} (L)^\dagger \ar[r]^{\cong}_{f^\dagger_\ast} &
 H_{k-1} (t_{<c} L)^\dagger
} \]
for $k-1< c.$

\begin{lemma} \label{lem.aboot}
The map
\[ \Psi_L: H^{k-1} (\tau_{<c} \Omega^\bullet (L)) \longrightarrow
  H_{k-1} (t_{<c}L)^\dagger \]
is an isomorphism for all $k$.
\end{lemma}
\begin{proof}
For $k-1 \geq c,$ both domain and target of $\Psi_L$ are zero. Thus
$\Psi_L$ is an isomorphism in this range of degrees. For $k-1<c,$
we only have to show that $\widetilde{\Psi}_L$ is an isomorphism.
But $\widetilde{\Psi}_L$ is the classical de Rham isomorphism
\[ H^\bullet (\Omega^\bullet (L)) \stackrel{\cong}{\longrightarrow}
  H_\bullet (\si_\bullet (L))^\dagger \]
given by integration on smooth singular chains (cf. \cite{lee},
Theorem 16.12, page 428).
\end{proof}

Next, we shall define an isomorphism
\[ \Psi_M: H^\bullet (\Omega^\bullet_{\bms} (N)) 
 \stackrel{\cong}{\longrightarrow} H_\bullet (\si_\bullet (M))^\dagger. \]
By Proposition \ref{prop.11star}, the inclusion
$\Omega^\bullet_{\bms} (N) \subset \Omega^\bullet (N)$ induces an
isomorphism
\[ H^\bullet (\Omega^\bullet_{\bms} (N)) \stackrel{\cong}{\longrightarrow}
 H^\bullet (\Omega^\bullet (N)). \]
The classical de Rham isomorphism
\[ \Psi_N: H^\bullet (\Omega^\bullet (N)) \stackrel{\cong}{\longrightarrow}
 H_\bullet (\si_\bullet (N))^\dagger \]
is given by
$\Psi_N [\omega][a] = \int_a \omega.$
Since the open manifold $N$ deformation retracts onto the compact
manifold $N_{\leq 0} = N - (0,1)\times L,$ the inclusion
$i_{\leq 0}: N_{\leq 0} \hookrightarrow N$ is a homotopy equivalence
and induces an isomorphism
$i_{\leq 0 \ast}: H_\bullet (\si_\bullet (N_{\leq 0}))
 \stackrel{\cong}{\longrightarrow} H_\bullet (\si_\bullet (N)).$
Let $\alpha: M \rightarrow N_{\leq 0}$ be a diffeomorphism which 
agrees with the diffeomorphism $\partial M \cong \partial N_{\leq 0}$ given
by the collar, so that the diagram
\begin{equation} \label{equ.d1}
\xymatrix{
M \ar[r]^{\alpha} & N_{\leq 0} \\
\partial M \ar@{^{(}->}[u]^J \ar[r]_{\operatorname{collar}} &
\partial N_{\leq 0}
\ar@{^{(}->}[u]
}
\end{equation}
commutes. It induces an isomorphism
$\alpha_\ast: H_\bullet (\si_\bullet (M))
 \stackrel{\cong}{\longrightarrow}
 H_\bullet (\si_\bullet (N_{\leq 0})).$
The isomorphism $\Psi_M$ is defined by the composition
\begin{eqnarray*}
\xymatrix{
H^\bullet (\Omega^\bullet_{\bms} (N)) \ar[r]^{\cong} &
H^\bullet (\Omega^\bullet (N)) \ar[r]^{\cong}_{\Psi_N} &
H_\bullet (\si_\bullet (N))^\dagger 
\ar[r]^>>>>{\cong}_>>>>{i_{\leq 0 \ast}^\dagger} &
} \\
\xymatrix{
H_\bullet (\si_\bullet (N_{\leq 0}))^\dagger 
  \ar[r]^>>>>>{\cong}_>>>>>{\alpha_\ast^\dagger} &
H_\bullet (\si_\bullet (M))^\dagger.\hspace{1.3cm}
} \end{eqnarray*}

\begin{lemma} \label{lem.bboot}
The diagram
\[ \xymatrix{
H^k (\Omega^\bullet_{\bms} (N)) \ar[r]^{\operatorname{j^\ast}} 
\ar[d]^{\cong}_{\Psi_M} &
H^k (\Omega^\bullet_{\bms} (E)) 
 \ar[r]^{\cong}_{\proj \circ \sigma^\ast_0} &
H^k (\tau_{<c} \Omega^\bullet (L)) \ar[d]_{\cong}^{\Psi_L} \\
H_k (\si_\bullet (M))^\dagger \ar[r]^{\cong}_{s^\dagger_\ast} &
H_k (M)^\dagger \ar[r]^{g^\dagger_\ast} &
H_k (t_{<c} L)^\dagger
} \]
commutes.
\end{lemma}
\begin{proof}
The statement holds trivially for $k\geq c,$ since then
$H_k (t_{<c} L)=0$. Assume that $k<c.$ We must prove that for all (closed) $k$-forms
$\omega \in \Omega^\bullet_{\bms} (N)$ and all classes
$[a]\in H_k (t_{<c} L),$ $a\in S_k (t_{<c}L)$ a $k$-cycle, the equation
\[ \Psi_L (\proj \circ \sigma^\ast_0 \omega|_E)[a] = \Psi_M (\omega)(s g_\ast
 (a)) \]
holds. The following computation verifies this, observing that in degrees $k<c,$
$\proj$ is the identity:
\begin{eqnarray*}
\Psi_L (\sigma^\ast_0 \omega|_E)[a] 
& = & f^\dagger_\ast s^\dagger_\ast \widetilde{\Psi}_L
  (\sigma^\ast_0 \omega|_E)[a] 
 =  \widetilde{\Psi}_L (\sigma^\ast_0 \omega|_E)[s f_\ast (a)] \\
& = & \int_{sf_\ast (a)} \omega|_{ \{ 0 \} \times \partial M = \partial
   N_{\leq 0}} 
 =  \int_{sf_\ast (a)} (i^\ast_{\leq 0} \omega)|_{\partial N_{\leq 0}} \\
& = & \int_{sf_\ast (a)} J^\ast \alpha^\ast (i^\ast_{\leq 0} \omega)
 \hspace{1cm} \text{ by (\ref{equ.d1})} \\
& = & \int_{i_{\leq 0 \ast} \alpha_\ast J_\ast sf_\ast (a)} \omega \\
& = & \int_{i_{\leq 0 \ast} \alpha_\ast sJ_\ast f_\ast (a)} \omega 
 \hspace{1.5cm} \text{ by (\ref{equ.naturalbndry})} \\
& = & \alpha^\dagger_\ast i^\dagger_{\leq 0 \ast} \Psi_N (\omega)
  (sJ_\ast f_\ast (a)) 
=  \Psi_M (\omega)(sg_\ast (a)).
\end{eqnarray*}
\end{proof}
Let us define a map
\[ \Psi_{\bar{p}}: H^k (\oip^\bullet (N)) \longrightarrow H_k (\sps_\bullet
 (g))^\dagger. \]
Given a closed form $\omega \in \oip^k (N)$ and a cycle
$(x,v) \in \sps_k (g) = H_{k-1} (t_{<c} L) \oplus \si_k (M),$ we set
\[ \Psi_{\bar{p}} [\omega][(x,v)] = \int_{i_{\leq 0 \ast} \alpha_\ast (v)} \omega, \]
where
\[ \xymatrix{ \si_k (M) \ar[r]^{\cong}_{\alpha_\ast} &
 \si_k (N_{\leq 0}) \ar[r]^{\simeq}_{i_{\leq 0 \ast}} & \si_k (N) } \]
are the chain maps induced by $\alpha$ and $i_{\leq 0}$.
\begin{prop} \label{prop.psibwelldefd}
The map $\Psi_{\bar{p}}$ is well-defined.
\end{prop} 
\begin{proof}
Let $\omega \in \oip^{k-1} (N)$ be any form and $(x,v)\in \sps_k (g)$ a cycle.
Suppose $k-1<c.$ This implies by definition of $\oip^\bullet (N)$
that $j^\ast \omega =0,$ $j: E\hookrightarrow N.$ Furthermore,
$0 = \partial (x,v) = (0, \partial v - sg_\ast q(x))$
so that
$\partial v = sg_\ast q (x) = J_\ast s f_\ast q(x).$
Hence,
\[ \Psi_{\bar{p}} (d\omega)(x,v) = \int_{i_{\leq 0\ast} \alpha_\ast (v)}
 d\omega = \int_v \alpha^\ast i^\ast_{\leq 0} d\omega 
 = \int_v d(\alpha^\ast i^\ast_{\leq 0} \omega) \]
\[ \hspace{2cm}= \int_{\partial v} \alpha^\ast i^\ast_{\leq 0} \omega =
 \int_{J_\ast sf_\ast q(x)} \alpha^\ast i^\ast_{\leq 0} \omega =
\int_{sf_\ast q(x)} J^\ast \alpha^\ast i^\ast_{\leq 0} \omega \]
\[ \hspace{1cm} = \int_{sf_\ast q(x)} (i^\ast_{\leq 0} 
\omega)|_{ \{ 0 \} \times \partial M} =0, \]
using Stokes' theorem for chains and
$(i^\ast_{\leq 0} \omega)|_{ \{ 0 \} \times \partial M} =
(j^\ast \omega)|_{ \{ 0 \} \times \partial M} =0.$
Suppose that $k-1\geq c.$ Then $x \in H_{k-1} (t_{<c} L)=0$ and
\[ \Psi_{\bar{p}} (d\omega)(x,v) = 
 \int_{sf_\ast q(x)} J^\ast \alpha^\ast i^\ast_{\leq 0} \omega =0. \]

Let $\omega \in \oip^{k-1} (N)$ be a closed form and
$(x,v) \in \sps_k (g)$ any chain. If $k-1 \geq c,$ then
$x \in H_{k-1} (t_{<c} L)=0$ is zero and
\[ \Psi_{\bar{p}} (\omega)(\partial (x,v)) = \Psi_{\bar{p}} (\omega)(0,\partial v) 
 \hspace{2cm} \]
\[ \hspace{2cm} = \int_{i_{\leq 0\ast} \alpha_\ast (\partial v)} \omega =
 \int_{\partial i_{\leq 0\ast} \alpha_\ast (v)} \omega =
\int_{i_{\leq 0\ast} \alpha_\ast (v)} d\omega =0, \]
as $\omega$ is closed. If $k-1<c,$ then $j^\ast \omega =0$ and
\begin{eqnarray*}
\Psi_{\bar{p}} (\omega)(\partial (x,v)) & = &
\Psi_{\bar{p}} (\omega)(0, \partial v - sg_\ast q(x)) 
 =  \int_{i_{\leq 0\ast} \alpha_\ast (\partial v)} \omega -
  \int_{i_{\leq 0\ast} \alpha_\ast sg_\ast q(x)} \omega \\
& = & \int_{i_{\leq 0\ast} \alpha_\ast (v)} d\omega -
  \int_{sf_\ast q(x)} J^\ast \alpha^\ast i^\ast_{\leq 0} \omega 
 =  - \int_{sf_\ast q(x)} (j^\ast \omega)|_{ \{ 0 \} \times \partial M} \\
& = & 0.
\end{eqnarray*}
\end{proof}

The inclusion $\oip^\bullet (N) \subset \Omega^\bullet_{\bms} (N)$
induces a map $HI^\bullet_{\bar{p}} (X)\rightarrow
H^\bullet (\Omega^\bullet_{\bms} (N)).$ The standard inclusions
$\si_k (M) \hookrightarrow H_{k-1}(t_{<c} L)\oplus \si_k (M) =
\sps_k (g),$ $v\mapsto (0,v),$ form a chain map
$\operatorname{inc}: \si_\bullet (M) \hookrightarrow \sps_\bullet (g),$
which induces on homology a map
$\operatorname{inc}_\ast: H_{\bullet} (\si_\bullet (M)) \rightarrow
H_\bullet (\sps_\bullet (g)).$
\begin{lemma} \label{lem.cboot}
The square
\[ \xymatrix{
HI^k_{\bar{p}} (X) \ar[r] \ar[d]_{\Psi_{\bar{p}}} &
  H^k (\Omega^\bullet_{\bms} (N)) \ar[d]_{\cong}^{\Psi_M} \\
H_k (\sps_\bullet (g))^\dagger \ar[r]^{\operatorname{inc}^\dagger_\ast}
& H_k (\si_\bullet (M))^\dagger
} \]
commutes.
\end{lemma}
\begin{proof}
For a closed form $\omega \in \oip^k (N)$ and a cycle 
$v\in \si_k (M),$ we calculate
\begin{eqnarray*}
\operatorname{inc}^\dagger_\ast \Psi_{\bar{p}} [\omega][v] 
& = & \Psi_{\bar{p}} [\omega] [\operatorname{inc} (v)] 
 =  \Psi_{\bar{p}} [\omega] [(0,v)] 
 =  \int_{i_{\leq 0\ast} \alpha_\ast (v)} \omega \\
& = & \Psi_N [\omega][i_{\leq 0\ast} \alpha_\ast (v)] 
 =  \alpha^\dagger_\ast i^\dagger_{\leq 0\ast} \Psi_N [\omega][v] 
 =  \Psi_M [\omega][v]. 
\end{eqnarray*}
\end{proof}

The short exact sequence (\ref{equ.Bbbhtlcbhe2}),
\[
0 \longrightarrow \oip^\bullet (N)
\longrightarrow \Omega^\bullet_{\bms} (N) \longrightarrow
\tau_{<c} \Omega^\bullet (L) \longrightarrow 0, 
\]
induces a long exact sequence on cohomology, which contains the 
connecting homomorphism
$\delta^\ast: H^{k-1} (\tau_{<c} \Omega^\bullet (L)) \to
H^k (\oip^\bullet (N)).$
The standard projections $\operatorname{pro}: \sps_k (g) =
H_{k-1}(t_{<c} L) \oplus \si_k (M) \rightarrow H_{k-1} (t_{<c}L),$
$(x,v)\mapsto x,$ form a chain map $\operatorname{pro}: \sps_\bullet (g)
\rightarrow H_{\bullet -1} (t_{<c}L),$ which induces on homology
$\operatorname{pro}_\ast: H_\bullet (\sps_\bullet (g))
\rightarrow H_{\bullet -1} (t_{<c}L).$
\begin{lemma} \label{lem.dboot}
The square
\[ \xymatrix{
H^{k-1} (\tau_{<c} \Omega^\bullet (L)) \ar[r]^{\delta^\ast}
 \ar[d]_{\Psi_L}^{\cong} & H^k (\oip^\bullet (N)) \ar[d]^{\Psi_{\bar{p}}} \\
H_{k-1} (t_{<c} L)^\dagger \ar[r]^{\operatorname{pro}^\dagger_\ast} &
H_k (\sps_\bullet (g))^\dagger
} \]
commutes.
\end{lemma}
\begin{proof}
If $k-1\geq c,$ then $H^{k-1}(\tau_{<c} \Omega^\bullet (L)) =0$ and the
statement of the lemma is correct. Assume that $k-1<c.$ Let
$\omega \in (\tau_{<c} \Omega^\bullet (L))^{k-1} = \Omega^{k-1} (L)$ 
be a closed form on $L=\partial M$.
We shall first describe $\delta^\ast (\omega).$ The form $\pi^\ast \omega$
can be smoothly extended to a form $\overline{\omega} \in 
\Omega^{k-1}_{\bms} (N).$ Its differential $d\overline{\omega}$ lies
in $\oip^k (N)\subset \Omega^k_{\bms} (N)$, since 
$j^\ast d\overline{\omega} = dj^\ast \overline{\omega} =
  d\pi^\ast \omega = \pi^\ast d\omega =0.$
The connecting homomorphism is then described as
\[ \delta^\ast (\omega) = d\overline{\omega}. \]
Let $(x,v)\in \sps_k (g)$ be a cycle, i.e.
$0 = \partial (x,v) = (0, \partial v - sg_\ast q(x)).$
The required commutativity is verified as follows:
\begin{eqnarray*}
\Psi_{\bar{p}} [\delta^\ast \omega](x,v) 
& = & \Psi_{\bar{p}} [d\overline{\omega}](x,v) \\
& = & \int_{i_{\leq 0\ast} \alpha_\ast (v)} d\overline{\omega} =
   \int_{i_{\leq 0\ast} \alpha_\ast (\partial v)} \overline{\omega} =
   \int_{i_{\leq 0\ast} \alpha_\ast sg_\ast q(x)} \overline{\omega} \\
& = & \int_{sf_\ast q(x)} J^\ast \alpha^\ast i^\ast_{\leq 0} \overline{\omega}
  = \int_{sf_\ast q(x)} \overline{\omega}|_{ \{ 0 \}\times \partial M} 
=  \int_{sf_\ast q(x)} \omega \\
& = & \widetilde{\Psi}_L (\omega)(s_\ast f_\ast [q(x)]) =
    \widetilde{\Psi}_L (\omega)(s_\ast f_\ast x) \hspace{2cm}
      \text{ by (\ref{equ.bqxbisx})} \\
& = & f^\dagger_\ast s^\dagger_\ast \widetilde{\Psi}_L (\omega)(x) =
     \Psi_L (\omega)(x) = \Psi_L (\omega)(\operatorname{pro} (x,v)) \\
& = & \operatorname{pro}^\dagger_\ast \Psi_L (\omega)(x,v).      
\end{eqnarray*}
\end{proof}

\begin{thm} \label{thm.derham}
(De Rham Description of $HI^\bullet_{\bar{p}}$.)
The map $\Psi_{\bar{p}},$ induced by integrating a form in $\oip^\bullet (N)$
over a smooth singular simplex in $N$, defines an isomorphism
\[ HI^\bullet_{\bar{p}} (X) \stackrel{\cong}{\longrightarrow}
 H_\bullet (\sps_\bullet (g))^\dagger \cong \redh_\bullet (\ip X)^\dagger
 \cong \redh^\bullet_s (\ip X). \]
\end{thm}
\begin{proof}
The short exact sequence (\ref{equ.Bbbhtlcbhe2}),
\[
0 \longrightarrow \oip^\bullet (N)
\longrightarrow \Omega^\bullet_{\bms} (N) \longrightarrow
\tau_{<c} \Omega^\bullet (L) \longrightarrow 0, 
\]
induces a long exact cohomology sequence
\[ H^{k-1}(\tau_{<c} \Omega^\bullet (L)) \longrightarrow
  H^k (\oip^\bullet (N)) \longrightarrow
  H^k (\Omega^\bullet_{\bms} (N)) \longrightarrow
  H^k (\tau_{<c} \Omega^\bullet (L)). \]
The short exact sequence
\[ 0 \longrightarrow \si_\bullet (M) 
  \stackrel{\operatorname{inc}}{\longrightarrow} \sps_\bullet (g)
 \stackrel{\operatorname{pro}}{\longrightarrow} H_{\bullet -1}
  (t_{<c}L) \longrightarrow 0 \]
induces a long exact sequence
\[ H_{k-1} (t_{<c}L)^\dagger \stackrel{\operatorname{pro}^\dagger_\ast}{\longrightarrow}
  H_k (\sps_\bullet (g))^\dagger \stackrel{\operatorname{inc}^\dagger_\ast}
 {\longrightarrow} H_k (\si_\bullet (M))^\dagger
 \stackrel{g^\dagger_\ast s^\dagger_\ast}{\longrightarrow}
 H_k (t_{<c} L)^\dagger. \]
By Lemmas \ref{lem.bboot}, \ref{lem.cboot} and \ref{lem.dboot}, the diagram
\[ \xymatrix{
H^{k-1}(\tau_{<c} \Omega^\bullet L) \ar[r] \ar[d]_{\Psi_L}^{\cong} &
  H^k (\oip^\bullet (N)) \ar[r] \ar[d]_{\Psi_{\bar{p}}} &
  H^k (\Omega^\bullet_{\bms} (N)) \ar[r] \ar[d]_{\Psi_M}^{\cong} &
  H^k (\tau_{<c} \Omega^\bullet L) \ar[d]_{\Psi_L}^{\cong} \\
H_{k-1} (t_{<c}L)^\dagger \ar[r] &
  H_k (\sps_\bullet (g))^\dagger \ar[r] &
 H_k (\si_\bullet (M))^\dagger \ar[r] &
 H_k (t_{<c} L)^\dagger
} \]
commutes. The maps $\Psi_L$ are isomorphisms by Lemma \ref{lem.aboot}.
The maps $\Psi_M$ are isomorphisms by construction. By the $5$-lemma,
$\Psi_{\bar{p}}$ is an isomorphism. The identification
$H_\bullet (\sps_\bullet (g))^\dagger \cong \redh_\bullet (\ip X)^\dagger$
follows from Proposition \ref{prop.partialsmooth} (Partial Smoothing).
\end{proof}

\section{The Differential Graded Algebra Structure}
\label{sec.dga}

The theory $HI^\bullet_{\bar{p}}$
possesses a perversity-internal cup product structure, as we shall now show.
\begin{thm} \label{thm.dga}
For every perversity $\bar{p}$,
the DGA structure $(\Omega^\bullet (N),d,\wedge)$ restricts to a DGA structure
$(\oip^\bullet (N),d,\wedge)$. In particular, the wedge product of forms induces a 
cup product
\[ \cup: HI^r_{\bar{p}} (X) \otimes HI^s_{\bar{p}}(X)\longrightarrow 
  HI^{r+s}_{\bar{p}}(X). \] 
\end{thm}
\begin{proof}
Let $\omega, \omega'$ be two forms in $\Omega I^\bullet_{\bar{p}} (N).$ Choose
$\eta, \eta' \in \ft_{\geq K} \Omega^\bullet_\ms (B)$ so that $j^\ast \omega = \pi^\ast \eta$ and
$j^\ast \omega' = \pi^\ast \eta'$. Over $p^{-1}(U_\alpha),$  $\eta$ and $\eta'$ have the forms
\[ \eta|_{p^{-1}U_\alpha} = \phi^\ast_\alpha \sum_i \pi^\ast_1 \eta_i \wedge \pi^\ast_2 \gamma_i,~
\eta'|_{p^{-1}U_\alpha} = \phi^\ast_\alpha \sum_j \pi^\ast_1 \eta'_j \wedge \pi^\ast_2 \gamma'_j, \]
with $\gamma_i, \gamma'_j \in \tau_{\geq K} \Omega^\bullet (F).$
Then the product $\gamma_i \wedge \gamma'_j$ again lies in $\tau_{\geq K} \Omega^\bullet (F)$ by
Proposition \ref{prop.cotruncsubdga}.
(Note that the direction in which we truncate enters
crucially here --- if we had used $\tau_{<K}$, the product would not usually lie 
in the truncated complex.)
The proof is completed by observing $j^\ast (\omega \wedge \omega') = \pi^\ast (\eta \wedge \eta')$
and
\begin{eqnarray*}
(\eta \wedge \eta')|_{p^{-1}U_\alpha} & = & \phi^\ast_\alpha \sum_{i,j} \pi^\ast_1 \eta_i 
\wedge \pi^\ast_2 \gamma_i \wedge \pi^\ast_1 \eta'_j \wedge \pi^\ast_2 \gamma'_j \\
& = & \phi^\ast_\alpha \sum_{i,j} (-1)^{\deg \gamma_i \deg \eta'_j}
\pi^\ast_1 (\eta_i \wedge \eta'_j)\wedge \pi^\ast_2 (\gamma_i \wedge \gamma'_j) 
\end{eqnarray*}
with $\gamma_i \wedge \gamma'_j \in \tau_{\geq K} \Omega^\bullet (F)$.
\end{proof}

\section{Foliated Stratified Spaces}
\label{sec.folstrat}

We shall here give a precise definition of what we mean by a stratified foliation. Since
this paper is mostly concerned with depth-1 spaces, we shall restrict our discussion
of foliations to the depth-1 case as well, though the definition can easily be
recursively extended to arbitrary stratified spaces. We will compare our definition
to the one given by Farrell and Jones in \cite{farrelljonesfolcont1} and to the
conical foliations of \cite{wolaksaralbicabgrpisom}.
The main formal difference is that our definition is purely topological, whereas
the definition of Farrell and Jones requires a system of metrics on the strata
satisfying a number of conditions with respect to Mather-type control data of
the stratification. The main result of this section (Theorem \ref{thm.strfolflatlinkbundle}) 
explains how flat link bundles arise in foliated stratified spaces.
To frame the discussion, it
is advantageous to lay down the definition of a stratified space, as understood in 
this paper. We shall work with spaces that possess Mather-type control data,
see for example \cite{matherstratmap} or \cite{piazzasignpackwitt}.
Again, we limit the definition to depth 1 although it is available in full generality.

\begin{defn} \label{def.twostrataspace}
A \emph{$2$-strata space} is a pair $(X,\Sigma)$ such that \\

\noindent (1) $X$ is a locally compact, Hausdorff, second-countable topological space, 
$\Sigma \subset X$ is a closed subspace and a closed, connected, smooth manifold, $X-\Sigma$
is a smooth manifold dense in $X$; \\

\noindent (2) $\Sigma$ possesses control data $(T,\pi,\rho),$ where \\
(2.1) $T\subset X$ is an open neighborhood of $\Sigma$, \\
(2.2) $\pi: T\to \Sigma$ is a continuous retraction, \\
(2.3) $\rho: T\to [0,2)$ is a continuous radial function such that $\rho^{-1}(0)=\Sigma,$ and \\
(2.4) the restrictions of $\pi$ and $\rho$ to $T-\Sigma$ are smooth; \\

\noindent (3) $\pi: T\to \Sigma$ is a locally trivial fiber bundle with fiber the cone
$cL = (L\times [0,2))/(L\times 0)$ over some closed smooth manifold $L$ (the \emph{link} of
$\Sigma$) and structure group given by homeomorphisms $cL\to cL$ of the form $c(\phi)$, where
$\phi: L\to L$ is a diffeomorphism. These $\phi$ are to vary smoothly with points in charts of
$\Sigma$; \\

\noindent (4) Locally, the radius $\rho$ is the cone-line coordinate: If $U\subset \Sigma$ is an open set
and 
\[ \xymatrix@C=17pt@R=17pt{
U \times cL \ar[rr]^{\psi}_{\cong} \ar[rd]_{\operatorname{proj}_1} & & \pi^{-1}(U) 
\ar[ld]^{\pi|} \\
& U
} \] 
a local trivialization with $\psi$ the identity on $U \times \{ c \}$ (where $c$ is the cone vertex),
then
\begin{equation} \label{equ.rhoconelinecoord}
\xymatrix@R=16pt{
U \times cL \ar[r]^{\psi} \ar[d]_{\operatorname{proj}_2} & \pi^{-1}(U) 
  \ar[d]^{\rho|} \\
 cL \ar[r]^{\tau} & [0,2) 
} 
\end{equation}
commutes, where $\tau (l,t)=t,$ $l\in L,$ $t\in [0,2)$.
\end{defn}

For $E=\rho^{-1}(1),$ the above axioms imply that the restriction $\pi|: E\to \Sigma$ is a
smooth fiber bundle with fiber $L$. We call this bundle the \emph{link bundle} of $\Sigma$.
Note that a space $X$ satisfying (1) is metrizable by Urysohn's metrization theorem.

\begin{defn} \label{def.stratdepthone}
A \emph{stratified space of depth 1} (or \emph{depth-1 space} for short) is a tuple
$(X,\Sigma_1,\ldots,\Sigma_r)$ such that $X$ is a locally compact, Hausdorff, second-countable
topological space and the $\Sigma_i$ are mutually disjoint, closed subspaces of $X$ such that
$(X-\bigcup_{j\not= i} \Sigma_j, \Sigma_i)$ is a $2$-strata space for every $i=1,\ldots,r$.
\end{defn}
(A locally compact, Hausdorff, second-countable space is normal --- thus every $\Sigma_i$
has an open neighborhood $T_i$ in $X$ such that $T_i \cap T_j =\varnothing$ for $i\not= j$.)

Recall that a (smooth) $k$-dimensional foliation $\Fa$ of a manifold $M^m$ without boundary
is a decomposition $\Fa = \{ F_j \}_{j\in J}$ of $M$ into connected immersed smooth 
submanifolds of dimension $k$ (called \emph{leaves}) so that the following local triviality condition
is satisfied: each point in $M$ has an open neighborhood $U\cong \real^m$ such that the partition
of $U$ into the connected components of the $U\cap F_j$, $j\in J,$ corresponds under the 
diffeomorphism $\phi: U\cong \real^m$ to the decomposition of $\real^m = \real^k \times \real^{m-k}$
into the parallel affine subspaces $\real^k \times \pt$. Such a $(U,\phi)$ is called a
\emph{foliation chart} and the connected components of the $U\cap F_j$ are called \emph{plaques}.
The plaques contained in a leaf constitute a basis for the topology of the leaf.
This topology does not, in general, coincide with the topology induced on the leaf by the topology
on $M$. Thus $F_j$ is not generally an embedded submanifold. The foliation $\Fa$ induces a foliation
$\Fa_V$ on any open subset $V\subset M$ by taking $\Fa_V$ to consist of the connected components
of all the $V\cap F_j$. 

\begin{defn} \label{def.coneoffol}
The \emph{cone on a foliation} $(M,\Fa)$ is the pair $(cM,c\Fa)$, where $cM$ is the cone on $M$ with
cone vertex $c$ and $c\Fa$ is the decomposition of $cM$ given by
\[ c\Fa = \{ F\times \{ t \} ~|~ F\in \Fa, t\in (0,2) \} \cup \{ c \}. \]
\end{defn} 
Note that $c\Fa$ is a ``singular foliation'' of $cM$, since it contains leaves of different dimensions.
The collection $c\Fa - \{ c \}$ is a smooth foliation of the manifold $cM - \{ c \} = M \times (0,2)$.

\begin{defn} \label{def.stratfol}
A \emph{stratified foliation} of a $2$-strata space $(X,\Sigma)$ is a pair $(\Xa, \Sa)$
such that \\
(1) $\Xa$ is a smooth foliation of the top stratum $X-\Sigma$, \\
(2) $\Sa$ is a smooth foliation of the singular stratum $\Sigma$, and \\
(3) every point in $\Sigma$ has an open neighborhood $U$ with a local trivialization
$\psi: U\times cL \stackrel{\cong}{\longrightarrow} \pi^{-1}(U)$
as in Definition \ref{def.twostrataspace} (4), such that the leaves of the product foliation
$\Sa_U \times (c\La - \{ c \})$ correspond under $\psi$ to the leaves of $\Xa_{\pi^{-1}(U)-\Sigma}$
for some smooth foliation $\La$ on $L$.
\end{defn}

(Note that the leaves of $\Sa_U \times \{ c \}$ are taken to the leaves of $\Sa_U$ automatically, as
$\psi$ is the identity on $U\times \{ c \}$.)

\begin{defn} \label{def.stratfoldepth1}
A \emph{stratified foliation} of a depth-1 space $(X,\Sigma_1,\ldots, \Sigma_r)$ is a tuple
$(\Xa, \Sa_1,\ldots, \Sa_r)$ such that, with $X_i = X- \bigcup_{j\not= i} \Sigma_j$,
$(\Xa_{X_i}, \Sa_i)$ is a stratified foliation of the $2$-strata space $(X_i, \Sigma_i)$ for every $i$.
\end{defn}

\begin{example}
The following type of foliated $2$-strata space plays a role in the work of Farrell and Jones on
the topological rigidity of negatively curved manifolds, \cite{farrelljonespnas}. 
Let $(Y,\Sigma)$ be a $2$-strata space and let $M$ be a connected manifold whose fundamental group 
$G$ acts on $Y$ preserving the two strata such that $\Sigma$ has a $G$-invariant tube $T$ with
equivariant retraction $\pi: T\to \Sigma$.
Let $\widetilde{M}$ be the universal cover of $M$.
The quotient
\[ X = \widetilde{M} \times_G Y \]
of $\widetilde{M} \times Y$ under the diagonal action of $G$ is a $2$-strata space with top
stratum $\widetilde{M} \times_G (Y-\Sigma)$ and bottom stratum  $\widetilde{M} \times_G \Sigma.$
A stratified foliation $(\Xa, \Sa)$ of $X$ is given by taking
\[ \begin{array}{lcl}
\Xa & = & \{ p(\widetilde{M} \times \{ y \}) ~|~ y\in Y-\Sigma \} \text{ and} \\
\Sa & = & \{ p(\widetilde{M} \times \{ y \}) ~|~ y\in \Sigma \},
\end{array} \]
where $p$ is the covering projection $p: \widetilde{M} \times Y \to X.$
To see this, trivialize locally the flat $Y$-bundle $X\to M$ induced by $\widetilde{M}\times Y \to
\widetilde{M},$ trivialize locally $\pi: T\to \Sigma$ and equip the link $L$ with the $0$-dimensional
foliation $\La$.
\end{example}

\begin{prop} \label{prop.dpidrho}
For a stratified foliation $(\Xa, \Sa)$ of a $2$-strata space $(X,\Sigma)$ with control data
$(T,\pi,\rho)$, the following statements hold: \\
$(i)$ If $v$ is a vector at a point in $T-\Sigma$ which is tangent to a leaf of $\Xa$, then $\pi_\ast (v)$
is tangent to a leaf of $\Sa$. \\
$(ii)$ The radial function $\rho$ is constant along the leaves of $\Xa_{T-\Sigma}$. In particular,
$\rho_\ast (v)=0$ for $v$ tangent to $\Xa_{T-\Sigma}$.
\end{prop}
\begin{proof}
$(i)$ Let $U\subset \Sigma$ be a chart such that $v$ is based at a point of $\pi^{-1}(U)-\Sigma$ and
consider the commutative diagram
\[ \xymatrix@R=18pt@C=5pt{
TU \times T(L\times (0,2)) \ar[rr]^>>>>>>{\psi_\ast}_>>>>>>{\cong} \ar[dr]_{\operatorname{proj}_1} & &
 T(\pi^{-1} (U)-\Sigma) \ar[dl]^{\pi_\ast} \\
& TU. & \\
} \]
Let $F \in \Xa_{\pi^{-1}(U)-\Sigma}$ be the leaf that $v$ is tangent to. Then by Definition
\ref{def.stratfol} (3), there exists a leaf $S\times K \times \{ t \},$ $S\in \Sa_U,$
$K\in \La,$ $t\in (0,2),$ such that $\psi (S\times K \times \{ t \})=F.$ Hence there is a vector
$(u,w)\in TS \oplus TK$ with $\psi_\ast (u,w,0)=v$. Then
\[ \pi_\ast (v) = \pi_\ast (\psi_\ast (u,w,0)) = \operatorname{proj}_1 (u,w,0) =u \]
with $u$ tangent to $S$, which is an open subset of a leaf of $\Sa$. \\

\noindent $(ii)$ It suffices to prove that $\rho$ is locally constant along the leaves of
$\Xa_T$, since leaves are connected. Let $F$ be a leaf in $\Xa_{\pi^{-1}(U)-\Sigma}$ and let
$S\in \Sa_U,$ $K\in \La,$ $t$ be such that $\psi (S\times K \times \{ t \})=F,$ as in $(i)$.
Using the commutative diagram (\ref{equ.rhoconelinecoord}) in
Definition \ref{def.twostrataspace}, we have
\[ \rho (F) = \rho \psi (S\times K \times \{ t \}) = \tau \circ \operatorname{proj}_2
 (S\times K \times \{ t \}) = \tau (K\times \{ t \})= \{ t \}. \]
Hence $\rho$ is constant on $F$.
\end{proof}
It follows from this proposition that our definition of a stratified foliation is compatible with
the definition of Farrell and Jones as given in \cite[Def. 1.4]{farrelljonesfolcont1}.
The latter requires essentially that \\

\noindent (a) for vectors $v$ tangent to $\Xa_{T-\Sigma},$ the ratio of the length
of $\pi_\ast (v)^\perp$ to the length of $v$, where $\pi_\ast (v)^\perp$ is the component of
$\pi_\ast (v)$ perpendicular to the leaves of $\Sa$, becomes as small as we like by taking the
base point of $v$ sufficiently close to $\Sigma$ as measured by $\rho$, and \\

\noindent (b) the same statement for the ratio of the size of
 $\rho_\ast (v)$ to the length of $v$. \\

\noindent Note that this definition requires endowing the strata with a system of
Riemannian metrics. Suppose that a $2$-strata space has a stratified foliation
in the sense of our Definition \ref{def.stratfol}.
As $\pi_\ast (v)^\perp =0$ by Proposition \ref{prop.dpidrho}$(i)$, condition (a) is
satisfied. As $\rho_\ast (v)=0$ by Proposition \ref{prop.dpidrho}$(ii)$, condition (b) is
satisfied as well. 

Furthermore, our stratified foliations are compatible with the ``conical foliations''
of \cite{wolaksaralbicabgrpisom}, which the authors define only for spherical links, that is, 
for $X$ a manifold. They do allow, however, singular foliations on the links, which we
do not. On the other hand, we allow the $0$-dimensional foliation on the link, which they
disable.

Let $(M,\Fa)$ be a foliated manifold and $N\subset M$ an immersed submanifold. One
says that $\Fa$ is \emph{tangent to $N$} if for each leaf $F$ in $\Fa$, either 
$F\cap N = \varnothing$ or $F\subset N$.
\begin{lemma} \label{lem.tangentfol}
If $\Fa$ is tangent to $N$, then
\[ \Ga = \{ F\in \Fa ~|~ F\cap N \not= \varnothing \} \]
is a smooth foliation of $N$.
\end {lemma}

\begin{thm} \label{thm.strfolflatlinkbundle}
Let $(X,\Sigma)$ be a $2$-strata space endowed with a stratified foliation which
is $0$-dimensional on the links. Then the restrictions of the link bundle to the leaves
of the singular stratum are flat bundles.
\end{thm}
\begin{proof}
The total space $E= \rho^{-1}(1)$ of the link bundle $\pi|:E\to \Sigma$ is a
submanifold of $X-\Sigma$ and $\Xa$ is tangent to $E$. Indeed, if $F$ is a leaf
of $\Xa$ such that $F\cap E \not= \varnothing,$ then there is a point
$x\in F$ such that $\rho (x)=1$. By Proposition \ref{prop.dpidrho}$(ii)$,
$\rho$ is constant along $F$. Thus $\rho|_F \equiv 1$ and so $F\subset E$.
By Lemma \ref{lem.tangentfol},
\[ \Ea = \{ F\in \Xa ~|~ F\cap E \not= \varnothing \} \]
is a foliation of $E$. Let $S$ be a leaf in $\Sigma$ and set
$E_S = \pi^{-1}(S)\cap E$. Then $E_S$ is an immersed submanifold of $E$.
We claim that 
\begin{center}
$\Ea$ is tangent to $E_S$. \hspace{1.5cm} $(\ast)$ \\
\end{center}

\noindent In order to see this, let $F\in \Ea$ be a leaf that touches $E_S$,
$F\cap E_S \not= \varnothing$. We have to show that $F\subset E_S$.
Since $F\cap E_S \not= \varnothing,$ there is a point $x_0 \in F$ with
$\pi (x_0)\in S$. We must show that $\pi (x)\in S$ for all $x\in F$.
Since $F$ is connected, we may join $x_0$ and $x$ by a path
$\gamma: [0,1]\to F,$ $\gamma (0)=x_0,$ $\gamma (1)=x$.
The compact space $\pi \gamma [0,1] \subset \Sigma$ can be covered by
finitely many open sets $U_0, \ldots, U_k \subset \Sigma$, each of which comes
with a diffeomorphism $\psi_i: U_i \times L \times \{ 1 \} \to \pi^{-1}(U_i)\cap E$
such that $\pi \psi_i = \operatorname{proj}_1$. By the Lebesgue number lemma,
there is an $N$ such that each $\pi \gamma (I_j),$
$I_j = [j/N, (j+1)/N],$ lies in some $U_i$. Then the claim $(\ast)$ is implied by
the following statement:
\begin{center}
\begin{tabular}{lr}
For all $0\leq j <N$: If $\pi \gamma (j/N)\in S$, then & \\
$\pi \gamma (t)\in S$ for all $t\in I_j.$ & $(\ast \ast)$
\end{tabular}
\end{center}
To prove $(\ast \ast)$, assume that $\pi \gamma (j/N)\in S$ and let $i$ be such that
$\pi \gamma (I_j)\subset U_i$. Let $F_0$ be the unique connected component
of $F\cap \pi^{-1}(U_i)$ that contains $\gamma (j/N)$. Then, as $\gamma (I_j)$
is connected and contained in $F\cap \pi^{-1}(U_i),$ we have
$\gamma (t)\in F_0$ for all $t\in I_j$. By the definition of a stratified foliation,
there is a leaf $S'$ in $\Sa$ and a leaf $K\in \La$ such that
$\psi_i (S'_0 \times K \times \{ 1 \})=F_0,$ where $S'_0$ is a connected component
of $S' \cap U_i$. Since $\pi \gamma (j/N)\in S$ and
\[ \pi \gamma (j/N) = \operatorname{proj}_1 \circ \psi^{-1}_i \circ \gamma
 (j/N) \in \operatorname{proj}_1 \circ \psi^{-1}_i (F_0) =
 \operatorname{proj}_1 (S'_0 \times K \times \{ 1 \}) = S'_0 \subset S', \]
the leaves $S$ and $S'$ have a point in common, which implies that $S'=S$.
In particular, $S'_0 \subset S$. Consequently, as $\gamma (t) \in F_0$ for all
$t\in I_j,$
\[ \pi \gamma (t) = \operatorname{proj}_1 \circ \psi^{-1}_i \circ \gamma (t)
 \in \operatorname{proj}_1 \circ \psi^{-1}_i (F_0) = S'_0 \subset S \]
for all $t\in I_j,$ which establishes statement $(\ast \ast)$, and thus also the
claim $(\ast)$. 
By Lemma \ref{lem.tangentfol}, 
\[
\Ea_S  =  \{ F\in \Ea ~|~ F\cap E_S \not= \varnothing \} 
  =  \{ F\in \Xa ~|~ F\cap E_S \not= \varnothing \}
\]
is a smooth foliation of $E_S$. So far, we have not used the assumption that the
foliations $\La$ on the links are zero-dimensional. We shall now use that
assumption to prove that $(\pi|:E_S \to S, \Ea_S)$ is a transversely foliated
bundle. Let $s = \dim \Sa$.
For every point $x\in S,$ we must find an open neighborhood
$V\subset S,$ $V \cong \real^s,$ and a diffeomorphism
$\varphi: V\times L \to \pi^{-1}(V)\cap E$ such that $\pi \varphi = 
\operatorname{proj}_1$ and $\varphi$ carries the product foliation
$\{ V\times \{ l \} \}_{l\in L}$ to the foliation $(\Ea_S)_{\pi^{-1}(V)\cap E}$.
This implies that $\Ea_S$ is transverse to the fibers of the link bundle and that
the restriction of $\pi$ to each leaf of $\Ea_S$ is a covering map. Let
$U\subset \Sigma$ be an open neighborhood of $x$ such that there is a
diffeomorphism $\psi: U\times L\times \{ 1 \} \to \pi^{-1}(U)\cap E$ with
$\pi \psi = \operatorname{proj}_1$. We may moreover take such a $U$ to be
the domain of a foliation chart $\phi: U \stackrel{\cong}{\longrightarrow}
\real^s \times \real^{\dim \Sigma -s}.$ Let $V$ be the unique plaque of $S$
in $U$ that contains $x$. Under $\phi,$ $V$ is mapped to $\real^s \times \pt$.
Let $\varphi: V\times L \to \pi^{-1}(V)\cap E$ be the restriction of $\psi$
to $V\times L$. A leaf $F_0$ in $(\Ea_S)_{\pi^{-1}(V)\cap E}$ is a connected
component of $F\cap \pi^{-1}(V),$ where $F$ is a leaf of $\Xa$ which maps to $S$
under $\pi$ and to $1$ under $\rho$. Let $F_1$ be the connected component
of $F\cap \pi^{-1}(U)$ which contains $F_0$. By definition of a stratified
foliation, there is a leaf $\{ l \}$ in $\La$, $l\in L,$ and a plaque $V'$ of $S$ in $U$
such that $\psi (V' \times \{ l \} \times \{ 1 \})=F_1.$ We have
$\pi (F_0)\subset V,$ as $F_0 \subset F \cap \pi^{-1}(V).$ Also,
$\pi (F_0)\subset \pi (F_1)\subset V'$ so that $\pi (F_0) \subset V\cap V'.$
But $V\cap V' =\varnothing$ unless $V=V'$. Since $\pi (F_0)$ is not empty,
we have $V=V'$ and thus $\psi (V \times \{ l \} \times \{ 1 \})=F_1.$
In particular, $\pi (F_1) =$ $\pi \psi (V\times \{ l \} \times \{ 1 \})=$
$\operatorname{proj}_1 (V\times \{ l \} \times \{ 1 \})= V$.
Hence $F_1 \subset F\cap \pi^{-1}(V).$ Since $F_1$ is connected,
$F_0 \subset F_1$, and $F_0$ is a connected component of $F\cap \pi^{-1}(V),$
we conclude that $F_1 = F_0$. Thus any leaf $F_0$ in  $(\Ea_S)_{\pi^{-1}(V)\cap E}$
corresponds under $\varphi$ to a leaf of the form $V\times \{ l \}$ for some
$l\in L$. We have shown that $\Ea_S$ is a transverse foliation of the link bundle over $S$.
This transverse foliation defines a flat connection on $\pi|:E_S \to S$, see also
\cite[Theorem 2.1.9]{candelconlon}.
\end{proof}

\bibliographystyle{amsalpha}
\bibliography{../../mybib}

\end{document}